\documentclass[mathpazo]{cicp}

\usepackage[T1]{fontenc}
\usepackage{textcomp}
\usepackage[varqu,varl]{inconsolata} 
\usepackage[cal=boondoxo]{mathalfa} 

\usepackage{graphicx}

\usepackage{amsfonts,amssymb,amsbsy,amsmath,amsthm}
\usepackage{euscript,mathrsfs,dsfont}
\usepackage{xcolor}
\usepackage[citecolor=blue,colorlinks=true]{hyperref}
\usepackage[left=2cm,right=2cm,top=3.5cm,bottom=3.5cm]{geometry}
\usepackage[framemethod=tikz]{mdframed}
\usepackage{multirow}
\usepackage{subcaption}

\usepackage{empheq}
\usepackage{geometry}
\usepackage{tikz}
\usepackage[framemethod=tikz]{mdframed}
\usetikzlibrary{patterns}
\usepackage[normalem]{ulem}
\usepackage{cancel}

\usepackage{braket}

\usepackage{bm}

\usepackage{soul}

\catcode`\@=11 \@addtoreset{equation}{section}

\catcode`\@=12

\allowdisplaybreaks

\newtheorem{Theorem}{Theorem}[section]
\newtheorem{Proposition}[Theorem]{Proposition}
\newtheorem{Lemma}[Theorem]{Lemma}
\newtheorem{Corollary}[Theorem]{Corollary}

\theoremstyle{Definition}
\newtheorem{Definition}[Theorem]{Definition}

\newtheorem{Example}[Theorem]{Example}
\newtheorem{Remark}[Theorem]{Remark}

\newcommand{\bTheorem}[1]{
\begin{Theorem} \label{T#1} }
\newcommand{\eT}{\end{Theorem}}

\newcommand{\bProposition}[1]{
\begin{Proposition} \label{P#1}}
\newcommand{\eP}{\end{Proposition}}

\newcommand{\bLemma}[1]{
\begin{Lemma} \label{L#1} }
\newcommand{\eL}{\end{Lemma}}

\newcommand{\bCorollary}[1]{
\begin{Corollary} \label{C#1} }
\newcommand{\eC}{\end{Corollary}}

\newcommand{\bRemark}[1]{
\begin{Remark} \label{R#1} }
\newcommand{\eR}{\end{Remark}}

\newcommand{\bDefinition}[1]{
\begin{Definition} \label{D#1} }
\newcommand{\eD}{\end{Definition}}

\newcommand{\bFormula}[1]{
	\begin{equation} \label{#1}}
\newcommand{\eF}{\end{equation}}

\newcommand{\vc}[1]{{\bm #1}}

\newcommand{\Div}{{\rm div}_{ x}}

\newcommand{\vr}{\varrho}

\newcommand{\projection}[1]{\Pi_h[ #1 ]}

\newcommand{\abs}[1]{\left| #1\right|}

\newcommand{\dt}{\,{\rm d} t }
\newcommand{\dx}{\,{\rm d} { x}}

\newcommand{\dthe}{\,{\rm d} {\theta}}
\newcommand{\dS}{\,{\rm d} {S}}
\newcommand{\dSt}{\,{\rm d} {\tilde {t}}}

\newcommand{\jump}[1]{[\! [ #1 ] \! ]}
\newcommand{\avs}[1]{\left\{\!\!\left\{ #1\right\}\!\!\right\}}

\newcommand{\dxi}{\,\mathrm{d}\xi}
\newcommand{\dsx}{\mathrm{d}S_x}

\newcommand{\dtau}{\mathrm{d}\tau}

\newcommand{\TS}{\Delta t}

\newcommand{\vF}{\vc{F}}
\newcommand{\aleq}{\lesssim}

\newcommand{\vu}{\vc{u}}
\newcommand{\vm}{\vc{m}}

\newcommand{\vn}{\vc{n}}

\newcommand{\vU}{\vc{U}}
\newcommand{\vcr}{\vc{r}}
\newcommand{\vl}{\vc{l}}

\newcommand{\R}{\mathbb{R}}

\newcommand{\I}{\mathbb{I}}

\newcommand{\br}{ \nonumber \\ }

\def\softd{{\leavevmode\setbox1=\hbox{d}%
          \hbox to 1.05\wd1{d\kern-0.4ex{\char039}\hss}}}
\definecolor{Cgrey}{rgb}{0.85,0.85,0.85}
\definecolor{Cblue}{rgb}{0.50,0.85,0.85}
\definecolor{Cred}{rgb}{1,0,0}
\definecolor{fancy}{rgb}{0.10,0.85,0.10}
\definecolor{forestgreen}{rgb}{0.13, 0.55, 0.13}

\begin{document}


\title{Convergence analysis for a finite volume evolution Galerkin method for multidimensional hyperbolic systems\footnote{
\emph{Dedicated to Professor LI Jiequan}} }


 \author[M\' aria Luk\' a\v cov\' a -- Medvi\softd ov\' a et~al.]{M\' aria Luk\' a\v cov\' a -- Medvi\softd ov\' a\affil{1}\comma\corrauth,
      Zhuyan Tang\affil{1}, and Yuhuan Yuan\affil{2}}
 \address{\affilnum{1}\ Institute of Mathematics,
           Johannes Gutenberg-University Mainz,
          Staudingerweg 9, 55128 Mainz, Germany. \\
           \affilnum{2}\ School of Mathematics,
           Nanjing University of Aeronautics and Astronautics,
           Jiangjun Avenue No. 29, 211106 Nanjing, P. R. China.}
 \emails{{\tt lukacova@uni-mainz.de} (M.~Lukacova), {\tt zhtang@uni-mainz.de} (Z.~Tang),
          {\tt yuhuanyuan@nuaa.edu.cn} (Y.~Yuan) \\
          }

\begin{abstract}
We study the convergence of a finite volume method based on the method of bicharacteristics for multidimensional hyperbolic conservation laws. In particular, we concentrate on the linear wave equation system and nonlinear Euler equations of gas dynamics. We show the stability and the consistency of the numerical approximations. By means of the generalized Lax equivalence principle we prove the convergence of numerical solutions to the strong solution on the lifespan. 
\end{abstract}



\maketitle

\section{Introduction}\label{Introduction}
Hyperbolic conservation laws are fundamental to model conservation principles arising in fluid dynamics, physics or biology. In this paper we consider mutidimensional hyperbolic conservation laws and concentrate on the linear wave equation system and the nonlinear Euler equations of gas dynamics.

These systems are approximated by a genuinely mutidimensional finite volume evolution Galerkin method which is based on the method of bicharacteristics, that has been developed by Luk\' a\v cov\' a, Morton and Warnecke \cite{LMW:2000}. Applying the generalized Lax equivalence principle, see \cite{FLMS}, and the recently developed concept of generalized solutions, the dissipative solutions, we will analyze the convergence of the genuinely multidimensional finite volume evolution Galerkin method. 

In our recent works \cite{FLM,LMY}, we have proved the convergence of the numerical solutions for the Euler equations obtained by the finite volume upwind method and the Godunov method, respectively. In general, we obtain only weak* convergence to a generalized, dissipative solution. If the limit is a weak entropy solution then the convergence is also strong. Moreover, if the Euler equations admit a strong solution then the numerical solutions converge strongly to the strong solution as long as the latter exists.

Our aim is to extend the previous convergence analysis to the genuinely multidimensional finite volume method based on the method of bicharacteristics. We point out that this is the first result available in the literature, where the convergence of the truly mutidimensional scheme is studied for multidimensional hyperbolic systems. In this context we refer to the recent work of Luk\' a\v cov\' a and Yuan \cite{LMY2023}, where the convergence of finite volume generalized Riemann problem method, see J. Li et al. \cite{Li-Zhang-Yang:1998}, \cite{LiDu}, has been analysed for scalar hyperbolic equation.

The hyperbolic conservation law on a bounded domain $\Omega \subset \R^d \, (d=2,3)$ reads
\begin{equation}\label{pde}
 \partial_t \vU + \Div \vF(\vU) = 0, \quad (t,x) \in (0,T) \times \Omega,
\end{equation}
where $\vU \in \R^N$ is the conservative variable and $\vF \in \R^{N\times d}$ is the flux function. 
System \eqref{pde} is accompanied with initial data $\vU(0,\cdot)=\vU_0$ on $\Omega$ and periodic boundary conditions.
Taking the second law of thermodynamics into account we further require that the entropy inequality holds, i.e.
\begin{equation}\label{I6}
\partial_t  \, S(\vU ) + \Div  \,\bm{Q}(\vU) \geq 0,\quad (t,x) \in (0,T) \times \Omega.
\end{equation}
We analyse specifically the linear wave equation system with
\begin{align}\label{wave}
&\vU = (\phi, \vu)^T, \quad \vF = \left( c\vu, \ c\phi \I \right)^T, \quad c>0,\\
& S = -\frac12 \abs{\vU}^2, \quad \vc{Q} = -c\phi\vU,
\end{align}
and the nonlinear Euler equations with
\begin{align}\label{euler1}
& \vU = (\vr, \vm, E)^T, \quad \vF =  \left(\vm,\, \frac{\vm \otimes \vm}{\vr}+ p \I ,\, \frac{\vm (E+p)}{\vr} \right)^T,\\
&S = C_v \vr s , \quad \vc{Q} = \frac{S \vm}{\vr} \quad \mbox{with } C_v =  \frac{1}{\gamma -1} \mbox{ and } s= \ln\left(\frac{p}{\vr^\gamma}\right),\quad \gamma>1.\label{euler2}
\end{align}

In this paper we concentrate on two-dimensional problem, i.e. $d=2$. The generalization to $d=3$ is possible but technical and the details will be presented in a future work.

In \eqref{wave}, $\phi$ denotes the wave pressure, $\vu=(u,v)$ is the velocity vector and $c=$constant is the wave speed. Further, in \eqref{euler1}, $\rho$ denotes the density, $\vm=(m_1,m_2)=(\rho u,\rho v)$ is the momentum, $E$ is the energy, $p=(\gamma-1)(E-\frac{1}{2}\frac{|\vm|^2}{\rho})$ is the pressure and $S$ is the entropy.

 It is well-known that for the multidimensional Euler equations there may exist infinitely many weak entropy solutions, i.e. the solutions satisfying \eqref{pde},\eqref{euler2} in the weak sense, cf. De Lellis and Sz\'ekelyhidi \cite{DeLellis-Szekelyhidi:2010}, Chiodaroli et al. \cite{Chiodaroli-DeLellis-Kreml:2015,Chiodaroli-Kreml-Macha-Schwarzacher:2021}, and Feireisl et al. \cite{Feireisl-Klingenberg-Kreml-Markfelder:2020}. This makes analysis of numerical schemes challenging. We present an elegant way of such analysis via a generalized Lax equivalence principle and dissipative solutions. 

The rest of the paper is organised as follows. Section 2 presents the numerical scheme. Section 3 is dedicated to the discussion of entropy stability, while Section 4 introduces the consistency formulation. In Section 5, we provide the convergence results. Numerical experiments illustrating theoretical results are presented in the last section.

\section{Finite volume evolution Galerkin method}\label{sec_FVEG}
In this section we introduce the finite volume evolution Galerkin method which is originally proposed in Luk{\'a}{\v c}ov{\'a}  et al. \cite{LMW:2000,LSW:2002}.

\subsection{Notations}
Let $\mathcal{T}_h$ be a uniform structured mesh formed by squares with the mesh parameter $h\in(0,1)$ such that $\overline{\Omega} := \bigcup_{K \in \mathcal{T}_h} \overline{K}$. We denote by $\sigma = L|R$ the common face of cells $L$ and $R$, by $\Sigma$ the set of all faces, 
and by $\Sigma(K)$ the set of all faces of a generic cell $K\in \mathcal{T}_h$.

We consider the space of piecewise constant functions
\begin{equation}
\mathcal{Q}_h = \{ v \colon  v|_{K}  = \mbox{constant}, ~ \mbox{for all}~ K \in \mathcal{T}_h \}, 
\end{equation}
and define the projection operator
\begin{equation}
\Pi_h {\colon}  L^1(\Omega) \rightarrow \mathcal{Q}_h, \quad \projection{\phi} = \sum_{K \in \mathcal{T}_h}\frac{1_{K}(x)}{|K|}\int_K \phi (x) \dx,
\end{equation}
where $|K|\approx h^2$ is the Lebesgue measure of $K$ and $ 1_{K}$ is the characteristic function. 

Further, we introduce the average and jump operators at the interface $\sigma$ for $v \in \mathcal{Q}_h$:
\[
\avs{v}(x) = \frac{v^{\rm in}(x) + v^{\rm out}(x) }{2},\ \ \ \
\jump{ v }(x) = v^{\rm out}(x) - v^{\rm in}(x),\]
where
\[
v^{\rm out}(x) = \lim_{\delta \to 0+} v(x + \delta \vc{n}),\ \ \ \
v^{\rm in}(x) = \lim_{\delta \to 0+} v(x - \delta \vc{n}),
\]
and $\vn$ is the outer normal vector to $\sigma$.

\subsection{Finite volume evolution Galerkin method}

We now proceed to formulate the finite volume evolution Galerkin (FVEG) method for the hyperbolic conservation law \eqref{pde}. Letting $\vU^n_h\in \mathcal{Q}_h^N$,  the FVEG method reads
\begin{equation}\label{FVEG-ex}
\vU^{n+1}_K=\vU^n_K -   \int^{\Delta t}_0 \frac{1}{|K|}  \sum_{\sigma \in \Sigma(K)} \int_{\sigma}\vF\left(\vU^{{n+\tau/{\Delta t}}}\right) \cdot \vn \dS_x  \dtau, \quad \vU_K = (\vU_h)|_K,
\end{equation}
where $\TS$ is the time step and $\vU^{{n+\tau/{\Delta t}}}$ is evolved by using the approximate evolution operator $E_{\tau}$ 
\begin{equation}\label{EG-app}
\vU^{{n+\tau/{\Delta t}}}(x)=E_{\tau}\vU_h^n, \quad x\in \sigma.
\end{equation}

Time interval $(0,T)$, $T>0$, is divided into time subintervals $(t^n, t^{n+1})$, $t^{n+1}=t^n+\Delta t$.

Next, we approximate the time integral with the midpoint rule and the space integral over cell-interface $\sigma$ with the Simpson quadrature rule, and obtain
\begin{subequations}\label{scheme}
\begin{align}
&\vU^{n+1}_K=\vU^n_K -   \frac{|\sigma| \TS}{|K|}  \sum_{\sigma \in \Sigma(K)} \vF_{\sigma}^{EG}\cdot \vn, \\ \label{FEG}
& \vF_{\sigma}^{EG} =  \frac16 \vF\Big(\vU^{EG}_A\Big) + \frac46 \vF\Big(\vU^{EG}_S \Big)+ \frac16 \vF\Big(\vU^{EG}_B\Big), 
\end{align} 
\end{subequations}
where $\vU^{EG}_X = \vU^{n+1/2}(X),~X \in \{A,B,S\}$ and $A,B$ are the corner points of the cell-interface $\sigma$, $S$ is the barycenter of $\sigma$, see Figure~\ref{fig:mesh}. 
For the prediction of $\vU^{EG}_X $ we have used the second-order approximate evolution operator in \eqref{EG-app}, which will be introduced in Section~\ref{sec_EG2-wave} and Section~\ref{sec_EG2-Euler}.

\begin{figure}[!h]
\centering
\begin{tikzpicture}[scale=1.0]
\draw[-,very thick](0,-2)--(2.5,-2)--(2.5,2)--(0,2)--(0,-2)--(-2.5,-2)--(-2.5,2)--(0,2);
\draw[-,very thick](0,-2)--(2.5,-2)--(2.5,0)--(0,0)--(0,-2)--(-2.5,-2)--(-2.5,0)--(0,0);
\draw[-,very thick](0,-2)--(2.5,-2)--(2.5,4)--(0,4)--(0,-2)--(-2.5,-2)--(-2.5,4)--(0,4);

\draw[gray] (0,1) circle (1);
\draw[red] (0,2) circle (1);

\draw[-,very thick, green=90!, pattern=north east lines, pattern color=green!30] (0,0)--(1.25,0)--(1.25,2)--(0,2)--(0,0);
\draw[-,very thick, blue=90!, pattern= north west lines, pattern color=blue!30] (0,0)--(0,2)--(-1.25,2)--(-1.25,0)--(0,0);
\draw[-,very thick, orange=90!, pattern=north east lines, pattern color=orange!30] (0,1)--(-1.25,1)--(-1.25,3)--(0,3)--cycle;
\draw[-,very thick, yellow=90!, pattern= north west lines, pattern color=yellow!30] (0,1)--(0,3)--(1.25,3)--(1.25,1)--(0,1);

\path node at (-1.75,1) { $L$};
\path node at (1.75,1) { $R$};
\path node at (-1.75,3) { $UL$};
\path node at (1.75,3) { $UR$};
\path node at (-1.75,-1) { $BL$};
\path node at (1.75,-1) { $BR$};
\path node at (-1.25,1) {$ \bullet$};
\path node at (-1.4,0.8) {$ x_L$};
\path node at (1.25,1) {$\bullet$};
\path node at (1.4,0.8) {$ x_R$};
\path node at (0,1) {$\bullet$};
\path node at (0.3,0.8) {$S$};
\path node at (0,2) {$\bullet$};
\path node at (0,0) {$\bullet$};
\path node at (0.3,1.8) {$A$};
\path node at (0.3,-0.2) {$B$};
\path node at (0,3) {$\bullet$};
\path node at (0,-1) {$\bullet$};
\path node at (-1.25,3) {$\bullet$};
\path node at (1.25,3) {$\bullet$};
\path (-0.3,0.9) node[rotate=90] { $\sigma={L|R}$};
\path (-0.3,3.2) node[rotate=90] { $\sigma_+$};
\path (-0.3,-1.1) node[rotate=90] { $\sigma_-$};
\path (-1.4,-0.2) node[] { $\sigma_{L^-}$};
\path (-1.4,2.2) node[] { $\sigma_{L^+}$};
\path (1.4,-0.2) node[] { $\sigma_{R^-}$};
\path (1.4,2.2) node[] { $\sigma_{R^+}$};
 \end{tikzpicture}
\caption{Illustration of the update point values on a Cartesian mesh.}
 \label{fig:mesh}
\end{figure}
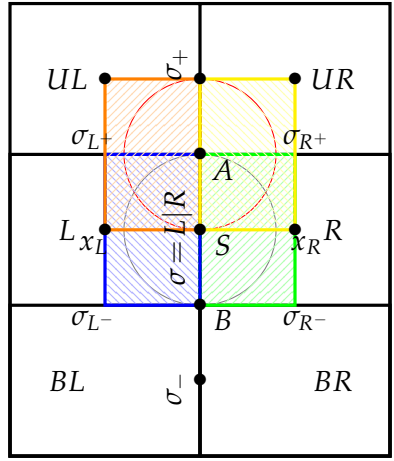

A crucial ingredient of the FVEG method is the evolution operator $E_{\tau}$ that takes mutidimensional evolution of a solution into account. In Appendices A and B we present such evolution operator for the wave equation system and the linearized Euler equations. Both evolution operators give the exact solution at a new time step $(t+\Delta t)$, however they are implicit in time due to the integration over the mantle of the bicharacteristic cone. In the following sections we present suitable approximation yielding time explicit evolution of a numerical solution.

\subsection{FVEG2  for the wave equation system}\label{sec_EG2-wave}
 In \cite{LMW:2000}, various approximations of the exact evolution operator \eqref{2.5}-\eqref{2.7} were derived. In this paper we will work with the second order approximate of time integrals 
 \[
 \int^{t+\Delta t}_{t}\int^{2\pi}_0S({\tilde t},\theta)\dthe\dt, \ \  \int^{t+\Delta t}_{t}\int^{2\pi}_0S({\tilde t},\theta)\cos\theta \dthe\dt \ \mbox{ and } \ \int^{t+\Delta t}_{t}\int^{2\pi}_0S({\tilde t},\theta)\sin\theta\dthe\dt
 \]
 based on the trapezoidal quadrature, here $S( \tilde t,\theta)=c[u_x(\tilde x,\tilde y, \tilde t)\sin^2\theta-(u_y(\tilde x,\tilde y,\tilde t)+v_x(\tilde x,\tilde y,\tilde t))\sin\theta\cos\theta+v_y(\tilde x,\tilde y,\tilde t)\cos^2\theta],$ cf.\eqref{A4}. The application of integration by parts yields the following approximate evolution operator. To be consistent with \cite{LMW:2000}, we denote it EG2: 
 \begin{subequations}\label{EG2-w}
\begin{equation}\label{2.9}
\phi_P=\frac{1}{\pi}\int^{2\pi}_{0}\phi_Q-u_Q\cos\theta-v_Q\sin\theta\dthe-\phi_{P'}+O(\Delta t^3),
\end{equation}
\begin{equation}\label{2.10}
u_P=\frac{1}{\pi}\int^{2\pi}_{0}-\phi_Q\cos\theta+u_Q\left(2\cos^2\theta-\frac{1}{2}\right)+2v_Q\sin\theta\cos\theta\dthe+O(\Delta t^3),
\end{equation}
\begin{equation}\label{2.11}
v_P=\frac{1}{\pi}\int^{2\pi}_{0}-\phi_Q\sin\theta+2u_Q\sin\theta\cos\theta+v_Q\left(2\sin^2\theta-\frac{1}{2}\right) \dthe+O(\Delta t^3),
\end{equation}
 \end{subequations}
where $P=(x,y,t^{n+1})$, $P^\prime=(x,y,t^n)$, $Q=(x-c\cos{\theta}\Delta t,y-c\sin{\theta}\Delta t, t^n)$.

Let us denote by $L, R, UL, UR, BL, BR$ the left, right, upper left, upper right, bottom left, bottom right neighbouring mesh cells to the edge $\sigma=L|R$, respectively, see Figure \ref{fig:mesh}.  Applying \eqref{2.9}-\eqref{2.11} we obtain following explicit expressions of $(\phi_P, u_P, v_P)$ for $P=S$ and $P=A$. The expression for the case $P=B$ shall be derived through mirror symmetry of $P=A$. 

\paragraph{Case $P=S$:}
\begin{subequations}\label{EG2-w-S}
\begin{eqnarray}\label{2.12}
\phi_P&=&\frac{1}{\pi}\bigg{(}\int^{\frac{\pi}{2}}_{-\frac{\pi}{2}}\phi_R\dthe+\int^{\frac{3\pi}{2}}_{\frac{\pi}{2}}\phi_L\dthe\bigg{)}-\frac{1}{\pi}\bigg{(}\int^{\frac{\pi}{2}}_{-\frac{\pi}{2}}u_R\cos\theta \dthe+\int^{\frac{3\pi}{2}}_{\frac{\pi}{2}}u_L\cos\theta\dthe\bigg{)}\\\nonumber
&~&-\frac{1}{\pi}\bigg{(}\int^{\frac{\pi}{2}}_{-\frac{\pi}{2}}v_R\sin\theta\dthe+\int^{\frac{3\pi}{2}}_{\frac{\pi}{2}}v_L\sin\theta\dthe\bigg{)}-\phi_{P'}\\\nonumber
&=&(\phi_R+\phi_L)-\frac{2}{\pi}(u_R-u_L)-\phi_{P'},
\end{eqnarray}
further, we have
\begin{eqnarray}
u_P&=&\frac{1}{\pi}\bigg{(}\int^{\frac{\pi}{2}}_{-\frac{\pi}{2}}-\phi_R\cos\theta\dthe+\int^{\frac{3\pi}{2}}_{\frac{\pi}{2}}-\phi_L\cos\theta\dthe\bigg{)}+\frac{1}{\pi}\bigg{(}\int^{\frac{\pi}{2}}_{-\frac{\pi}{2}}u_R(2\cos^2\theta-\frac{1}{2})\dthe\\\nonumber
&~&+\int^{\frac{3\pi}{2}}_{\frac{\pi}{2}}u_L(2\cos^2\theta-\frac{1}{2})\dthe\bigg{)}+\frac{1}{\pi}\bigg{(}\int^{\frac{\pi}{2}}_{-\frac{\pi}{2}}2v_R\sin\theta\cos\theta\dthe+\int^{\frac{3\pi}{2}}_{\frac{\pi}{2}}2v_L\sin\theta\cos\theta\dthe\bigg{)}\\\nonumber
&=& -\frac{2}{\pi}(\phi_R-\phi_L)+\frac{1}{2}(u_R+u_L),
\end{eqnarray}
\begin{eqnarray}
v_P&=&\frac{1}{\pi}\bigg{(}\int^{\frac{\pi}{2}}_{-\frac{\pi}{2}}-\phi_R\sin\theta\dthe+\int^{\frac{3\pi}{2}}_{\frac{\pi}{2}}-\phi_L\sin\theta\dthe\bigg{)}+\frac{1}{\pi}\bigg{(}\int^{\frac{\pi}{2}}_{-\frac{\pi}{2}}2u_R\sin\theta\cos\theta\dthe\\\nonumber
&~&+\int^{\frac{3\pi}{2}}_{\frac{\pi}{2}}2u_L\sin\theta\cos\theta\dthe\bigg{)}+\frac{1}{\pi}\bigg{(}\int^{\frac{\pi}{2}}_{-\frac{\pi}{2}}v_R(2\sin^2\theta-\frac{1}{2})\dthe+\int^{\frac{3\pi}{2}}_{\frac{\pi}{2}}u_L(2\sin^2\theta-\frac{1}{2})\dthe\bigg{)}\\\nonumber
&=&\frac{1}{2}(v_R+v_L).
\end{eqnarray}
\end{subequations}

\paragraph{Case $P=A$:}
\begin{subequations}\label{EG2-w-A}
\begin{eqnarray}\label{2.15}
\phi_P&=&\frac{1}{\pi}\bigg{(}\int^0_{-\frac{\pi}{2}}\phi_R\dthe+\int^{\frac{\pi}{2}}_0\phi_{UR}\dthe+\int^{\pi}_{\frac{\pi}{2}}\phi_{UL}\dthe+\int^{\frac{3\pi}{2}}_{\pi}\phi_L\dthe\bigg{)}\\\nonumber
&&-\frac{1}{\pi}\bigg{(}\int^0_{-\frac{\pi}{2}}u_{R}\cos\theta\dthe+\int^{\frac{\pi}{2}}_0u_{UR}\cos\theta\dthe+\int^{\pi}_{\frac{\pi}{2}}u_{UL}\cos\theta\dthe+\int^{\frac{3\pi}{2}}_{\pi}u_L\cos\theta\dthe\bigg{)}\\\nonumber
&&-\frac{1}{\pi}\bigg{(}\int^0_{-\frac{\pi}{2}}v_{R}\sin\theta\dthe+\int^{\frac{\pi}{2}}_0v_{UR}\sin\theta\dthe+\int^{\pi}_{\frac{\pi}{2}}v_{UL}\sin\theta\dthe+\int^{\frac{3\pi}{2}}_{\pi}v_L\sin\theta\dthe\bigg{)}-\phi_{P'}\\\nonumber
&=&\frac{1}{2}(\phi_R+\phi_{UR}+\phi_{UL}+\phi_L)-\frac{1}{\pi}(u_R+u_{UR}-u_{UL}-u_{L})+\frac{1}{\pi}(v_{R}-v_{UR}-v_{UL}+v_{L})-\phi_{P'},
\end{eqnarray}
\begin{eqnarray}
u_P&=&-\frac{1}{\pi}\bigg{(}\int^0_{-\frac{\pi}{2}}\phi_R\cos\theta\dthe+\int^{\frac{\pi}{2}}_0\phi_{UR}\cos\theta\dthe+\int^{\pi}_{\frac{\pi}{2}}\phi_{UL}\cos\theta\dthe+\int^{\frac{3\pi}{2}}_{\pi}\phi_L\cos\theta\dthe\bigg{)}\\\nonumber
&&+\frac{1}{\pi}\bigg{(}\int^0_{-\frac{\pi}{2}}u_{R}\left(2\cos^2\theta-\frac{1}{2}\right)\dthe+\int^{\frac{\pi}{2}}_0u_{UR}\left(2\cos^2\theta-\frac{1}{2}\right)\dthe+\int^{\pi}_{\frac{\pi}{2}}u_{UL}\left(2\cos^2\theta-\frac{1}{2}\right)\dthe\\\nonumber
&&+\int^{\frac{3\pi}{2}}_{\pi}u_L\left(2\cos^2\theta-\frac{1}{2}\right)\dthe\bigg{)}+\frac{1}{\pi}\bigg{(}\int^0_{-\frac{\pi}{2}}2v_{R}\sin\theta\cos\theta\dthe+\int^{\frac{\pi}{2}}_02v_{UR}\sin\theta\cos\theta\dthe\\\nonumber
&&+\int^{\pi}_{\frac{\pi}{2}}2v_{UL}\sin\theta\cos\theta\dthe+\int^{\frac{3\pi}{2}}_{\pi}2v_L\sin\theta\cos\theta\dthe\bigg{)}\\\nonumber
&=&-\frac{1}{\pi}(\phi_R+\phi_{UR}-\phi_{UL}-\phi_{L})+\frac{1}{4}(u_R+u_{UR}+u_{UL}+u_L)+\frac{1}{\pi}(-v_{R}+v_{UR}-v_{UL}+v_L),
\end{eqnarray}
\begin{eqnarray}
v_P&=&-\frac{1}{\pi}\bigg{(}\int^0_{-\frac{\pi}{2}}\phi_R\sin\theta\dthe+\int^{\frac{\pi}{2}}_0\phi_{UR}\sin\theta\dthe+\int^{\pi}_{\frac{\pi}{2}}\phi_{UL}\sin\theta\dthe+\int^{\frac{3\pi}{2}}_{\pi}\phi_L\sin\theta\dthe\bigg{)}\\\nonumber
&&+\frac{1}{\pi}\bigg{(}\int^0_{-\frac{\pi}{2}}2u_{R}\sin\theta\cos\theta\dthe+\int^{\frac{\pi}{2}}_02u_{UR}\sin\theta\cos\theta\dthe+\int^{\pi}_{\frac{\pi}{2}}2u_{UL}\sin\theta\cos\theta\dthe\\\nonumber
&&+\int^{\frac{3\pi}{2}}_{\pi}2u_L\sin\theta\cos\theta\dthe\bigg{)}+\frac{1}{\pi}\bigg{(}\int^0_{-\frac{\pi}{2}}v_{R}\left(2\sin^2\theta-\frac{1}{2}\right)\dthe+\int^{\frac{\pi}{2}}_0v_{UR}\left(2\sin^2\theta-\frac{1}{2}\right)\dthe\\\nonumber
&&+\int^{\pi}_{\frac{\pi}{2}}v_{UL}\left(2\sin^2\theta-\frac{1}{2}\right)\dthe+\int^{\frac{3\pi}{2}}_{\pi}v_L\left(2\sin^2\theta-\frac{1}{2}\right)\dthe\bigg{)}\\\nonumber
&=&\frac{1}{\pi}(\phi_R-\phi_{UR}-\phi_{UL}+\phi_{L})+\frac{1}{\pi}(-u_R+u_{UR}-u_{UL}+u_L)+\frac{1}{4}(v_{R}+v_{UR}+v_{UL}+v_L).
\end{eqnarray}
\end{subequations}

\subsection{FVEG2 for the Euler equations}\label{sec_EG2-Euler}
Applying the same approximation of the time integrals to the exact evolution operator for the linearized Euler equations, see Appendix \ref{evo-Euler}, we obtain the following approximate evolution operator for the density, velocities and pressure.
 \begin{subequations}\label{EG2-E}
\begin{equation}\label{2.23}
\rho_P=\rho_{P^{\prime}}-\frac{p_{P^{\prime}}}{c^{{\prime}2}}+\frac{1}{\pi}\int^{2\pi}_{0}\frac{p_Q}{c^{{\prime}2}}-\frac{\rho^{\prime}}{c^{\prime}}u_Q\cos\theta-\frac{\rho^{\prime}}{c^{\prime}}v_Q\sin\theta\dthe+O(\Delta t^3),
\end{equation}
\begin{equation}\label{2.24}
u_P=\frac{1}{\pi}\int^{2\pi}_{0}-\frac{p_Q}{\rho^{\prime}c^{\prime}}\cos\theta+u_Q(2\cos^2\theta-\frac{1}{2})+2v_Q\sin\theta\cos\theta\dthe+O(\Delta t^3),
\end{equation}
\begin{equation}\label{2.25}
v_P=\frac{1}{\pi}\int^{2\pi}_{0}-\frac{p_Q}{\rho^{\prime}c^{\prime}}\sin\theta+2u_Q\sin\theta\cos\theta+v_Q(2\sin^2\theta-\frac{1}{2}) \dthe+O(\Delta t^3),
\end{equation}
\begin{equation}\label{2.26}
p_P=-p_{P^{\prime}}+\frac{1}{\pi}\int^{2\pi}_0p_Q-\rho^{\prime}c^{\prime}u_Q\cos\theta-\rho^{\prime}c^{\prime}v_Q\sin\theta\dthe+O(\Delta t^3),
\end{equation}
 \end{subequations}
where $P=(x,y,t^{n+1})$, $P^\prime=(x-u^\prime\Delta t,y-v^\prime\Delta t,t^n)$, $Q=(x-(u^\prime-c^\prime\cos{\theta})\Delta t,y-(v^\prime-c^\prime\sin{\theta})\Delta t, t^n)$. We recall that we have linearized locally the Euler equations and $\rho^\prime, u^\prime, v^\prime, c^\prime$ are fixed linearized density, flow and sound velocity, see Appendix \ref{evo-Euler}.

Applying the approximate evolution operator \eqref{2.23}-\eqref{2.26} to a piecewise constant data on a regular rectangular grid, we obtain analogous expression for the predicted solution $\vU^{n+1/2}(P)$ for $P=A, B$ and $P=S$ as above. Here $S$ is a midpoint of  an arbitrary fixed edge $\sigma=L|R$ and $A, B$ are the corner points.

Note that the bicharacteristic cone is slanted by $(u^\prime\frac{\Delta t}{2}, v^\prime\frac{\Delta t}{2})$ and therefore the base of the cone may get various positions. Taking the CFL stability condition $\max \big((u^\prime+c^\prime)\Delta t, (v^\prime+c^\prime)\Delta t\big)\leqslant h$ into account the base circle only falls down into neighbouring cells to the corresponding edge $\sigma=L|R$.

In what follow, we present two cases for $A$ and $S$.
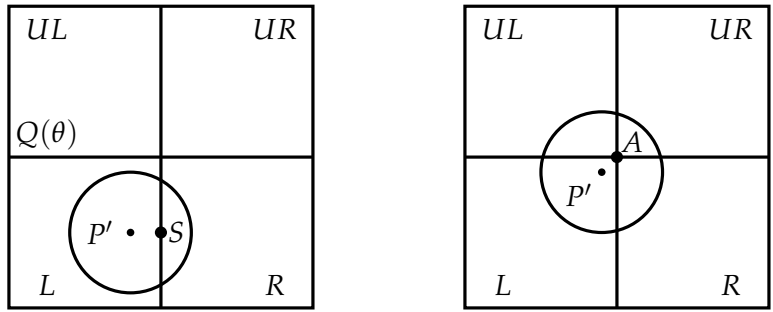
\begin{figure}[!h]
\centering 
\begin{tikzpicture}
\draw[help lines, black, very thick] (-2,-2) -- (2,-2)--(2,2)--(-2,2)--cycle;
\draw [very thick](-2,0)--(2,0);
\draw[very thick](0,-2)--(0,2);
\draw [very thick] (-0.4,-1) circle (0.8);
\fill (0,-1) circle (0.08);
\fill (-0.4,-1) circle (0.05);
\node at (0.2,-1) {$S$};
\node at (-0.8,-1) {$P^\prime$};
\node at (-1.5,0.3){$Q(\theta)$};
\node at (-1.5,1.7) {$UL$};
\node at (1.5,1.7) {$UR$};
\node at (-1.5,-1.7) {$L$};
\node at (1.5,-1.7) {$R$};

\draw[help lines, black, very thick] (4,-2) -- (8,-2)--(8,2)--(4,2)--cycle;
\draw [very thick](4,0)--(8,0);
\draw[very thick](6,-2)--(6,2);
\node at (6.2,0.2) {$A$};
\node at (4.5,1.7) {$UL$};
\node at (7.5,1.7) {$UR$};
\node at (4.5,-1.7) {$L$};
\node at (7.5,-1.7) {$R$};
\node at (5.5,-0.5) {$P^\prime$};
\fill (6,0) circle (0.08);
\fill (5.8,-0.2) circle (0.05);
\draw [very thick] (5.8,-0.2) circle (0.8);
\end{tikzpicture}
\caption{Illustration of possible update of point values for the linearized Euler equations.}
\end{figure}

\paragraph{Case $P=S$:}
we obtain the following approximation,
\begin{subequations}\label{EG2-E-S}
\begin{eqnarray}\label{2.33}
\rho_P&=&\rho_{P^{\prime}}-\frac{p_{P^{\prime}}}{c^{{\prime}2}}+\frac{1}{\pi}\int^{2\pi}_{0}\frac{p_Q}{c^{{\prime}2}}-\frac{\rho^{\prime}}{c^{\prime}}u_Q\cos\theta-\frac{\rho^{\prime}}{c^{\prime}}v_Q\sin\theta\dthe\\\nonumber
&=&\rho_L-\frac{p_{L}}{c^{{\prime}2}}+\frac{1}{\pi c^{{\prime}2}}\bigg{(}\int^\alpha_{-\alpha}p_R\dthe+\int^{2\pi-\alpha}_{\alpha}p_L\dthe\bigg{)}-\frac{\rho^\prime}{\pi c^\prime}\bigg{(}\int^\alpha_{-\alpha}u_R\cos\theta\dthe\\\nonumber
&&+\int^{2\pi-\alpha}_{\alpha}u_L\cos\theta\dthe\bigg{)}-\frac{\rho^\prime}{\pi c^\prime}\bigg{(}\int^\alpha_{-\alpha}v_R\sin\theta\dthe+\int^{2\pi-\alpha}_{\alpha}v_L\sin\theta\dthe\bigg{)}\\\nonumber
&=&\rho_L+\frac{2\alpha}{\pi c^{{\prime}2}}(p_R-p_L)-\frac{2\rho^\prime\sin\alpha}{\pi c^\prime}(u_R-u_L),
\end{eqnarray}
where $\alpha=\arccos(u^\prime/c^\prime)$,
\begin{eqnarray}\label{2.34}
u_P&=&\frac{1}{\pi}\int^{2\pi}_{0}-\frac{p_Q}{\rho^{\prime}c^{\prime}}\cos\theta+u_Q(2\cos^2\theta-\frac{1}{2})+2v_Q\sin\theta\cos\theta\dthe\\\nonumber
&=&-\frac{1}{\pi\rho^{\prime}c^{\prime}}\bigg{(}\int^\alpha_{-\alpha}p_R\cos\theta\dthe+\int^{2\pi-\alpha}_{\alpha}p_L\cos\theta\dthe\bigg{)}+\frac{1}{\pi}\bigg{(}\int^\alpha_{-\alpha}u_R(2\cos^2\theta-\frac{1}{2})\dthe\\\nonumber
&&+\int^{2\pi-\alpha}_{\alpha}u_L(2\cos^2\theta-\frac{1}{2})\dthe\bigg{)}+\frac{1}{\pi}\bigg{(}\int^\alpha_{-\alpha}v_R2\sin\theta\cos\theta\dthe+\int^{2\pi-\alpha}_{\alpha}v_L2\sin\theta\cos\theta\dthe\bigg{)}\\\nonumber
&=&-\frac{2\alpha\sin\alpha}{\pi\rho^{\prime}c^{\prime}}(p_R-p_L)+\frac{\sin 2\alpha+\alpha}{\pi}(u_R-u_L)+u_L.
\end{eqnarray}
Similarly, we can express $v_P$ as
\begin{eqnarray}\label{2.35}
v_P&=&\frac{1}{\pi}\int^{2\pi}_{0}-\frac{p_Q}{\rho^{\prime}c^{\prime}}\sin\theta+2u_Q\sin\theta\cos\theta+v_Q(2\sin^2\theta-\frac{1}{2})\dthe\\\nonumber
&=&-\frac{1}{\pi\rho^{\prime}c^{\prime}}\bigg{(}\int^\alpha_{-\alpha}p_R\sin\theta\dthe+\int^{2\pi-\alpha}_{\alpha}p_L\sin\theta\dthe\bigg{)}+\frac{1}{\pi}\bigg{(}\int^\alpha_{-\alpha}u_R2\sin\theta\cos\theta\dthe\\\nonumber
&&+\int^{2\pi-\alpha}_{\alpha}u_L2\sin\theta\cos\theta\dthe\bigg{)}+\frac{1}{\pi}\bigg{(}\int^\alpha_{-\alpha}v_R(2\sin^2\theta-\frac{1}{2})\dthe+\int^{2\pi-\alpha}_{\alpha}v_L(2\sin^2\theta-\frac{1}{2})\dthe\bigg{)}
\\\nonumber
&=&\frac{\alpha-\sin 2\alpha}{\pi}(v_R-v_L)+v_L,
\end{eqnarray}
and
\begin{eqnarray}\label{2.36}
p_p&=&-p_{P^{\prime}}+\frac{1}{\pi}\int^{2\pi}_0p_Q-\rho^{\prime}c^{\prime}u_Q\cos\theta-\rho^{\prime}c^{\prime}v_Q\sin\theta\dthe\\\nonumber
&=&-p_L+\frac{1}{\pi}\bigg{(}\int^\alpha_{-\alpha}p_R\dthe+\int^{2\pi-\alpha}_{\alpha}p_L\dthe\bigg{)}-\frac{\rho^\prime c^\prime}{\pi }\bigg{(}\int^\alpha_{-\alpha}u_R\cos\theta\dthe+\int^{2\pi-\alpha}_{\alpha}u_L\cos\theta\dthe\bigg{)}\\\nonumber
&&-\frac{\rho^\prime c^\prime}{\pi }\bigg{(}\int^\alpha_{-\alpha}v_R\sin\theta\dthe+\int^{2\pi-\alpha}_{\alpha}v_L\sin\theta\dthe\bigg{)}\\\nonumber
&=&p_L+\frac{2\alpha}{\pi}(p_R-p_L)-\frac{2\rho^\prime c^\prime\sin\alpha}{\pi}(u_R-u_L).
\end{eqnarray}
\end{subequations}

\paragraph{Case $P=A$:}
the EG2 approximate evolution operator reads
\begin{subequations}\label{EG2-E-A}
\begin{eqnarray}\label{2.37}
\rho_P&=&\rho_L-\frac{2p_L}{c^{{\prime}2}}+\frac{1}{\pi}\bigg{(}\int^{\alpha_1}_{-\alpha_2}p_R\dthe+\int^{\alpha_2}_{\alpha_1}p_{UR}\dthe+\int^{\pi-\alpha_1}_{\alpha_2}p_{UL}\dthe+\int^{2\pi-\alpha_2}_{\pi-\alpha_1}p_L\dthe\bigg{)}\\\nonumber
&&-\frac{\rho^\prime}{\pi c^\prime}\bigg{(}\int^{\alpha_1}_{-\alpha_2}u_{R}\cos\theta\dthe+\int^{\alpha_2}_{\alpha_1}u_{UR}\cos\theta\dthe+\int^{\pi-\alpha_1}_{\alpha_2}u_{UL}\cos\theta\dthe+\int^{2\pi-\alpha_2}_{\pi-\alpha_1}u_L\cos\theta\dthe\bigg{)}\\\nonumber
&&-\frac{\rho^\prime}{\pi c^\prime}\bigg{(}\int^{\alpha_1}_{-\alpha_2}v_{R}\sin\theta\dthe+\int^{\alpha_2}_{\alpha_1}v_{UR}\sin\theta\dthe+\int^{\pi-\alpha_1}_{\alpha_2}v_{UL}\sin\theta\dthe+\int^{2\pi-\alpha_2}_{\pi-\alpha_1}v_L\sin\theta\dthe\bigg{)}\\\nonumber
&=&\rho_L-\frac{1}{c^{{\prime}2}}(p_L-p_{UL})+\frac{\alpha_1}{\pi c^{\prime 2}}(p_R-p_{UR}-p_{UL}+p_L)+\frac{\alpha_2}{\pi c^{\prime 2}}(p_R+p_{UR}-p_{UL}-p_L)
\\\nonumber
&&-\frac{\rho^\prime\sin\alpha_1}{\pi c^\prime}(u_R-u_{UR}+u_{UL}-u_L)-\frac{\rho^\prime\sin\alpha_2}{\pi c^\prime}(u_R+u_{UR}-u_{UL}-u_L)\\\nonumber
&&+\frac{\rho^\prime\cos\alpha_1}{\pi c^\prime}(v_R-v_{UR}-v_{UL}+v_L)-\frac{\rho^\prime\cos\alpha_2}{\pi c^\prime}(v_R-v_{UR}+v_{UL}-v_L)),
\end{eqnarray}
\begin{eqnarray}\label{2.38}
u_P&=&-\frac{1}{\pi\rho^{\prime}c^{\prime}}\bigg{(}\int^{\alpha_1}_{-\alpha_2}p_R\cos\theta\dthe+\int^{\alpha_2}_{\alpha_1}p_{UR}\cos\theta\dthe+\int^{\pi-\alpha_1}_{\alpha_2}p_{UL}\cos\theta\dthe\\\nonumber
&&+\int^{2\pi-\alpha_2}_{\pi-\alpha_1}p_L\cos\theta\dthe\bigg{)}+\frac{1}{\pi}\bigg{(}\int^{\alpha_1}_{-\alpha_2}u_R(2\cos^2\theta-\frac{1}{2})\dthe+\int^{\alpha_2}_{\alpha_1}u_{UR}(2\cos^2\theta-\frac{1}{2})\dthe\\\nonumber
&&+\int^{\pi-\alpha_1}_{\alpha_2}u_{UL}(2\cos^2\theta-\frac{1}{2})\dthe+\int^{2\pi-\alpha_2}_{\pi-\alpha_1}u_L(2\cos^2\theta-\frac{1}{2})\dthe\bigg{)}+\frac{1}{\pi}\bigg{(}\int^{\alpha_1}_{-\alpha_2}v_R2\sin\theta\cos\theta\dthe\\\nonumber
&&+\int^{\alpha_2}_{\alpha_1}v_{UR}2\sin\theta\cos\theta\dthe+\int^{\pi-\alpha_1}_{\alpha_2}v_{UL}2\sin\theta\cos\theta\dthe+\int^{2\pi-\alpha_2}_{\pi-\alpha_1}v_L2\sin\theta\cos\theta\dthe\bigg{)}\\\nonumber
&=&-\frac{\sin\alpha_1}{\rho\prime c^\prime}(p_R-p_{UR}+p_{UL}-p_L)-\frac{\sin\alpha_2}{\rho\prime c^\prime}(p_R+p_{UR}-p_{UL}-p_L)\\\nonumber
&&+\frac{\sin{2\alpha_1}+\alpha_1}{2\pi}(u_R-u_{UR}-u_{UL}+u_L)+\frac{\sin{2\alpha_2}+\alpha_2}{2\pi}(u_R+u_{UR}-u_{UL}-u_L)\\\nonumber
&&+\frac{\sin^2\alpha_1-\sin^2\alpha_2}{\pi}(v_R-v_{UR}+v_{UL}-v_L)+\frac{1}{2}u_{UL}+\frac{1}{2}u_{L},
\end{eqnarray}
\begin{eqnarray}\label{2.39}
v_P&=&-\frac{1}{\pi\rho^{\prime}c^{\prime}}\bigg{(}\int^{\alpha_1}_{-\alpha_2}p_R\sin\theta\dthe+\int^{\alpha_2}_{\alpha_1}p_{UR}\sin\theta\dthe+\int^{\pi-\alpha_1}_{\alpha_2}p_{UL}\sin\theta\dthe\\\nonumber
&&+\int^{2\pi-\alpha_2}_{\pi-\alpha_1}p_L\sin\theta\dthe\bigg{)}+\frac{1}{\pi}\bigg{(}\int^{\alpha_1}_{-\alpha_2}u_R2\sin\theta\cos\theta\dthe+\int^{\alpha_2}_{\alpha_1}u_{UR}2\sin\theta\cos\theta\dthe\\\nonumber
&&+\int^{\pi-\alpha_1}_{\alpha_2}u_{UL}2\sin\theta\cos\theta\dthe+\int^{2\pi-\alpha_2}_{\pi-\alpha_1}u_{L}2\sin\theta\cos\theta\dthe\bigg{)}+\frac{1}{\pi}\bigg{(}\int^{\alpha_1}_{-\alpha_2}v_R(2\sin^2\theta-\frac{1}{2})\dthe\\\nonumber
&&+\int^{\alpha_2}_{\alpha_1}v_{UR}(2\sin^2\theta-\frac{1}{2})\dthe+\int^{\pi-\alpha_1}_{\alpha_2}v_{UL}(2\sin^2\theta-\frac{1}{2})\dthe+\int^{2\pi-\alpha_2}_{\pi-\alpha_1}v_L(2\sin^2\theta-\frac{1}{2})\dthe\bigg{)}\\\nonumber
&=&\frac{\cos\alpha_1}{\rho\prime c^\prime}(p_R-p_{UR}-p_{UL}+p_L)-\frac{\cos\alpha_2}{\rho\prime c^\prime}(p_R-p_{UR}+p_{UL}-p_L)\\\nonumber
&&+\frac{\alpha_1-\sin{2\alpha_1}}{2\pi}(v_R-v_{UR}-v_{UL}+v_L)+\frac{\alpha_2-\sin{2\alpha_2}}{2\pi}(v_R+v_{UR}-v_{UL}-v_L)\\\nonumber
&&+\frac{\sin^2\alpha_1-\sin^2\alpha_2}{\pi}(u_R-u_{UR}+u_{UL}-u_L)+\frac{1}{2}v_{UL}+\frac{1}{2}v_{L},
\end{eqnarray}
\begin{eqnarray}\label{2.40}
p_P&=&-p_{L}+\frac{1}{\pi}\bigg{(}\int^{\alpha_1}_{-\alpha_2}p_R\dthe+\int^{\alpha_2}_{\alpha_1}p_{UR}\dthe+\int^{\pi-\alpha_1}_{\alpha_2}p_{UL}\dthe+\int^{2\pi-\alpha_2}_{\pi-\alpha_1}p_L\dthe\bigg{)}\\\nonumber
&&+\frac{\rho^\prime c^\prime}{\pi}\bigg{(}\int^{\alpha_1}_{-\alpha_2}u_{R}\cos\theta\dthe+\int^{\alpha_2}_{\alpha_1}u_{UR}\cos\theta\dthe+\int^{\pi-\alpha_1}_{\alpha_2}u_{UL}\cos\theta\dthe+\int^{2\pi-\alpha_2}_{\pi-\alpha_1}u_L\cos\theta\dthe\bigg{)}\\\nonumber
&&+\frac{\rho^\prime c^\prime}{\pi}\bigg{(}\int^{\alpha_1}_{-\alpha_2}v_{R}\sin\theta\dthe+\int^{\alpha_2}_{\alpha_1}v_{UR}\sin\theta\dthe+\int^{\pi-\alpha_1}_{\alpha_2}v_{UL}\sin\theta\dthe+\int^{2\pi-\alpha_2}_{\pi-\alpha_1}v_L\sin\theta\dthe\bigg{)}\\\nonumber
&=&\frac{\alpha_1}{\pi}(p_R-p_{UR}-p_{UL}+p_L)+\frac{\alpha_2}{\pi}(p_R+p_{UR}-p_{UL}-p_L)-\frac{\rho^\prime c^\prime\sin\alpha_1}{\pi }(u_R-u_{UR}+u_{UL}-u_L)\\\nonumber
&&-\frac{\rho^\prime c^\prime\sin\alpha_2}{\pi }(u_R+u_{UR}-u_{UL}-u_L)+\frac{\rho^\prime c^\prime\cos\alpha_1}{\pi}(v_R-v_{UR}-v_{UL}+v_L)\\\nonumber
&&-\frac{\rho^\prime c^\prime\cos\alpha_2}{\pi}(v_R-v_{UR}+v_{UL}-v_L)),
\end{eqnarray}
\end{subequations}
where $\alpha_1=\arcsin{v^\prime/c^\prime}$, $\alpha_2=\arccos{u^\prime/c^\prime}$.

We note that in order to compute numerically nonlinear Euler equaitons the local linearization is done by local averaging at the point $P^\prime$. Thus $\rho^\prime=\avs{\rho_h}$, $u^\prime=\avs{u_h}$, $v^\prime=\avs{v_h}$, $p^\prime=\avs{p_h}$, $c^\prime=\sqrt{\frac{\gamma p^\prime}{\rho^\prime}}$.

\section{Stability}\label{sec_sta}
In order to prove the convergence of the FVEG method we need to show its stability and consistency.
 
\begin{itemize}
\item For linear wave equation system the stability of the FVEG method  has been studied in \cite{LWZ:2006}. By means of the von Neumann analysis and estimate of the amplification matrix it has been shown that under a suitable CFL condition $c\Delta t/h\leq CFL <1$, the FVEG method is stable in the sense of $\|\vU_h\|_{L^\infty}\leq\|\vU_0\|_{L^\infty}$. 
 
\item To show the similar property for nonlinear system of the Euler equations is a nontrivial task. Following our work \cite[Lemma 3.1]{LMY} we may however prove that if a numerical scheme is entropy stable and there exists 
\begin{equation}\label{ass}
\underline{\rho},\overline{E}>0,~\mbox{such that}~0<\underline{\rho}\leq \rho_h,\quad E_h\leq \overline{E},\quad h\to 0,
\end{equation}
then numerical solution $\{\rho_h,\vm_h,E_h\}$ stays uniformly bounded, i.e. 
\begin{equation}\label{Assu}
\exists~C>0,\quad  \|(1/\rho_h, \rho_h, \vu_h, 1/E_h, E_h)\|_{L^\infty(\Omega; \R^{d+2})}\leq C, \quad  h\to 0.
\end{equation}
\end{itemize}

Hence, in what follows we may assume the stability property \eqref{Assu} holds when analysing the Euler problem. We point out that for nonlinear conservation law we need in addition the weak BV estimates generated from the entropy stability of the FVEG method. This question will be studied below.

\subsection{Entropy stability of the FVEG flux}

We proceed to analyze the entropy stability of the semi-discrete version of the FVEG method \eqref{scheme}, i.e.\
\begin{align}\label{scheme-1}
\frac{d}{d t}\vU_K= -\frac{|\sigma| }{|K|}  \sum_{\sigma \in \Sigma(K)} \vF_{\sigma}^{EG}(t)\cdot \vn,
\end{align} 
where $\vF_{\sigma}^{EG}$ is given in \eqref{FEG}. 
In this section we concentrate on the FVEG method in one-space dimension, i.e.
\[
\vF_{\sigma}^{EG} = \vF\left( \vU^{EG}_{\sigma} \right), \quad \sigma = L|R.
\]

We note that the analysis presented here can be directly generalized to a two-dimensional case by applying a direction by direction approach.

Following \cite{LME}, we construct a path in the phase space to connect $\vU_L$ and $\vU_R$ with $(N+1)$-states, i.e.
\[
 \{\vU^{j}_{\sigma}\}_{ j =1}^{N+1}  \quad \mbox{with}\quad \vU^{1}_{\sigma} = \vU_L, \quad \vU^{N+1}_{\sigma} = \vU_R,
 \] 
using the right eigenvectors of the associated linearized Jacobian matrix $\hat{\bm{A}}(\vU_L, \vU_R)$, $\hat{\bm{A}}(\vU, \vU)=\bm{A}(\vU)=\vF_{\vU}(\vU)$:
\[
\left\{\vcr^j_{\sigma},\ \|\vcr^j_{\sigma}\| = 1, \  j=1,\cdots,N  \right\}  \quad \mbox{such that}\quad  \vU^{j+1}_{\sigma} - \vU^{j}_{\sigma} = \alpha^{j}_{\sigma}\vcr^j_{\sigma}.
\]
We denote by $\left\{\vl^j_{\sigma},\ \|\vl^j_{\sigma}\| = 1, \  j=1,\cdots,N \right\}$ the corresponding left eigenvector system, i.e.\ $\braket{\vl_{\sigma}^j, \vcr^k_{\sigma}} = \delta_{jk}$, which yields $\alpha_{\sigma}^j=\braket{\vl_{\sigma}^j, \jump{\vU_h}_{\sigma}}$. 
Then, we connect from $\vU^{j}_{\sigma}$ to $\vU^{j+1}_{\sigma}$ with the sub-path $\vU_{\sigma}^j(\xi), \, \xi \in[-1/2, 1/2]$  through shock wave such that
\begin{align}
F(\vU^j_{\sigma}(\xi))-F(\vU^j_{\sigma})=\lambda^j_{\sigma}(\xi)(\vU^j_{\sigma}(\xi)-\vU^j_{\sigma}),\quad \mbox{with}\quad \vU^j_{\sigma}\left(-\frac{1}{2} \right) = \vU^j_{\sigma}, \quad \vU^j_{\sigma}\left(\frac{1}{2} \right) = \vU^{j+1}_{\sigma}.
\end{align}
Hence, the Rankin-Hugoniot conditions hold for $\xi = 1/2$, i.e.
\[
F(\vU^{j+1}_{\sigma})-F(\vU^j_{\sigma})=\lambda^j_{\sigma}(\vU^{j+1}_{\sigma}-\vU^j_{\sigma}), \quad \mbox{with}\quad \lambda^j_{\sigma}\left( \frac12 \right) = \lambda^j_{\sigma}.
\]
With the above notations, we know from the construction of the FVEG method that 
\begin{equation}\label{FEG-1D}
\vU^{EG}_{\sigma} = \vU_{L}+\sum\limits_{{j:\lambda^j_{\sigma}\leqslant 0}}\alpha^j_{\sigma} \vcr^j_{\sigma},
\quad
\vF^{EG}_{\sigma} = \avs{\vF(\vU_h)} - \frac12 \sum_{j=1}^N q_j \alpha^j_{\sigma} \vl^j_{\sigma},
\quad 
q_j = \abs{\lambda^j_{\sigma}}.
\end{equation}

Next, we study the entropy stability of $\vF^{EG}_{\sigma}$ by means of Tadmor's comparision principle, see \cite{LME}.  
To this end, we recall the explicit form of the entropy-conservative flux \cite{Tadmor-acta}:
\begin{equation}\label{FEC-1D}
\vF^{EC}_{\sigma} = \avs{\vF(\vU_h)} - \frac12 \sum_{j=1}^N q^*_j \alpha^j_{\sigma} \vl^j_{\sigma}
\end{equation}
with
\begin{equation}\label{qEC-1D}
q^*_j=\int^{\frac{1}{2}}_{-\frac{1}{2}}2\xi\braket{\bm{A} (\vU^j_{\sigma}(\xi))\vcr^j_{\sigma},\vcr^j_{\sigma}}\dxi,
\end{equation}
and calculate the difference between $q_j $ and $q^*_j $ stated in the following Lemma.

\begin{Lemma}\label{l31}
Suppose that
 \begin{align} \label{ass-flux}
 |\bar{\vcr}^j_{\sigma}-\vcr^j_{\sigma}|+|\bar{\lambda}^j_{\sigma}-\lambda^j_{\sigma}|\leqslant c|\jump {\vU_j}|^2, \quad \jump {\vU_j} = \vU^{j+1}_{\sigma} - \vU^j_{\sigma},
 \end{align}
 where $\{\bar{\vcr}^k_+\}^N_{k=1}$ is the right orthonormal eigenvectors of $\bar{\bm{A}}^{j+\frac{1}{2}}:=\bm{A}(\vU^j_{\sigma}(0))$, with the corresponding eigenvalues $\{\bar{\lambda}^k_+\}^N_{k=1}$.
 
Then it holds
\begin{align}\label{q-es}
q^*_j\leqslant  q_j +\kappa[\lambda^j_{\sigma}]^+ + \tilde{c}|\jump {\vU_j} |^2, \quad  \kappa\geqslant\frac{1}{4},
\end{align}
where 
\[
[\lambda^j_{\sigma}]^+=\max\{ [\lambda^j_{\sigma}],0\}, \quad [\lambda^j_{\sigma}]:=\lambda^j\big{(}\bm{A}(\vU^{j+1}_{\sigma})\big{)}-\lambda^j\big{(}\bm{A}(\vU^j_{\sigma})\big{)},
\]
and $\tilde{c}>\frac{(c_1)^2c_2}{36} +  \frac1{16} c_2c_1^2 $ with $c_1=\max|\frac{\mathrm{d}\vcr^j}{\mathrm{d}\vU}|$,
 $c_2=\max|\rho({\bm{A}}(\vU))|$, $\rho({\bm{A}}(\vU))$ represents the spectral radius of the Jacobian matrix ${\bm{A}}(\vU)$.
\end{Lemma}
\begin{proof}
See Appendix \ref{app-proof}.
\end{proof}

\begin{Remark}
We point out that assumption \eqref{ass-flux} holds for linear systems, since $\bar{\vcr}^j_{\sigma} = \vcr^j_{\sigma},\, \bar{\lambda}^j_{\sigma} = \lambda^j_{\sigma}$ holds for any $j =1,\dots,N$. Moreover, it can be removed if a new FVEG method has a numerical diffusion $q^*_j\leqslant  q_j +\kappa[\lambda^j_{\sigma}]^++ C|\jump{\vU_h}|^2$, we leave the proof to an interested reader.
\end{Remark}

\begin{Remark}
We can extend these results directly to the higher-dimensional case by selecting the numerical flux at the midpoint of the cell boundaries. Specifically, we utilize the one-dimensional numerical flux for both x- and y-directions.
\end{Remark}

Inspired by estimate \eqref{q-es}, we slightly modify our FVEG scheme such that the new version is entropy stable, 
\begin{equation}\label{FEG-1D}
\vF^{EG}_{\sigma} = \avs{\vF(\vU_h)} - \frac12 \sum_{j=1}^N \tilde{q}_j \alpha^j_{\sigma} \vl^j_{\sigma},
\quad 
\tilde{q}_j = \abs{\lambda^j_{\sigma}} + \frac14 [\lambda^j_{\sigma}]^++\tilde{c}|\jump{\vU_j}|^2.
\end{equation}

Lemma \ref{l31} implies that numerical viscosity of the FVEG method with $F_\sigma^{EG}$, cf.\eqref{FEG-1D}, is larger than that of the entropy conservative finite volume method with the numerical flux $F_{\sigma}^{EC}$. Following Tadmor's comparison principle \cite{LME} the FVEG method \eqref{scheme}, \eqref{FEG-1D} is entropy stable.

Indeed, multipling \eqref{scheme-1} with $\omega_h=\omega(\vU_h),~\omega=\bigtriangledown_{\vU}S(\vU)$, we obtain, 
\begin{align}
\omega_K\cdot\frac{d}{dt}\vU_K(t)=\frac{d S(\vU_K)}{dt}=-\omega_K\cdot\frac{|\sigma| }{|K|}  \sum_{\sigma \in \Sigma(K)} \vF_{\sigma}^{EG}(t)\cdot \vn,
\end{align}
resulting to 
\begin{align}
\int_{\Omega}\frac{d S(\vU_h)}{dt}\dx&-
\sum\limits_{K}\omega_K\cdot|K|\cdot\frac{|\sigma| }{|K|}\bigg(\sum_{\sigma \in \Sigma(K)} \left(\vF_{\sigma}^{EG}(t)-\vF_{\sigma}^{EC}(t)\right)\cdot \vn+\sum_{\sigma \in \Sigma(K)} \vF_{\sigma}^{EC}(t)\cdot\vn\bigg)\nonumber \\\nonumber
&=\int_{\Sigma}\left(\vF_{\sigma}^{EG}(t)-\vF_{\sigma}^{EC}(t)\right)\cdot\jump{\omega_h}_\sigma\dsx\\
&\geq \int_{\Sigma}\tilde{C}\jump{\vU_h}_\sigma\jump{\omega_\nu}_\sigma\dsx\geq \underline{C}\int_{\Sigma}\jump{\vU_h}_\sigma^2\dsx,
\end{align}
where $\tilde{C}<0$, $\underline{C}>0$ are constants depending on $\vU_h$.

Entropy stability of the modified FVEG method \eqref{scheme-1}, \eqref{FEG-1D} implies the following weak BV estimate.

\begin{Lemma}[Weak BV estimate]
Let $\vU_h$ be the numerical solution obtained by the FVEG method, under the assumptions \eqref{Assu} and \eqref{ass-flux} hold. Then we have
\begin{align}\label{BVE}
\int^{\tau}_0\int_{\Sigma}\|\jump{\vU_h}_{\sigma}\|^2_2\dS_x\dt
&\leqslant \int^{\tau}_0\int_{\Omega}\frac{d S(\vU_h)}{dt}\dx\dt\leqslant C.
\end{align}
\end{Lemma}

\begin{Remark}
In the above lemma we have showed that if the FVEG method is entropy stable then the BV estimates \eqref{BVE} hold. We note that entropy stability of the FVEG follows from Lemma \ref{l31} for the case that the cell-interface integral is approximated by the midpoint rule, i.e. one-dimensional entropy stability analysis can be directly applied. In a general case when the Simpson rule for the cell-interface integrals is used the analysis is analogous but more tedious.

Following our recent work \cite{LMY2023}, we conjecture here that the modified FVEG scheme based on the approximate evolution operator and an added numerical diffusion, cf. \eqref{FEG-1D}, is entropy stable also for the Simpson rule approximation of the cell-interface integrals, \eqref{FEG}.
\end{Remark}

\section{Consistency}\label{sec_con}
Having obtained the weak BV estimate, we proceed by showing consistency of the FVEG method \eqref{scheme-1}. 

\subsection{Difference between $U^{EG}_{X}$ and $U_L, U_R$}
We start by showing some useful estimates of  $\| \vU^{EG}_{X}-\vU_K\|,~X \in \{A,B,S\},~K \in \{L,R\}$, where the concrete formulas of $U^{EG}_{X}$ are derived in Section~\ref{sec_EG2-wave} for the wave equation system and Section~\ref{sec_EG2-Euler} for the Euler equations.    
Let us introduce the following notations for the later use
\begin{align*}
&\bullet\quad a \lesssim b \quad \mbox{ if } \quad a \leq c\, b \quad \mbox{ with a generic positive constant c},
\\
&\bullet\quad a \approx b \quad \mbox{ if } \quad a \lesssim b \quad \mbox{ and  } \quad b \lesssim a.
\end{align*}

\begin{Lemma}\label{l41} 
Consider the linear wave equation system. 
Let $U^{EG}_{X},~X \in \{A,B,S\}$ be given by \eqref{EG2-w-S} and \eqref{EG2-w-A}. Then it holds for any $\sigma\in{\Sigma}$, $\sigma = L|R$ that
\begin{equation}
\|\vU_K-\vU_{X}^{EG}\|\lesssim\sum\limits_{\hat{\sigma}\in S(\sigma)}\|\jump{\vU_h}_{\hat{\sigma}}\|, \quad X \in \{A,B,S\}, \ K \in \{L,R\}
\end{equation}
with 
\begin{align}
S(\sigma) = \left\{ \hat{\sigma}\in{\Sigma}| \hat{\sigma} \cap \sigma \neq \emptyset \right\}.  
\end{align}

\end{Lemma}

\begin{proof} 

In what follows we analyze $\| \vU^{EG}_{X} -  U_K\|,~X \in \{A,B,S\},~K = L$ case by case. Note that the case $K=R$ shall be done analogously.   
To this goal, let us specify $S(\sigma)$ as follows
\[
S(\sigma)=\{\sigma,\sigma_+,\sigma_-,\sigma_{L^+},\sigma_{L^-},\sigma_{R^+},\sigma_{R^-}\},
\]
where we have denoted by $\sigma_{L^+}$ and $\sigma_{L^-}$ the edge intersecting of the mesh cell $UL$ and $L$, $L$ and $BL$ respectively. Analogous notation holds for  $\sigma_{R^+}$ and $\sigma_{R^-}$, see Figure~\ref{fig:mesh} for illustration.

For the first component $\phi$,  it follows from \eqref{2.15} that
\begin{eqnarray*}
\phi_L-\phi_A&=&\phi_L-\bigg{(}\frac{1}{4}(\phi_R+\phi_{UR}+\phi_{UL}+\phi_L)-\frac{1}{\pi}(u_R+u_{UR}-u_{UL}-u_{L})
+\frac{1}{\pi}(v_{R}-v_{UR}-v_{UL}+v_{L})\bigg{)}\\\nonumber
&=&\bigg{(}\frac{3}{4}\phi_L-\frac{1}{4}(\phi_R+\phi_{UR}+\phi_{UL})\bigg{)}+\frac{1}{\pi}(u_R-u_L+u_{UR}-u_{UL})
-\frac{1}{\pi}(v_R-v_{UR}+v_L-v_{UL})\\\nonumber
&=&\bigg{(}\frac{1}{2}(\phi_L-\phi_R)+\frac{1}{4}(\phi_L-\phi_{UL})+\frac{1}{4}(\phi_L-\phi_{UR})\bigg{)}+\frac{1}{\pi}(u_R-u_L+u_{UR}-u_{UL})\\\nonumber
&&-\frac{1}{\pi}(v_R-v_{UR}+v_L-v_{UL}),
\end{eqnarray*}
which yields 
\[
\abs{\phi_L-\phi_A} \leq \frac12 \Big(\abs{\jump\phi_\sigma}+\abs{\jump\phi_{\sigma_{L^+}}}+ \abs{\jump\phi_{\sigma_{R^+}}}+\abs{\jump u_{\sigma}}+\abs{\jump u_{\sigma_+}}+\abs{\jump v_{\sigma_{R^+}}}+\abs{\jump v_{\sigma_{L^+}}} \Big).
\]
Same calculations give
\begin{eqnarray*}
\abs{\phi_L-\phi_B} \leq \frac12 \Big( \abs{\jump\phi_\sigma}+\abs{\jump\phi_{\sigma_{L^-}}}+ \abs{\jump\phi_{\sigma_{R^-}}}+ \abs{\jump u_{\sigma}}+\abs{\jump u_{\sigma_-}}+\abs{\jump v_{\sigma_{R^-}}}+\abs{\jump v_{\sigma_{L^-}}} \Big).
\end{eqnarray*}
Further, thanks to the definition of $\phi_S$, i.e.\ \eqref{2.12}, we get
\begin{eqnarray*} 
\phi_L-\phi_S&=&\phi_L-\bigg{(}(\frac{1}{2}(\phi_L+\phi_R)-\frac{2}{\pi}(u_R-u_L)\bigg{)}
=\frac{1}{2}(\phi_L-\phi_R)+\frac{2}{\pi}(u_R-u_L)
\end{eqnarray*}  
implying $\abs{\phi_L-\phi_S} \leq \frac{2}{\pi} \big( \abs{\phi_L-\phi_R} + \abs{u_R-u_L}\big)$. Altogether, we have proved $\| \phi^{EG}_{X} -  \phi_K\| \aleq \sum\limits_{\hat{\sigma}\in S(\sigma)}\|\jump{\vU_h}_{\hat{\sigma}}\|$, where for the vector $\vU_h$ the $L^1-$norm is considered.

In the same manner, we obtain from the definitions of $u_A$ and $u_B$, see \eqref{EG2-w}, that 
\begin{eqnarray*}
u_L-u_A&=&u_L-\bigg{(}-\frac{1}{\pi}(\phi_R+\phi_{UR}-\phi_{UL}-\phi_{L})+\frac{1}{4}(u_R+u_{UR}+u_{UL}+u_L)
+\frac{1}{\pi}(-v_{R}+v_{UR}-v_{UL}+v_L)\bigg{)}\\\nonumber
&=&\frac{1}{\pi}(\phi_R-\phi_{L}+\phi_{UR}-\phi_{UL})+\bigg{(}\frac{1}{2}(u_L-u_R)+\frac{1}{4}(u_{L}-u_{UL})+\frac{1}{4}(u_R-u_{UR})\bigg{)}\\\nonumber
&&-\frac{1}{\pi}(-v_{R}+v_{UR}-v_{UL}+v_L),
\end{eqnarray*}
and
\begin{eqnarray}
u_L-u_S&=&u_L-\bigg{(}-\frac{2}{\pi}(\phi_R-\phi_L)+\frac{1}{2}(u_R+u_L)\bigg{)} = \frac{2}{\pi}(\phi_R-\phi_L) +\frac{1}{2}(u_R- u_L).
\end{eqnarray}
Consequently, we obtain 
\begin{align*}
&  |u_L-u_A|\leq \frac12 \bigg{(}|\jump \phi_{\sigma}|+|\jump \phi_{\sigma_+}| + |\jump u_{\sigma}| + |\jump u_{\sigma_{L^+}}|+ |\jump \phi_{\sigma_{R^+}}|+ |\jump v_{\sigma_{R^+}}| + |\jump v_{\sigma_{L^+}}| \bigg{)},\\
& |u_L-u_S| \leq \frac{2}{\pi} \bigg{(} |\jump\phi_{\sigma}| + |\jump u_{\sigma}|  \bigg{)}.
\end{align*}
Similar estimate as for $\abs{u_L-u_A}$ holds for $\abs{u_L-u_B}$, and we obtain $\| u^{EG}_{X} -  u_K\| \aleq \sum\limits_{\hat{\sigma}\in S(\sigma)}\|\jump{\vU_h}_{\hat{\sigma}}\|$. 

Analyzing  $\| v^{EG}_{X} -  v_K\|$ analogously as $\| u^{EG}_{X} -  u_K\|$ finishes the proof.                                                                                                                                                                                                                                                                                                                                                                                                                                                                                                                                                                                                                                                                                                                                                                                                                                                                                                                                                                                                                                                                                                                
\end{proof}

\medskip

\begin{Lemma}\label{l42} 
Consider the Euler equations. 
Let $U^{EG}_{X}, ~X \in \{A,B,S\}$ be given by \eqref{EG2-E-S} and \eqref{EG2-E-A}. Under the assumption \eqref{Assu}, it holds for any $\sigma\in\Sigma$, $\sigma = L|R$ that
\begin{equation}
\|\vU_K-\vU_{X}^{EG}\|\lesssim\sum\limits_{\hat{\sigma}\in S(\sigma)}\|\jump{\vU_h}_{\hat{\sigma}}\|, \quad X \in \{A,B,S\}, \ K \in \{L,R\}.
\end{equation}
\end{Lemma}

\begin{proof}
The proof can be done case-by-case calculations analogously as the proof of Lemma \ref{l41}. Hence, we show here only the detailed calculations for $\rho$. 

With the definition of $\rho_A$ and $\rho_S$, see \eqref{2.37} and \eqref{2.33}, we obtain 
\begin{eqnarray}
\rho_L-\rho_A&=&\frac{p_L}{c^{{\prime}2}}(p_L-p_{UL})-\frac{\alpha_1}{\pi c^{\prime 2}}(p_R-p_{UR}-p_{UL}+p_L)-\frac{\alpha_2}{\pi c^{\prime 2}}(p_R+p_{UR}-p_{UL}-p_L)
\\\nonumber
&&+\frac{\rho^\prime\sin\alpha_1}{\pi c^\prime}(u_R-u_{UR}+u_{UL}-u_L)+\frac{\rho^\prime\sin\alpha_2}{\pi c^\prime}(u_R+u_{UR}-u_{UL}-u_L)\\\nonumber
&&-\frac{\rho^\prime\cos\alpha_1}{\pi c^\prime}(v_R-v_{UR}-v_{UL}+v_L)-\frac{\rho^\prime\cos\alpha_2}{\pi c^\prime}(v_R-v_{UR}+v_{UL}-v_L)).
\end{eqnarray}
Due to the boundedness of numerical solution we obtain
\begin{equation}
|\rho_L-\rho_A|\leq C\sum\limits_{\hat{\sigma}\in S(\sigma)}({\jump p_{\hat\sigma}}+\jump u_{\hat\sigma}+\jump v_{\hat\sigma}).
\end{equation}
Futher, we have
\begin{equation}
|\rho_L-\rho_S|\leqslant c(\jump p_{\sigma}+\jump u_{\sigma})
\end{equation}
which finishes the proof.
\end{proof}

\subsection{Consistency formulation}

\begin{Theorem}[Consistency formulation]  
Let $\vU_h$ be the numerical solution obtained by the FVEG method \eqref{scheme} for the wave system or \eqref{scheme}, \eqref{FEG-1D} for the Euler equations with $\vU_{0,h}=\projection{\vU_0}$. 
For the Euler equations we further assume that the assumption \eqref{Assu} holds.

Then for all $\tau\in(0,T)$ we have 
\begin{equation}
\bigg{[}\int_{\Omega}\vU_h{\bm{\phi}}\dx\bigg{]}^{t=\tau}_{t=0}=\int^\tau_0\int_{\Omega}(\vU_h\partial_t{\bm\phi}+\vF_h:\nabla_x{\bm\phi})\dx\dt+ \bm{e}_{\vU}(\tau,\bm{\phi}), \quad \vF_h=\vF(\vU_h),
\end{equation}
and the error $\bm{e}_{\vU}(\tau,\bm{\phi})$ is bounded for any $\tau\in(0,T)$ as follows
\begin{equation}
\abs{\bm{e}_{\vU}(\tau,\bm{\phi})} \lesssim h^{1/2}\|\bm{\phi}\|_{C^1([0,T]\times\overline\Omega)}\left(\int^T_0\int_{\Sigma}\|\jump\vU_\sigma\|^2_2{\dS}_x\dt \right)^{1/2}.
\end{equation}
 \end{Theorem}

\begin{proof}
Realizing that
\begin{equation}
\jump {ab}=\avs {a}\jump b+\jump a\avs b,
\end{equation}
we can rewrite the convection term   $\int_{\Omega}\vF_h:\nabla_x\bm \phi\dx$ as follows
\begin{eqnarray}
\int_{\Omega}\vF_h:\nabla_x\bm \phi\dx&=&\sum\limits_{K\in \mathcal{T}_h}\int_{K}\vF_h:\nabla_x\bm \phi\dx=\sum\limits_{K\in \mathcal{T}_h}\int_{\partial K}\vF_h\cdot\bm \phi\cdot\bm n_K\dx=-\int_{\Sigma}\jump {\vF_h} \cdot\bm {\phi}\cdot \bm {n}{\dS}_x\\\nonumber
&=&-\int_{\Sigma}{(\jump {\vF_h}\cdot(\bm \phi-\avs{\Pi_h[\bm \phi]})+\jump {\vF_h}\cdot\avs{\Pi_h[\bm \phi]})\cdot\bm n}\dS_x\\\nonumber
&=&-\int_{\Sigma}{(\jump {\vF_h}\cdot(\bm \phi-\avs{\Pi_h[\bm \phi]})-\jump {\Pi_h[\bm \phi]}\cdot\avs{\vF_h}+\jump {\vF_h\cdot\Pi_h[\bm\phi]})\cdot\bm n}\dS_x\\\nonumber
&=&\int_{\Sigma}\vF^{EG}_\sigma\cdot\bm n\cdot \jump {\Pi_h[\bm \phi]}\dS_x-\int_{\Sigma}\jump {\vF_h}\cdot n\cdot(\bm {\phi}-\avs{\Pi_h[\bm \phi]})\dS_x\\\nonumber
&&+\int_{\Sigma}(\avs{\vF_h}-\vF^{EG}_\sigma)\cdot n\cdot\jump {\Pi_h[\bm \phi]}\dS_x.
\end{eqnarray}
Note that for the last equality we have used the Gauss theorem and the periodic boundary conditions, i.e.
$$\int_{\Sigma}\jump{ \vF_h\cdot\Pi_h[\bm\phi]} \cdot\bm n\dS_x=0.$$
Hence, let us denote the error terms as follows
\begin{align*}
e_1(t,\bm{\phi})&=\int_{\Sigma}\jump {\vF_h(t)}\cdot n\cdot(\bm\phi-\avs{\Pi_h[\bm \phi]})\dS_x, \\
e_2(t,\bm{\phi})&=-\int_{\Sigma}(\avs{\vF_h(t)}-\vF^{EG}_\sigma(t))\cdot n\cdot\jump {\Pi_h[\bm \phi]}\dS_x.
\end{align*}
Thanks to the Lipschitz-continuity of $\vF$, i.e. $\|\jump{\vF_h}\cdot \vn\|\lesssim\|\jump{\vU_h}\|$, and due to the projection error estimate $\|\bm\phi-\avs{\Pi_h[\bm \phi]}\|\lesssim h\|\bm \phi\|_{C^1(\overline\Omega)}$ for all $x\in\sigma\in\Sigma$, we shall control $e_1$:
\begin{eqnarray}
|e_1(t,\bm{\phi})|&\lesssim &h\|\bm\phi\|_{C([0,T]\times\overline\Omega)}\int_{\Sigma}\| \jump{\vF_h(t)}\cdot \bm n\|\dS_x\lesssim h\|\phi\|_{C([0,T]\times\overline\Omega)}\int_{\Sigma}\|\jump{\vU_h(t)}\|\dS_x\\\nonumber
&\lesssim& h\|\bm\phi\|_{C([0,T]\times\overline\Omega)}\left(\int_{\Sigma}\|\jump{\vU_h(t)}\|^2\dS_x \right)^{1/2}\left(\int_{\Sigma}1\dS_x\right)^{1/2}\\\nonumber
&\lesssim& h^{1/2}\|\bm\phi\|_{C([0,T]\times\overline\Omega)}\left(\int_{\Sigma}\|\jump{\vU_h(t)}\|^2\dS_x\right)^{1/2}.
\end{eqnarray}
On the other hand, due to the projection error estimate   $\jump {{\Pi}_h[\bm \phi]}\lesssim h\|\bm \phi\|_{C^1(\overline\Omega)}$, we obtain
\begin{eqnarray}
|e_2(t,\bm{\phi})|&\lesssim& h\|\bm \phi\|_{C^1([0,T]\times\overline\Omega)}\int_{\Sigma}\|(\avs{\vF_h(t)}-\vF^{EG}_\sigma(t))\cdot n\|\dS_x\\\nonumber
&\lesssim& h\|\bm \phi\|_{C^1([0,T]\times\overline\Omega)}\sum\limits_{\sigma:=L|R\in\Sigma}\int_{\sigma}(\|(\vF(\vU_L(t))-\vF(\vU^{EG}_X(t))\cdot n\|+\|(\|\vF(\vU_R(t))-\vF(\vU^{EG}_X(t))\cdot n\|)\dS_x\\\nonumber
&\lesssim& h\|\bm \phi\|_{C^1([0,T]\times\overline\Omega)}\sum\limits_{\sigma:=L|R\in\Sigma}\int_{\sigma}(\|\vU_L(t)-\vU^{EG}_\sigma(t)\|+\|\vU_R(t)-\vU^{EG}_\sigma(t)\|)\dS_x\\\nonumber
&\lesssim&h\|\bm \phi\|_{C^1([0,T]\times\overline\Omega)}\int_{\Sigma}\|\jump{ \vU_h(t)}\|\dS_x\lesssim h^{1/2}\|\bm \phi\|_{C^1([0,T]\times\overline\Omega)} \left( \int_{\Sigma}\|\jump {\vU_h(t)}\|^2\dS_x \right)^{1/2},
\end{eqnarray}
where we have applied Lemma \ref{l31} and Lemma \ref{l42}.

Finally, we have proved the consistency of the FVEG method, i.e.
\begin{eqnarray*}
\left[\int_{\Omega}\vU_h{\bm{\phi}}\dx \right]^{t=\tau}_{t=0}&=&\int^\tau_0\int_\Omega\frac{d}{\dt}(\vU_n\cdot\bm\phi)\dx=\int^\tau_0\int_\Omega(\vU_h\cdot\partial_t\bm \phi+\bm\phi\cdot\frac{d}{\dt}\vU_h)\dx\dt\\
&=&\int^\tau_0\int_\Omega\vU_h\cdot\partial_t\bm\phi\dx\dt+\int^\tau_0\int_{\Sigma}\vF(\vU^{EG}_\sigma)\cdot \bm n\jump\phi\dS_x\\
&=&\int^\tau_0\int_\Omega(\vU_h\cdot\partial_t\bm\phi+\vF_h:\nabla_x\phi)\dx\dt+e_{\vU}(\tau,\bm{\phi}),
\end{eqnarray*}
and $e_{\vU}(\tau,\bm{\phi}) = \int^\tau_0(e_1+e_2)(t,\bm\phi)\dt$.
\end{proof}

\section{Convergence}
The aim of this section is to prove that the numerical solution $\{\vU_h\}_{h\searrow 0}$ generated by the FVEG method \eqref{scheme}, \eqref{FEG-1D} converges as $h\to 0$. Due to the lack of compactness of $\{\vU_h\}_{h\searrow 0}$ for the Euler equations, we can only show the convergence to  a generalized, the so-called, dissipative weak solution, which we will define in what follows.

As we will show later, in the case of linear wave equation system, our results directly imply the convergence of the FVEG method to a weak solution.

\begin{Definition}\label{d4.1}
{\bf{(Dissipative weak (DW) solution)}} Let the initial data for the Euler equations satisfy
$$
\begin{aligned}
	& \varrho_0 \in L^\gamma\left(\Omega\right), \boldsymbol{m}_0 \in L^{\frac{2 \gamma}{\gamma+1}}\left(\Omega ; \mathbb{R}^d\right), S_0 \in L^\gamma\left(\Omega\right), \\
	& E_0=E\left(\varrho_0, \boldsymbol{m}_0, S_0\right) \text { and } \int_{\Omega} E\left(\varrho_0, \boldsymbol{m}_0, S_0\right) \dx<\infty .
\end{aligned}
$$

We say that $(\varrho, \boldsymbol{m}, S)$ is a dissipative weak solution to the Euler equations in $[0, T) \times\Omega$, $0<T \leq \infty$, if the following holds:
\begin{itemize}
	\item {\bf{Regularity.}} The solution $(\varrho, \boldsymbol{m}, S)$ belongs to the class 
	$$
	\begin{aligned}
		& \varrho \in C_{\text {weak,loc }}\left([0, T) ; L^\gamma\left(\Omega\right)\right), \boldsymbol{m} \in C_{\text {weak,loc }}\left([0, T) ; L^{\frac{2 \gamma}{\gamma+1}}\left(\Omega; \mathbb{R}^d\right)\right), \\
		& S \in L^{\infty}\left(0, T ; L^\gamma\left(\Omega\right)\right) \cap B V_{\text {weak }}\left([0, T) ; L^\gamma\left(\Omega\right)\right) ,\\
		& \quad \int_{\Omega} E(\varrho, \boldsymbol{m}, S)(t, \cdot) \mathrm{d} x \leq \int_{\Omega} E\left(\varrho_0, \boldsymbol{m}_0, S_0\right)\dx\text { for any } 0 \leq t<T .
	\end{aligned}
	$$
	\item  {\bf{Equation of continuity.}} The integral identity
	$$
	\int_0^T \int_{\Omega}\left[\varrho \partial_t \varphi+\boldsymbol{m} \cdot \nabla_x \varphi\right] \mathrm{d} x \mathrm{~d} t=-\int_{\Omega} \varrho_0 \varphi(0, \cdot) \mathrm{d} x,
	$$
	holds for any $\varphi \in C^1_c\left([0, T) \times\Omega\right)$.
	\item  {\bf{Momentum equation.}} The integral identity
	$$
	\begin{gathered}
		\int_0^T \int_{\Omega}\left[\boldsymbol{m} \cdot \partial_t \varphi+\mathbb{1}_{\varrho>0} \frac{\boldsymbol{m} \otimes \boldsymbol{m}}{\varrho}: \nabla_x \varphi+p(\varrho, S) \mathrm{div}_x \varphi\right] \mathrm{d} x \mathrm{d} t \\
		\quad=-\int_0^T \int_{\Omega} \nabla_x \varphi: \mathrm{d} \Re(t) \mathrm{d} x-\int_{\Omega} \boldsymbol{m}_0 \cdot \varphi(0, \cdot) \mathrm{d} x
	\end{gathered}
	$$
	holds for any $\varphi \in C^1_c\left([0, T) \times \Omega; \mathbb{R}^d\right)$, where the Reynolds stress defect  reads as
	$$
	\Re \in L^{\infty}\left(0, T ; \mathcal{M}^{+}\left(\Omega ; \mathbb{R}_{\mathrm{sym}}^{d \times d}\right)\right) .
	$$
	\item {\bf{Entropy inequality.}}
	$$
	\begin{array}{r}
		{\left[\int_{\Omega} S \varphi \mathrm{d} x\right]_{t=\tau_1-}^{t=\tau_2+} \geq \int_{\tau_1}^{\tau_2} \int_{\Omega}\left[S \partial_t \varphi+\left\langle\mathcal{V}_{t, x} ; 1_{\tilde{\varrho}>0}(\tilde{S} \tilde{\boldsymbol{u}})\right\rangle \cdot \nabla_x \varphi\right] \mathrm{d} x \mathrm{d} t,} \\
		S(0-, \cdot)=S_0,
	\end{array}
	$$
	for any $0 \leq \tau_1 \leq \tau_2<T$, any $\varphi \in C^1\left([0, T) \times \Omega\right), \varphi \geq 0$, where $\left\{\mathcal{V}_{t, x}\right\}_{(t, x) \in(0, T) \times\Omega}$ is a parametrized probability (Young) measure,
	$$
	\begin{gathered}
		\left.\mathcal{V}_{t, x} \in L^{\infty}\left((0, T) \times \Omega\right) ; \mathcal{P}\left(\mathbb{R}^{d+2}\right)\right), \mathbb{R}^{d+2}=\left\{\tilde{\varrho} \in \mathbb{R}, \widetilde{\boldsymbol{m}} \in \mathbb{R}^d, \widetilde{S} \in \mathbb{R}\right\},\\
		\langle\mathcal{V} ; \tilde{\varrho}\rangle=\varrho,\langle\mathcal{V} ; \widetilde{\boldsymbol{m}}\rangle=\boldsymbol{m},\langle\mathcal{V} ; \widetilde{S}\rangle=S .
	\end{gathered}
	$$
	\item {\bf{Compatibility of the energy and Reynolds stress defects.}} There exists a nonincreasing function $\mathcal{E}:[0, T) \rightarrow[0, \infty)$ satisfying
	$$
	\begin{aligned}
		& \mathcal{E}(0-)=\int_{\Omega} E\left(\varrho_0, \boldsymbol{m}_0, S_0\right) \mathrm{d} x, \\
		& \mathcal{E}(\tau+)=\int_{\Omega} E(\varrho, \boldsymbol{m}, S)(\tau, \cdot) \mathrm{d} x+\mathfrak{F}, \quad \text { for any } 0 \leq \tau<T,
	\end{aligned}
	$$
	where $\mathfrak{F} \in L^{\infty}\left(0, T ; \mathcal{M}^{+}\left(\Omega\right)\right)$ is the energy defect satisfying
	$$
	\min \{2, d(\gamma-1)\} \mathfrak{F} \leq \operatorname{trace}[\mathfrak{R}] \leq \max \{2, d(\gamma-1)\} \mathfrak{F}.
	$$
\end{itemize}

\end{Definition}

Following \cite[Proposition 7.2]{FLMS}, we obtain the following results on weak and strong convergence of the FVEG method.

\begin{Theorem}\label{weakC}
{\bf{(Weak convergence) }}Let the initial data $\left\{\varrho_{0, h}, \boldsymbol{m}_{0, h}, E_{0, h}\right\}_{h\searrow 0}$ satisfy
$$
\varrho_{0, h} \geq \underline{\varrho}>0, \quad E_{0, h}-\frac{1}{2} \frac{\left|\boldsymbol{m}_{0, h}\right|^2}{\varrho_{0, h}}>0, \quad h\to 0.
$$

We note that a suitable choice of the discrete initial data is  $\rho_{h,0}=\projection{\rho_0}$, $m_{h,0}=\projection{m_0}$ and $E_{h,0}=\projection{E_0}$. Let $\left\{\varrho_h, \boldsymbol{m}_h, S_h\right\}$ be a solution obtained by the FVEG method. Here $S_h$ is computed from $\rho_h$ and $E_h$. Further, suppose that assumption \eqref{Assu} holds.
Then up to a subsequence, as the case may be, $\left\{\varrho_h, \boldsymbol{m}_h, S_h\right\}$ generates a DW solution $(\varrho, \boldsymbol{m}, S)$ in the sense of Definition \ref{d4.1}
$$
\left(\varrho_h, \boldsymbol{m}_h, S_h\right) \rightarrow(\varrho, \boldsymbol{m}, S) \text { weakly-(*) in } L^{\infty}\left((0, T) \times\Omega; \mathbb{R}^{d+2}\right), \quad \text { as } h \rightarrow 0 .
$$

Moreover,
$$
E\left(\varrho_h, \boldsymbol{m}_h, S_h\right) \rightarrow\left\langle\mathcal{V}_{t, x} ; E(\tilde{\varrho}, \widetilde{\boldsymbol{m}}, \widetilde{S})\right\rangle \text { weakly-(*)~in } L^{\infty}\left((0, T) \times\Omega\right), h \rightarrow 0.
$$
\end{Theorem}
We note that due to \eqref{Assu} the Reynolds stress concentration defect and energy concentration defect vanish.

\begin{Theorem}\label{strogC}
{\bf{ (Strong convergence)}}
 Let the sequence $\left\{\varrho_h, m_h, S_h\right\}_{h\searrow 0}$ be a solution obtained by the FVEG method. Let the assumptions of Theorem \ref{weakC} hold. Then up to a subsequence $\left\{\varrho_h, m_h, S_h\right\}_{h\searrow 0}$ converges strongly to a DW solution $(\varrho, \boldsymbol{m}, S)$ in the following sense.
\begin{itemize}
	\item {\bf{Strong convergences of Cesàro averages}}
\end{itemize}
$$
\begin{aligned}
	& \frac{1}{N} \sum_{k=1}^N\left(\varrho_{h_k}, \boldsymbol{m}_{h_k}, S_{h_k}\right) \rightarrow(\varrho, \boldsymbol{m}, S),\quad
 \text { as } N \rightarrow \infty \text { in } L^q\left((0, T) \times\Omega ; \mathbb{R}^{d+2}\right) \text { for any } 1 \leq q<\infty \text {, } \\
	& \frac{1}{N} \sum_{k=1}^N E\left(\varrho_{h_k}, \boldsymbol{m}_{h_k}, S_{h_k}\right) \rightarrow\left\langle\mathcal{V}_{t, x}, E(\tilde{\varrho}, \widetilde{\boldsymbol{m}}, \widetilde{S})\right\rangle, \quad
	\text { as } N \rightarrow \infty \text { in } L^q\left((0, T) \times\Omega\right) \text { for any } 1 \leq q<\infty.\\
	&
\end{aligned}
$$

\begin{itemize}
\item {\bf{Weak solution}}
If $(\rho, \boldsymbol{m}, S)$ is a weak entropy solution of the Euler equations with initial data $\left(\rho_0, \boldsymbol{m}_0, S_0\right)$, then
$$
\mathcal{V}_{t, x}=\delta_{(\rho(t, x), m(t, x), S(t, x))}, \quad \text { for a.a. }(t, \boldsymbol{x}) \in(0, T) \times \Omega,
$$
and the strong convergence holds, that is,
$$
\begin{array}{lr}
	\left(\rho_{h_n}, \boldsymbol{m}_{h_n}, S_{h_n}\right) \rightarrow(\rho, \boldsymbol{m}, S) & \text { in }~~L^q\left((0, T) \times \Omega ; \mathbb{R}^{d+2}\right), \\
	E\left(\rho_{h_n}, \boldsymbol{m}_{h_n}, S_{h_n}\right) \rightarrow E(\rho, \boldsymbol{m}, S) & \text { in }~~L^q((0, T) \times \Omega),
\end{array}
$$
for any $1 \leq q<\infty$.

\item {\bf{Strong solution}}
Suppose that the Euler equations admit a strong solution $(\rho, \boldsymbol{m}, S)$ in the class
$$
\rho, S \in W^{1, \infty}((0, T) \times \Omega), ~\boldsymbol{m} \in W^{1, \infty}\left((0, T) \times \Omega ; \mathbb{R}^d\right), ~\rho \geq \underline{\rho}>0 \text { in }~[0, T) \times \Omega,
$$
emanating from the initial data $\left(\rho_0, \boldsymbol{m}_0, S_0\right)$. Then it holds
$$
\left(\rho_h, \boldsymbol{m}_h, S_h\right) \rightarrow(\rho, \boldsymbol{m}, S) \quad \text { in } L^q\left((0, T) \times \Omega ; \mathbb{R}^{d+2}\right),
$$
$$
E\left(\rho_h, \boldsymbol{m}_h, S_h\right) \rightarrow E(\rho, \boldsymbol{m}, S) \quad \text { in } L^q((0, T) \times \Omega),
$$
for any $1 \leq q<\infty$.
\end{itemize}

\end{Theorem}

\begin{Remark}
Since the wave equation system is linear, the weak convergence of numerical sequence $\{\phi_h, u_h, v_h\}_{h\searrow 0}$ directly implies the convergence to a weak solution in Theorem \ref{weakC}. Due to the weak-strong uniqueness principle \cite[Theorem 6.2]{FLMS} . We obtain strong convergence $$\phi_h\to \phi, u_h\to u, v_h\to v \text { in }L^q((0,T); L^2(\Omega)),\quad q\in[1,\infty),$$ if the strong solution $(\phi,u,v)$ exists on $[0,T]$.
\end{Remark}

\section{Numerical results}
Here we will present several numerical experiments for the linear wave equation system as well as for the nonlinear Euler equations to illustrate behaviour of the FVEG method.
\begin{Example}\label{e1}
Let $\Omega=[-1,1]\times[-1,1]$, consider the following initial data for the wave equation system, 
\begin{align}
& \phi(x,y,0)=-\frac{1}{c}(\sin(2\pi x)+\sin(2\pi y)),\\\nonumber
&\vu(x,y,0)=(0,0),\quad \vu=(u,v).
\end{align}
The speed of sound is set to $c=1$ and periodic boundary conditions are imposed. It can be easily verified that the exact solution has the form
\begin{align}
&\phi(x,y,t)=-\frac{1}{c}\cos(2\pi ct)(\sin(2\pi x)+\sin(2\pi y)),\\\nonumber
&\vu(x,y,t)=\left(-\frac{1}{c}(\sin(2\pi ct)\cos(2\pi x)),\frac{1}{c}(\sin(2\pi ct)\cos(2\pi y))\right).
\end{align}
\end{Example}

The structure of the solution obtained by FVEG method \eqref{scheme} is illustrated in Figure~\ref{fig:three_images}, which is consistent with the results presented in \cite{LMW:2000}. We compute the error in the $L^1$-norm and its experimental order of convergence (EOC) at time $T = 0.1 $ for the grid sizes $1/h = 20, 40, 80, 160$ and $320$. Based on the results shown in Figure~\ref{fig3}, the FVEG method achieves the first-order convergence rate, see Table~\ref{tab01}.
%

\begin{table}[ht]
        \centering
       \begin{tabular}{c c  c c c c c}
\hline
$1/h$& $\phi$ & EOC & $u$ & EOC &$v$ & EOC \\
\hline
 20& 2.25e-01 & ~&1.33e-01 &~&1.33e-01&~\\
\hline
 40&  1.18e-01& 0.9311& 6.64e-02&1.0079&6.64e-02& 1.0079\\
\hline
 80&  6.09e-02 & 0.9567&3.18e-02& 1.0590&3.18e-02&1.0590\\
\hline
 160& 3.03e-02  &  1.0067&1.66e-02&0.9436&1.66e-02&0.9436\\
 \hline
320& 1.52e-02&0.9905 & 8.27e-03&1.0009&8.27e-03&1.0009\\
\hline
\end{tabular}
\caption{Errors and EOC (experimental order of convergence) with the initial value presented in Example~\ref{e1} in $L^1$-norm with $CFL=0.267$ at $T=0.1$.}
\label{tab01}
\end{table}

\begin{figure}
      \centering
        \includegraphics[width=0.45\textwidth]{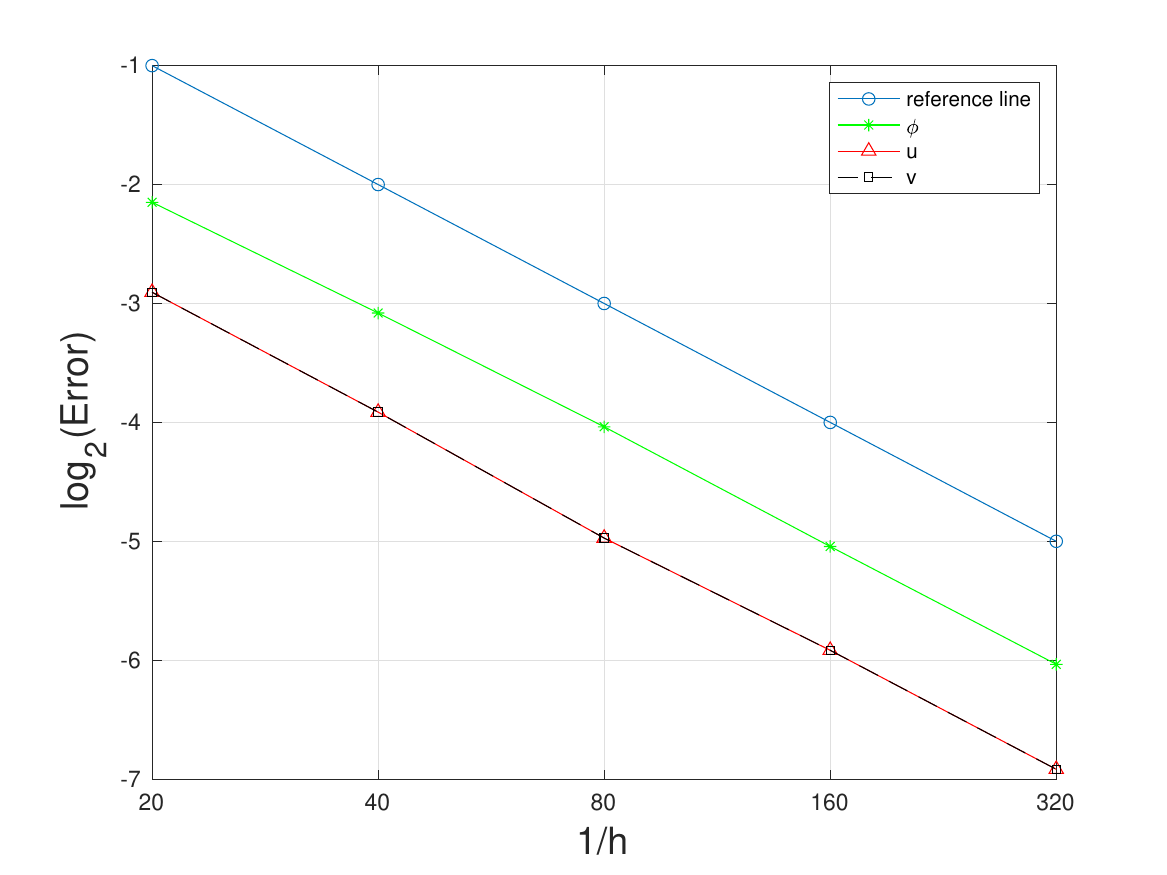} 
        \caption{Example~\ref{e1}: first-order convergence rates.}
        \label{fig3}
\end{figure}

\begin{figure}[h]
    \centering
    \begin{subfigure}[b]{0.3\textwidth}
        \centering
        \includegraphics[width=2.5in]{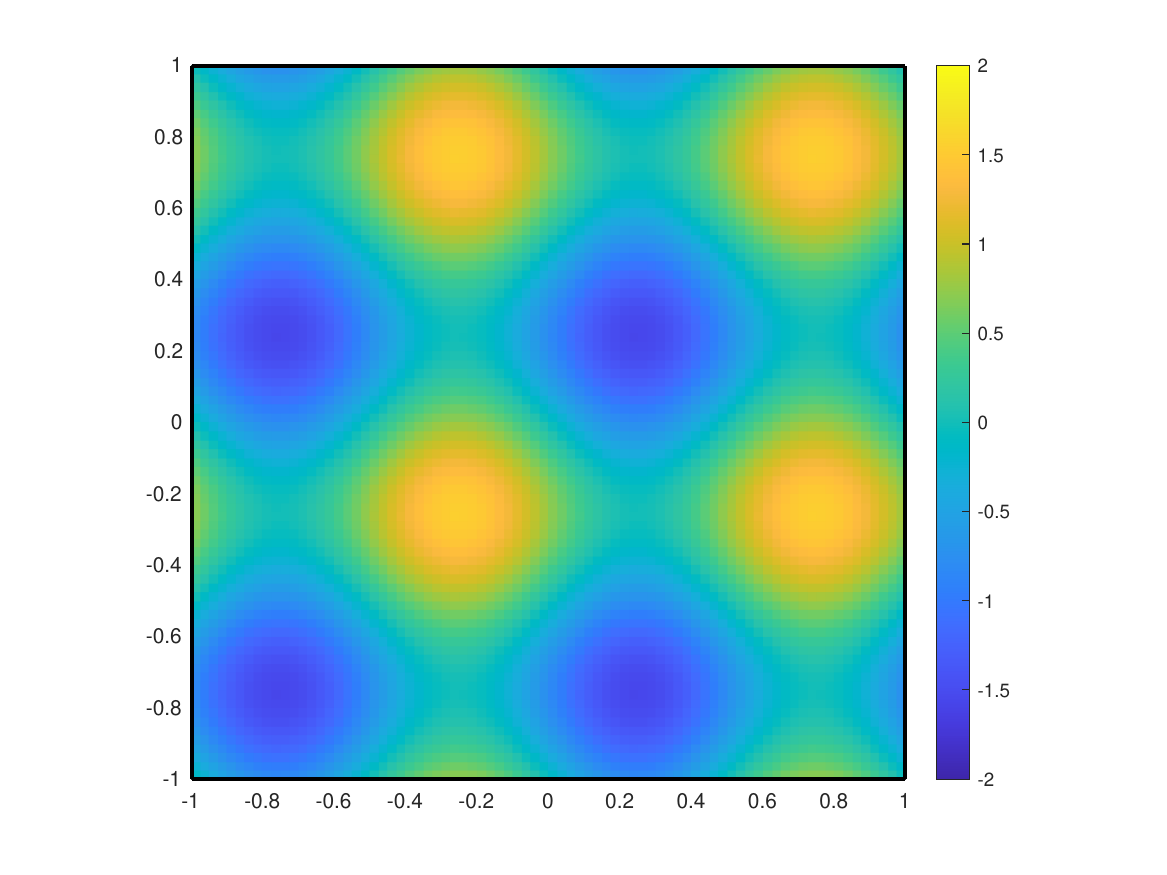} 
    \end{subfigure}
    \hfill
    \begin{subfigure}[b]{0.3\textwidth}
        \centering
        \includegraphics[width=2.5in]{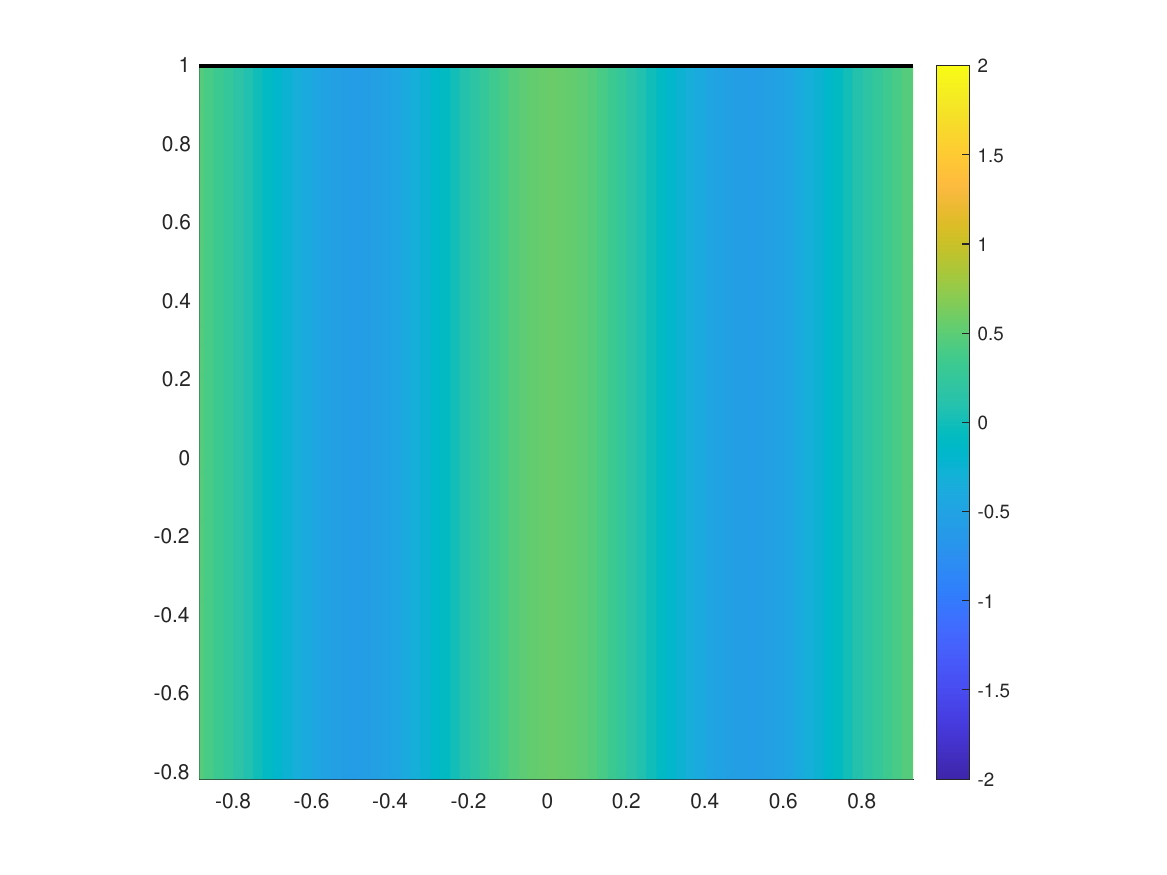} 
    \end{subfigure}
    \hfill
    \begin{subfigure}[b]{0.3\textwidth}
        \centering
        \includegraphics[width=2.5in]{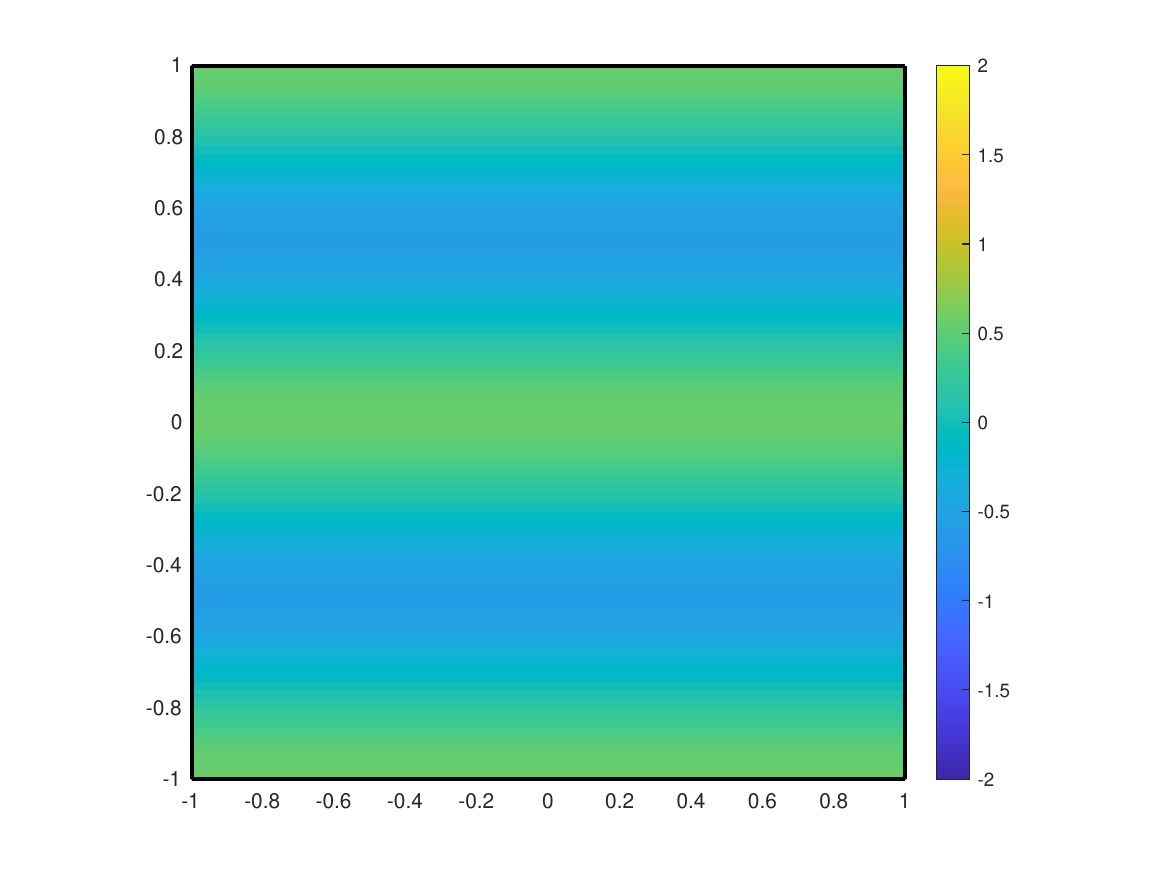} 
    \end{subfigure}
    \caption{Solutions of Example~\ref{e1} at $T=0.1$ and $1/h=80$: $\phi$ (left), $u$ (middle), $v$ (right).}
    \label{fig:three_images}
\end{figure}

\begin{Example}
The Gresho problem is a stationary rotating vortex problem for the Euler equation with the initial data:
\begin{eqnarray*}
& &\rho(r,0)=1, \quad  \bm{u}(r,0)=\bm{n}\left\{
\begin{aligned}
& 5r,&  &0 \leqslant r\leqslant 0.2, \\
&2-5r,&&0.2 \leqslant r\leqslant 0.4, \\
&0, &&r>0.4,
\end{aligned}
\right.\\
& &p(r,0)=\left\{
\begin{aligned}
&5+12.5r^2,& & 0\leqslant r\leqslant 0.2, \\
&9-4\ln0.2+12.5r^2-20r+4\ln r,&  & 0.2\leqslant r<0.4,\\
&3+4\ln2,&&r\geqslant 0.4.
\end{aligned}
\right.\\ 
\end{eqnarray*}
Here $r=\sqrt{x^2+y^2}$, $\bm{n}=(-\sin\theta,\cos\theta)^T$, $\theta\in[0,2\pi]$ and $\bm{u}=(u,v)^T$. 
We choose the computational domain $\Omega=[-0.75,0.75]\times[-0.75,0.75]$ and periodic boundary conditions.
\end{Example}

Figure~\ref{Gresho} presents the numerical results generated by the FVEG method to solve the Euler equations on a $512 \times 512$ cell grid at time $T = 1$. The CFL number used was chosen according to the linear stability analysis, cf. \cite{EHL}. The results reveal that the density does not remain constant throughout the simulation. Instead, it fluctuates, with observed values ranging from $0.9968$ to $1.0165$.

\begin{figure}[h]
    \centering
    \begin{subfigure}[b]{0.3\textwidth}
        \centering
        \includegraphics[width=2.5in]{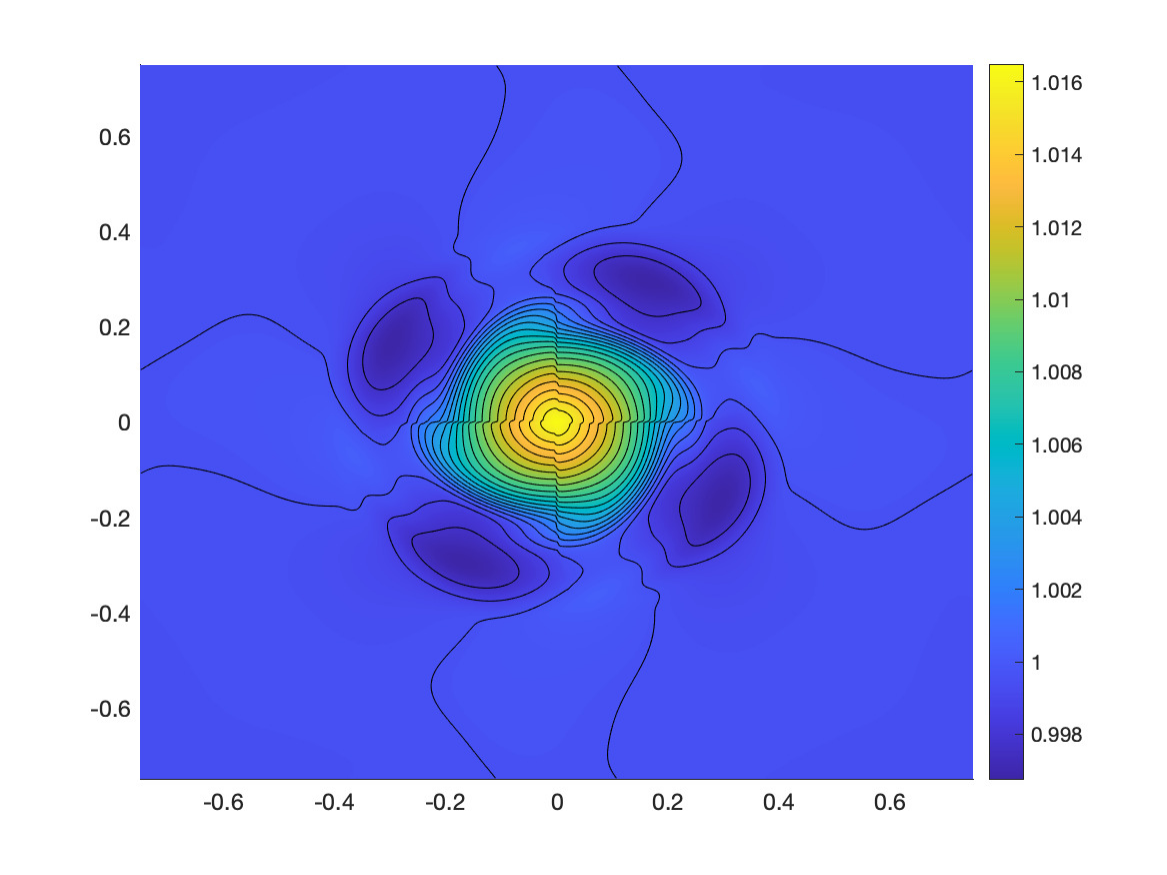} 
    \end{subfigure}
    \hfill
    \begin{subfigure}[b]{0.3\textwidth}
        \centering
        \includegraphics[width=2.5in]{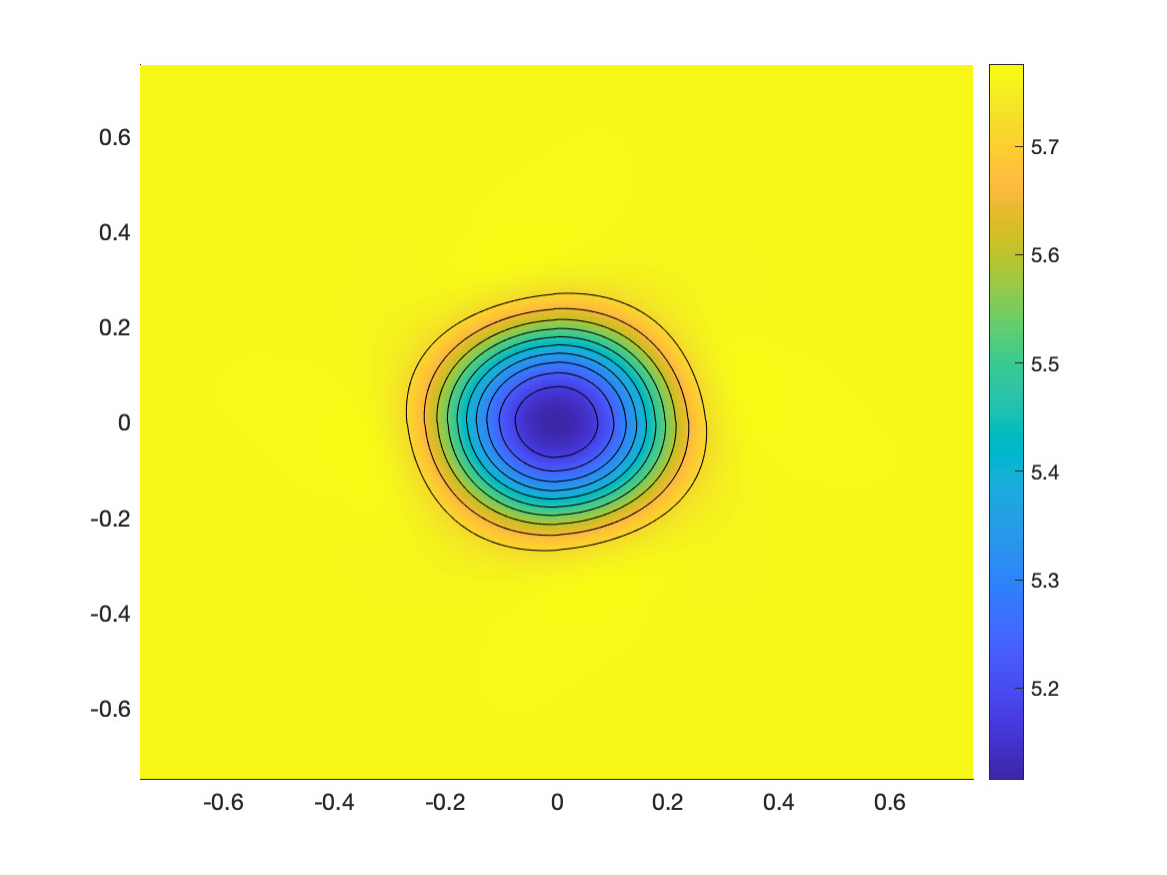} 
    \end{subfigure}
    \hfill
    \begin{subfigure}[b]{0.3\textwidth}
        \centering
        \includegraphics[width=2.5in]{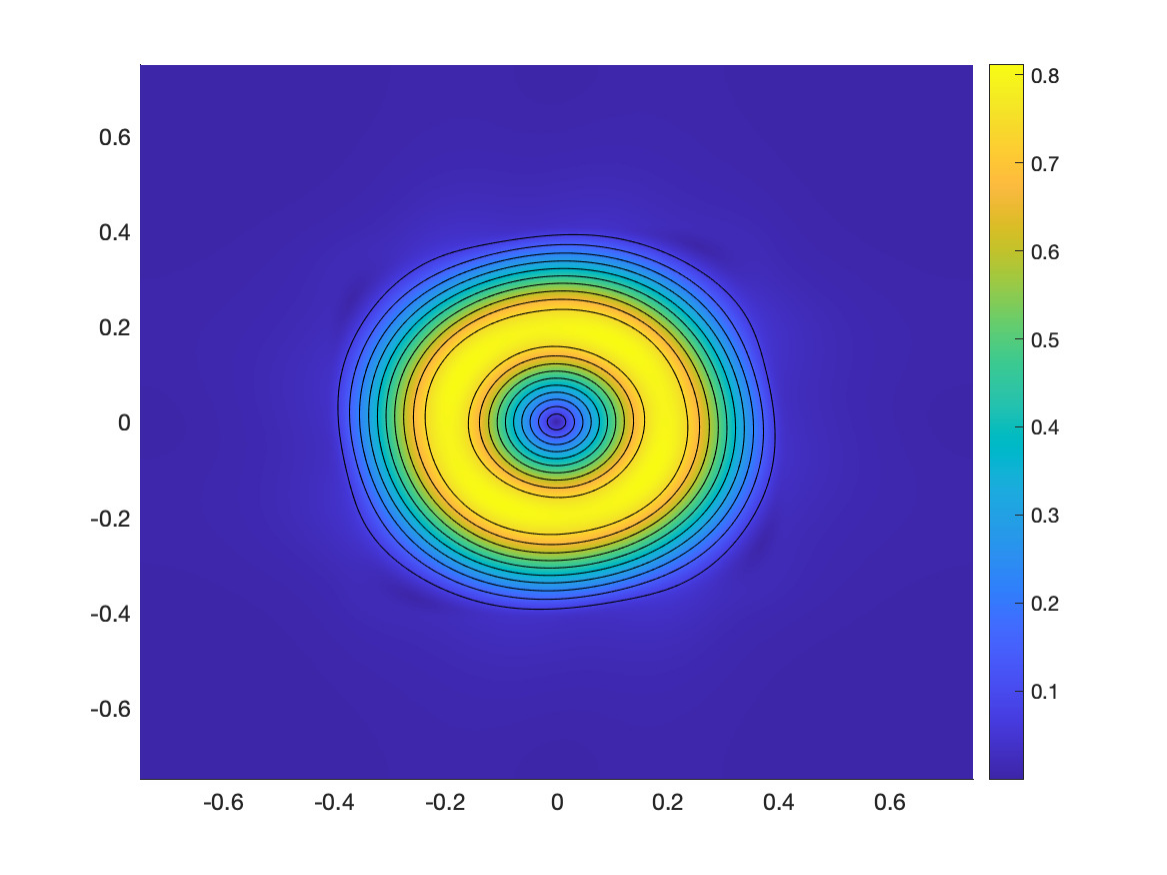} 
    \end{subfigure}
    \caption{Solutions of Gresho problem computed at $T=1$ on a grid with $512\times512$: $\rho$ (left), $p$ (middle), $|\vu|$ (right).}
    \label{Gresho}
\end{figure}

\begin{Example}\label{e3}
We consider a smooth traveling vortex for the Euler equations, a rotating vortex is initially located at $(0.5,0.5)$ and propagates with constant speed $(u_c,v_c)$. The initial data are:
\begin{eqnarray*}
& &\rho(x,y,t)=\left\{
\begin{aligned}
&\rho_c+\frac{1}{2}(1-r^2)^6, & & r<1, \\
&\rho_c,& & \text{otherwise},
\end{aligned}
\right.\\
& &u(x,y,t)=\left\{
\begin{aligned}
&u_c-1024\sin(\theta)(1-r)^6r^6,& & r<1, \\
&u_c,& & \text{otherwise},
\end{aligned}
\right.\\
& &v(x,y,t)=\left\{
\begin{aligned}
&v_c+1024\cos(\theta)(1-r)^6r^6,& & r<1, \\
&v_c,& & \text{otherwise},
\end{aligned}
\right.\\
& &p(x,y,t)=\left\{
\begin{aligned}
&p_c+(p(r)-p(1)),&  & r<1, \\
&p_c,& & \text{otherwise}.
\end{aligned}
\right.\\
\end{eqnarray*}
Here $r$ is the scaled distance from the initial center of the vortex, i.e.$r=\sqrt{(x-0.5)^2+(y-0.5)^2}/R$, where $R$ is the radius of the vortex.The function $p(r)$ is described in \cite{EHL} .
In our computation we use $R = 0.4$, $\rho_c = 0.5$, $u_c = v_c = 1$ and $p_c = 0.1$. We simulate the vortex on the domain $[0, 1]\times[0, 1]$ using periodic boundary conditions. At $T = 1$ the exact solution agrees with the initial values.
\end{Example}

Table~\ref{t1} lists the error and convergence rate of traveling vortex problem computed by the FVEG scheme. Figure~\ref{Travel Vortex128} show the solution structure of the problem. At $T = 1$, as depicted in the bottom line of the figure, there are some changes in numerical solution due to numerical diffusion. These perturbations are visibly pronounced, indicating that the initially smooth geometry has been affected by the vortex dynamics over time. Naturally, by opting for a finer grid resolution, the vortex structure is preserved more accurately, see Figures~\ref{Travel Vortex256}, \ref{Travel Vortex512} with $256\times256$ and $512\times512$ mesh resolution.

\begin{table}[ht]
    \centering
    \begin{tabular}{c c c c c c c c c}
        \hline
        $1/h$ & $\rho$ & EOC & $\rho u$& EOC & $\rho v$ & EOC & $E$ & EOC\\
        \hline
        32 & 2.93e-02 &  ~&3.39e-02 & ~&3.48e-02 & ~& 3.54e-02 & ~\\
        \hline
         64&   2.05e-02& 0.5173& 2.45e-02 & 0.4706 &2.49e-02 &0.4763&2.69e-02&0.3934\\
         \hline
       128& 1.28e-02 &  0.6814&1.58e-02& 0.6304& 1.61e-02&  0.6306&1.79e-02& 0.5852\\
       \hline
       256& 7.13e-03 & 0.8419&9.28e-03 &0.7706&9.42e-03&0.7770&1.11e-02 &  0.6956 \\
        \hline
        512 & 3.72e-03 & 0.9396&5.16e-03&0.8461& 5.13e-03&0.8782&6.42e-03& 0.7868\\
        \hline
    \end{tabular}
     \caption{Error in the $L^1$-norm and EOC for the density, momentum and energy at $T = 1$ for Example~\ref{e3}. }
    \label{t1}
\end{table}

\begin{figure}[h]
    \centering
    \begin{subfigure}[b]{0.3\textwidth}
        \centering
        \includegraphics[width=2.5in]{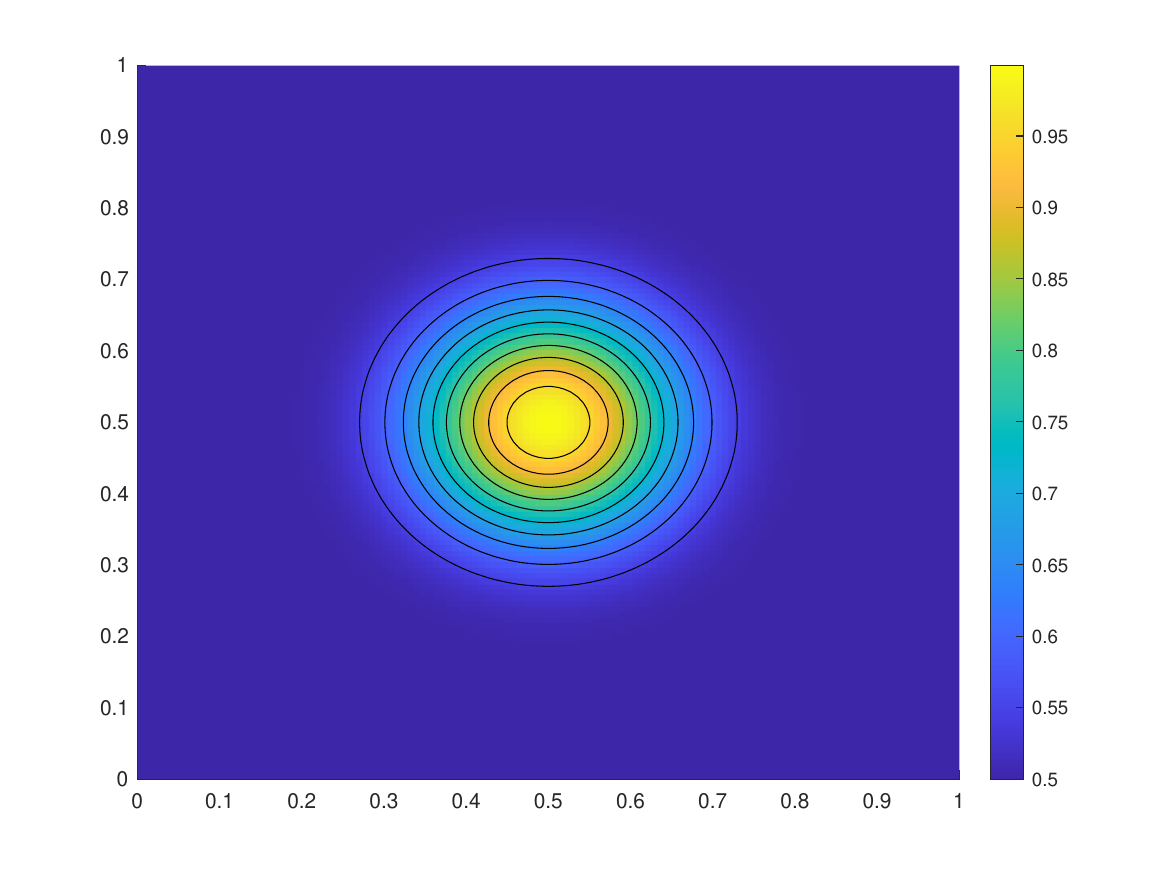} 
    \end{subfigure}
    \hfill
    \begin{subfigure}[b]{0.3\textwidth}
        \centering
        \includegraphics[width=2.5in]{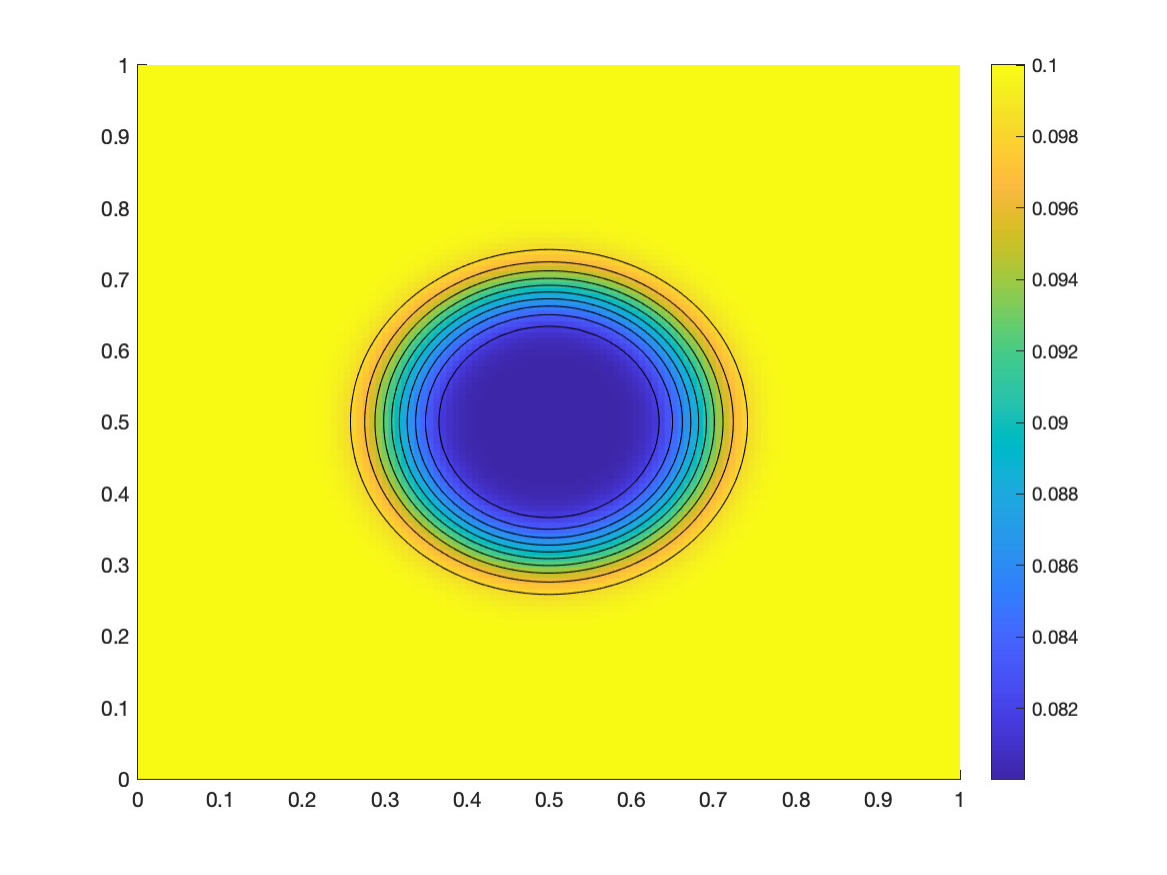} 
    \end{subfigure}
    \hfill
    \begin{subfigure}[b]{0.3\textwidth}
        \centering
        \includegraphics[width=2.5in]{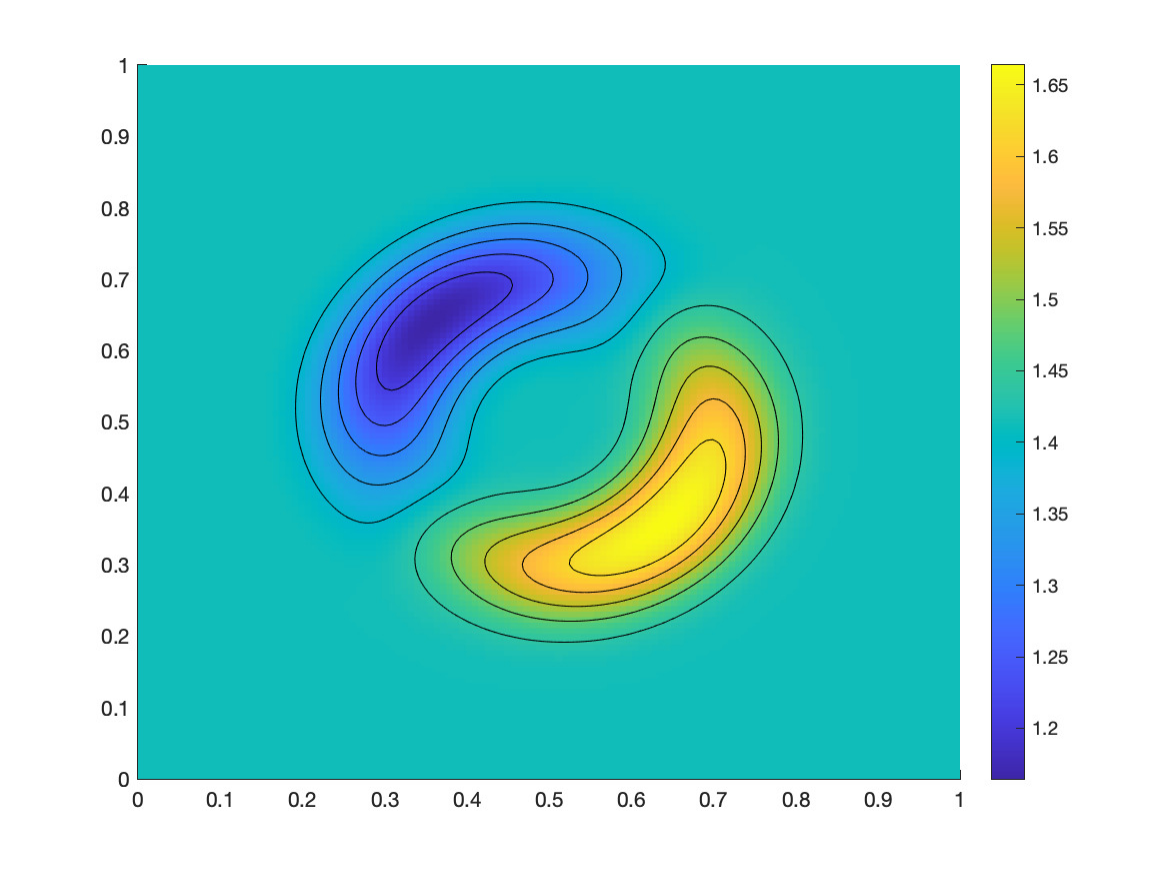} 
    \end{subfigure}
    
        \begin{subfigure}[b]{0.3\textwidth}
        \centering
        \includegraphics[width=2.5in]{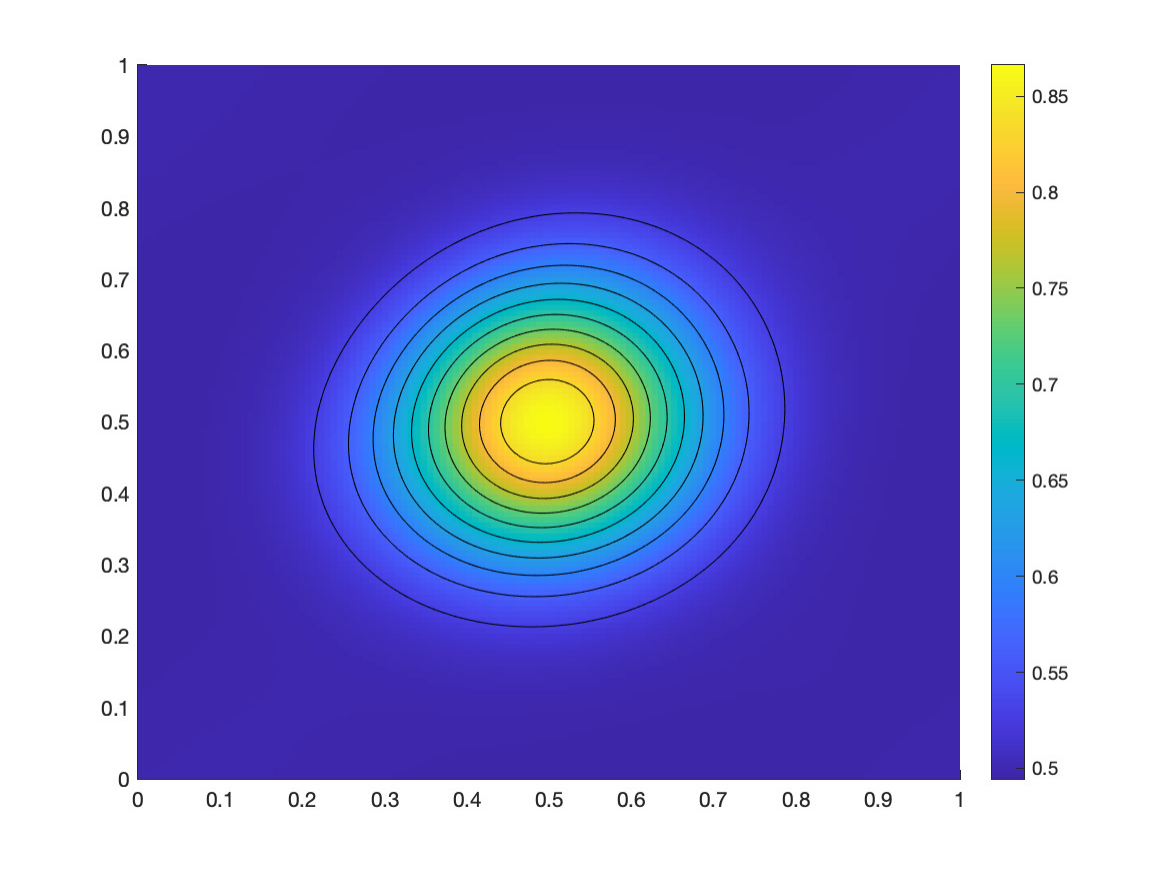} 
    \end{subfigure}
    \hfill
    \begin{subfigure}[b]{0.3\textwidth}
        \centering
        \includegraphics[width=2.5in]{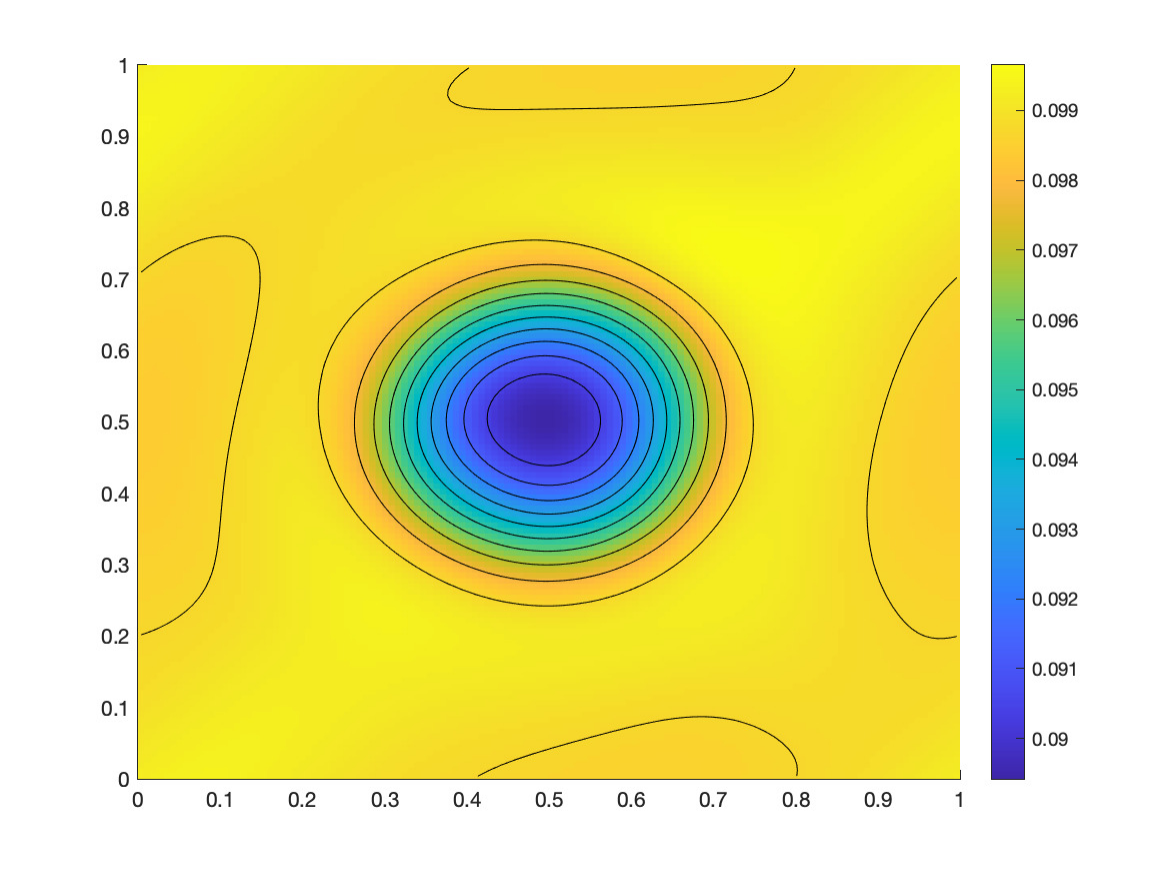} 
    \end{subfigure}
    \hfill
    \begin{subfigure}[b]{0.3\textwidth}
        \centering
        \includegraphics[width=2.5in]{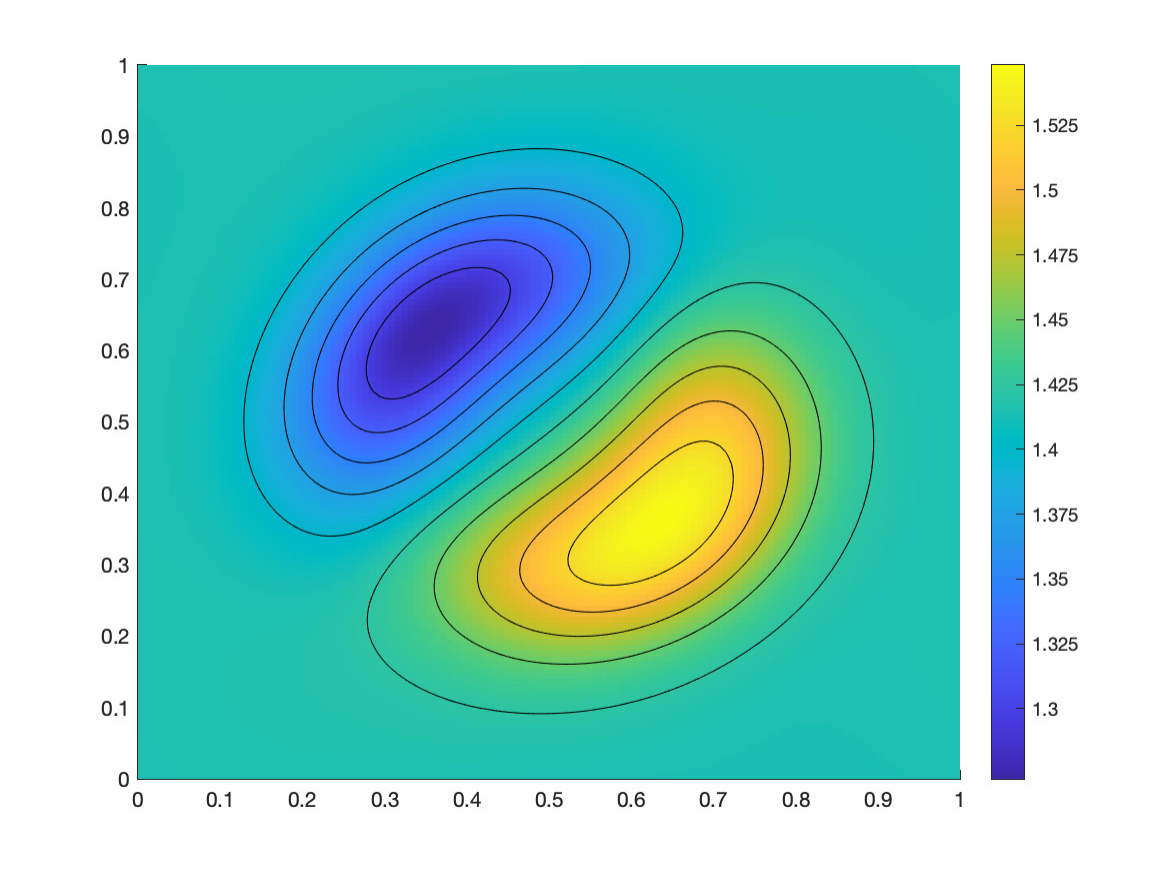} 
    \end{subfigure}
    \caption{Solutions structure of traveling vortex problem at $T=0$ (top), $T=1$ (bottom) on a grid with $128\times128$: $\rho$ (left), $p$ (middle), $|\vu|$ (right).}
    \label{Travel Vortex128}
\end{figure}

\begin{figure}[h]
    \centering
    \begin{subfigure}[b]{0.3\textwidth}
        \centering
        \includegraphics[width=2.5in]{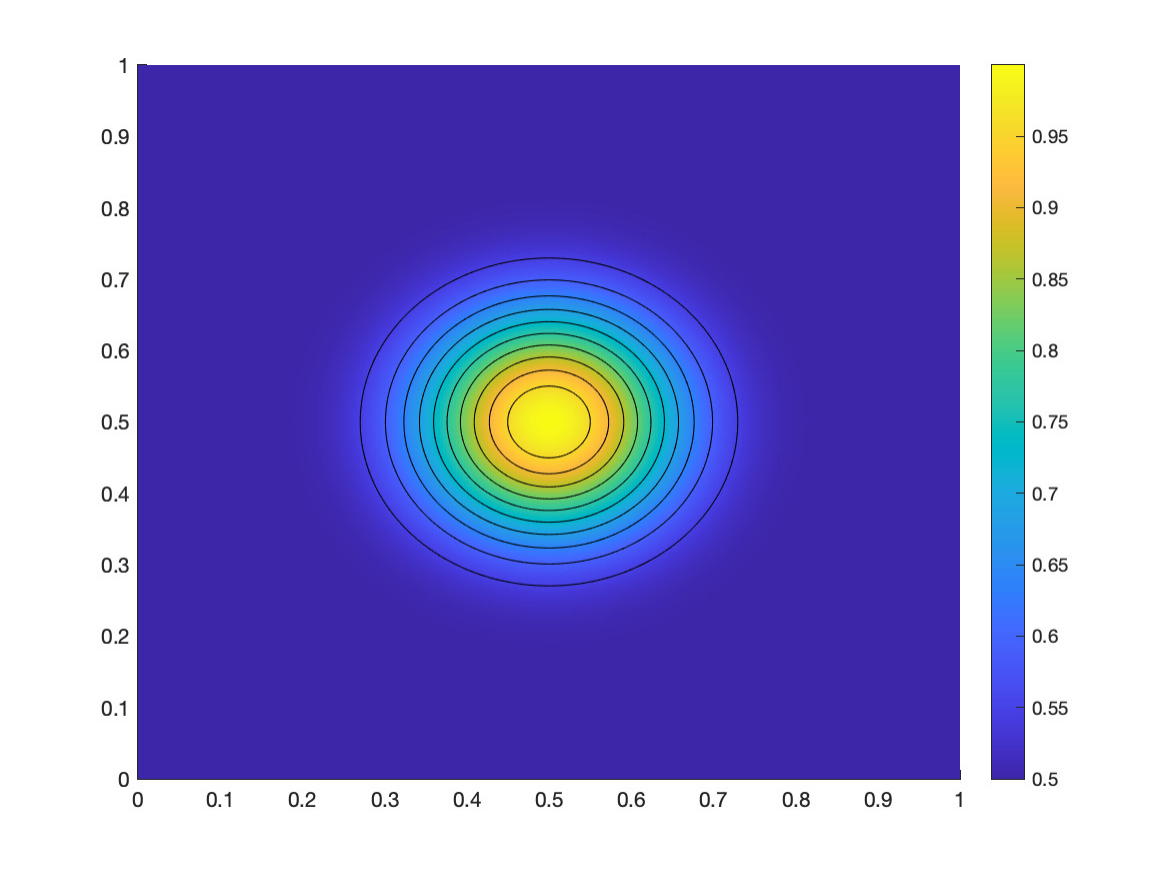} 
    \end{subfigure}
    \hfill
    \begin{subfigure}[b]{0.3\textwidth}
        \centering
        \includegraphics[width=2.5in]{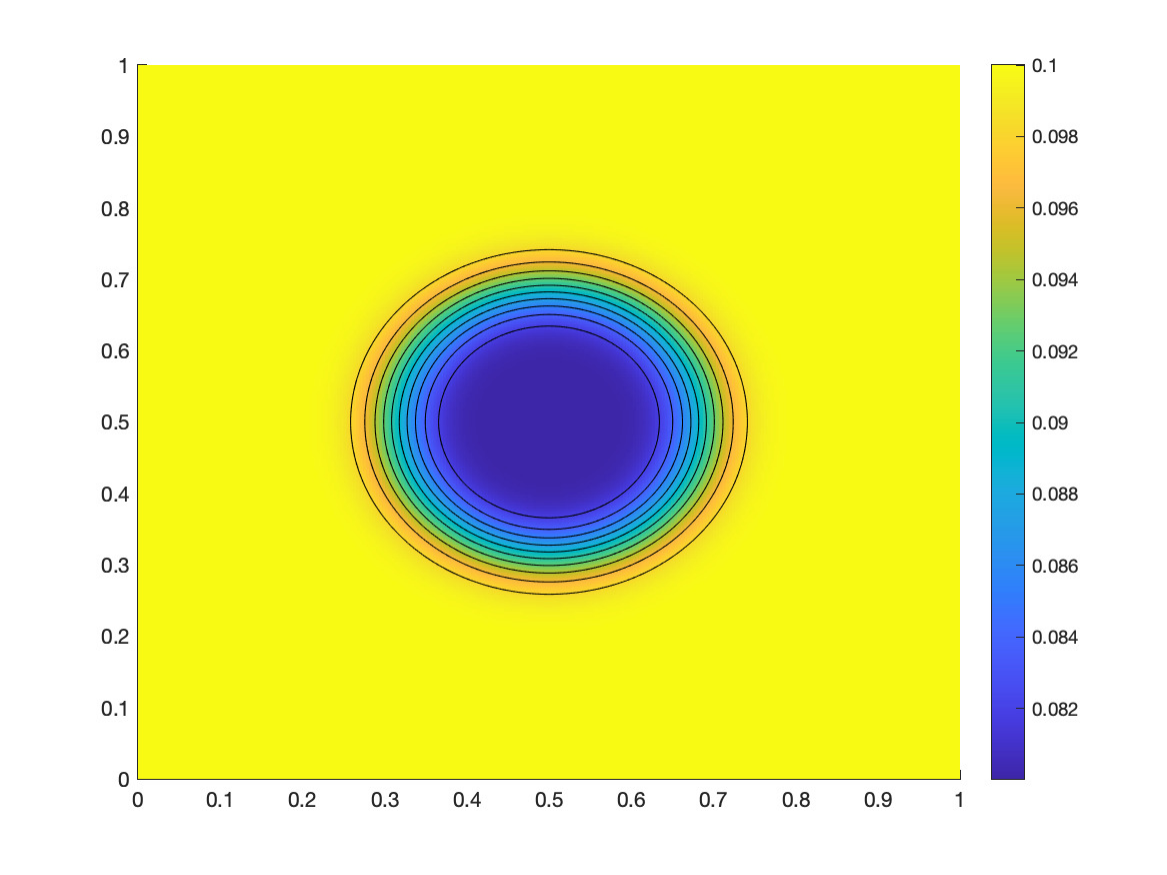} 
    \end{subfigure}
    \hfill
    \begin{subfigure}[b]{0.3\textwidth}
        \centering
        \includegraphics[width=2.5in]{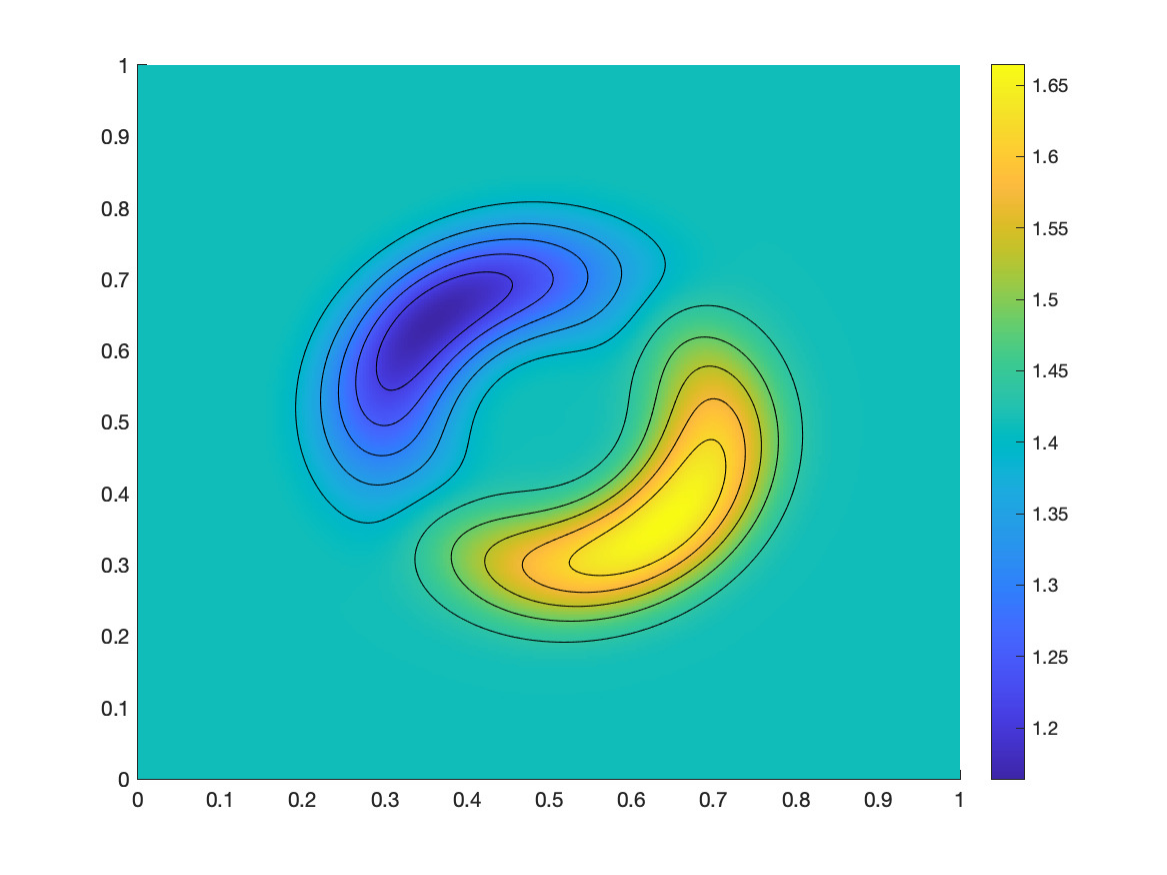} 
    \end{subfigure}
    
        \begin{subfigure}[b]{0.3\textwidth}
        \centering
        \includegraphics[width=2.5in]{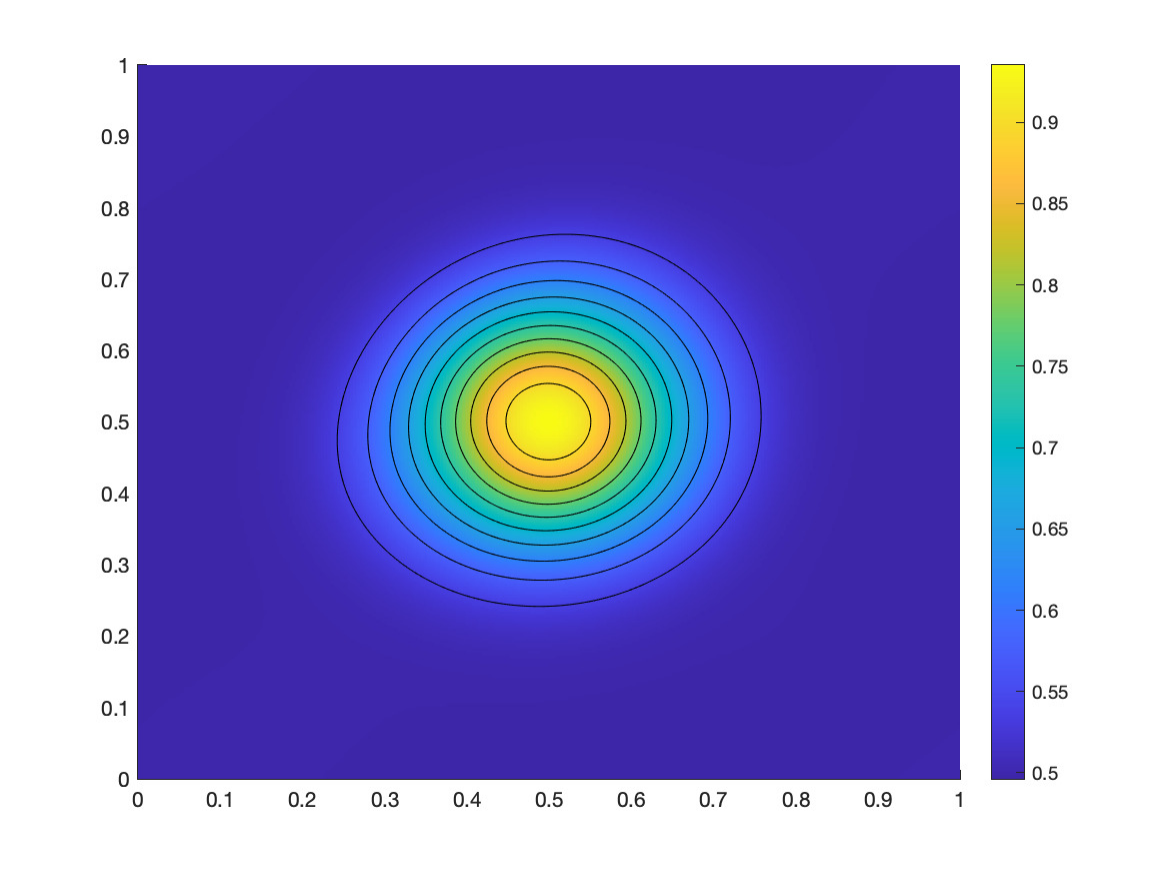} 
    \end{subfigure}
    \hfill
    \begin{subfigure}[b]{0.3\textwidth}
        \centering
        \includegraphics[width=2.5in]{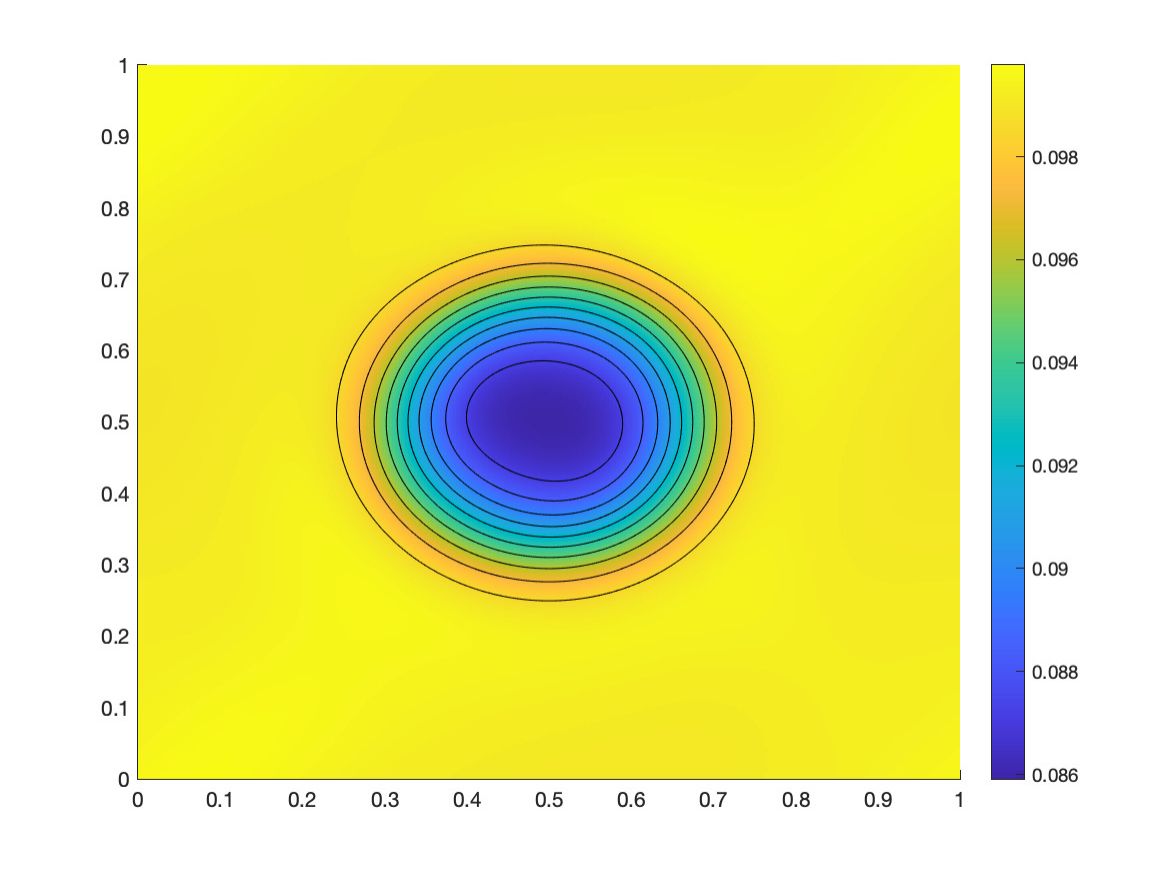} 
    \end{subfigure}
    \hfill
    \begin{subfigure}[b]{0.3\textwidth}
        \centering
        \includegraphics[width=2.5in]{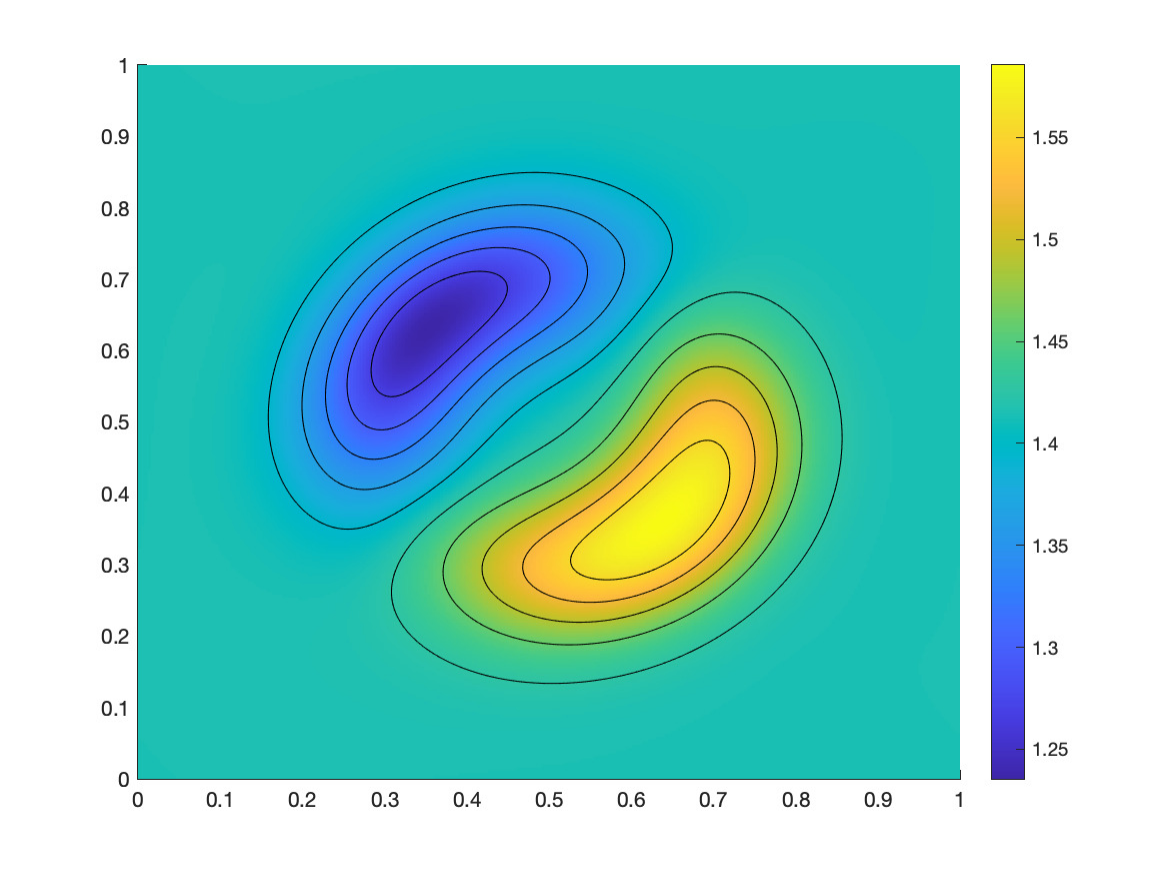} 
    \end{subfigure}
    \caption{Solutions structure of traveling vortex problem at $T=0$ (top), $T=1$ (bottom) on a grid with $256\times256$: $\rho$ (left), $p$ (middle), $|\vu|$ (right).}
    \label{Travel Vortex256}
\end{figure}

\begin{figure}[h]
    \centering
    \begin{subfigure}[b]{0.3\textwidth}
        \centering
        \includegraphics[width=2.5in]{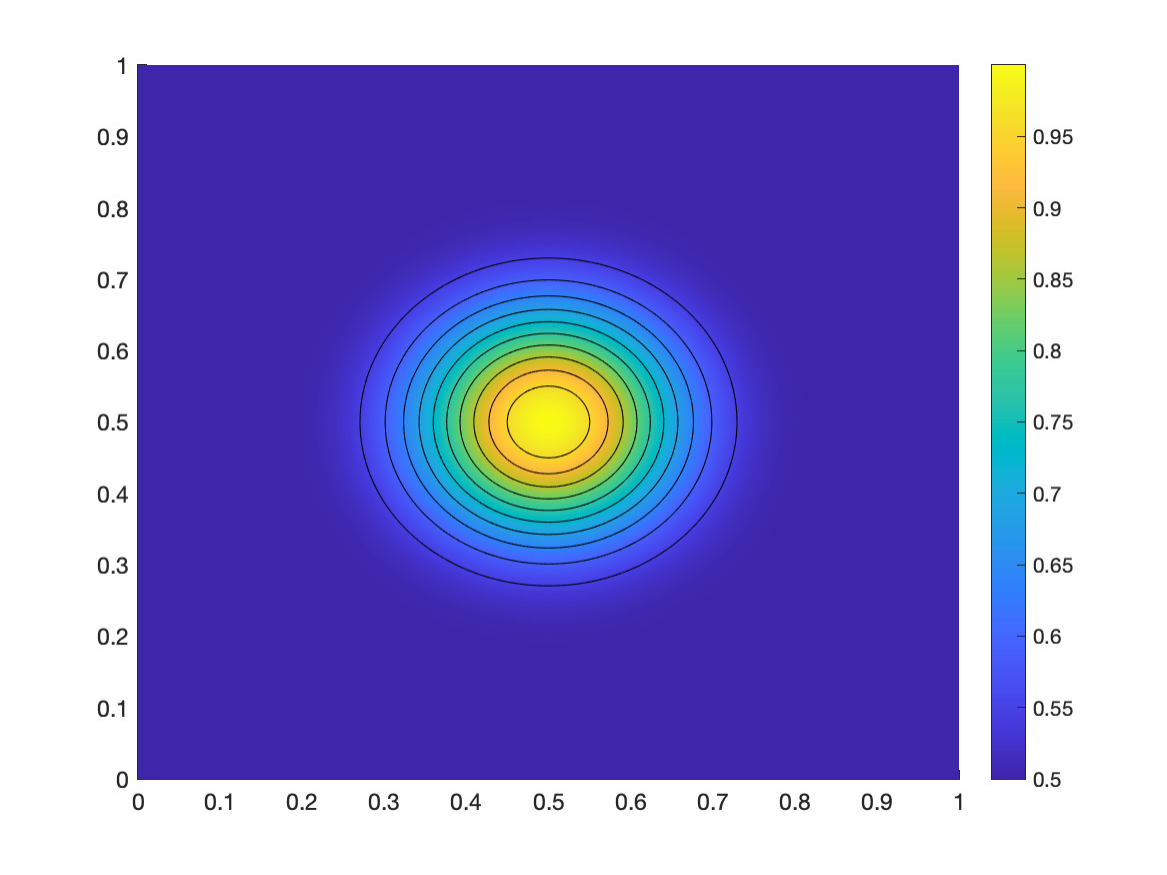} 
    \end{subfigure}
    \hfill
    \begin{subfigure}[b]{0.3\textwidth}
        \centering
        \includegraphics[width=2.5in]{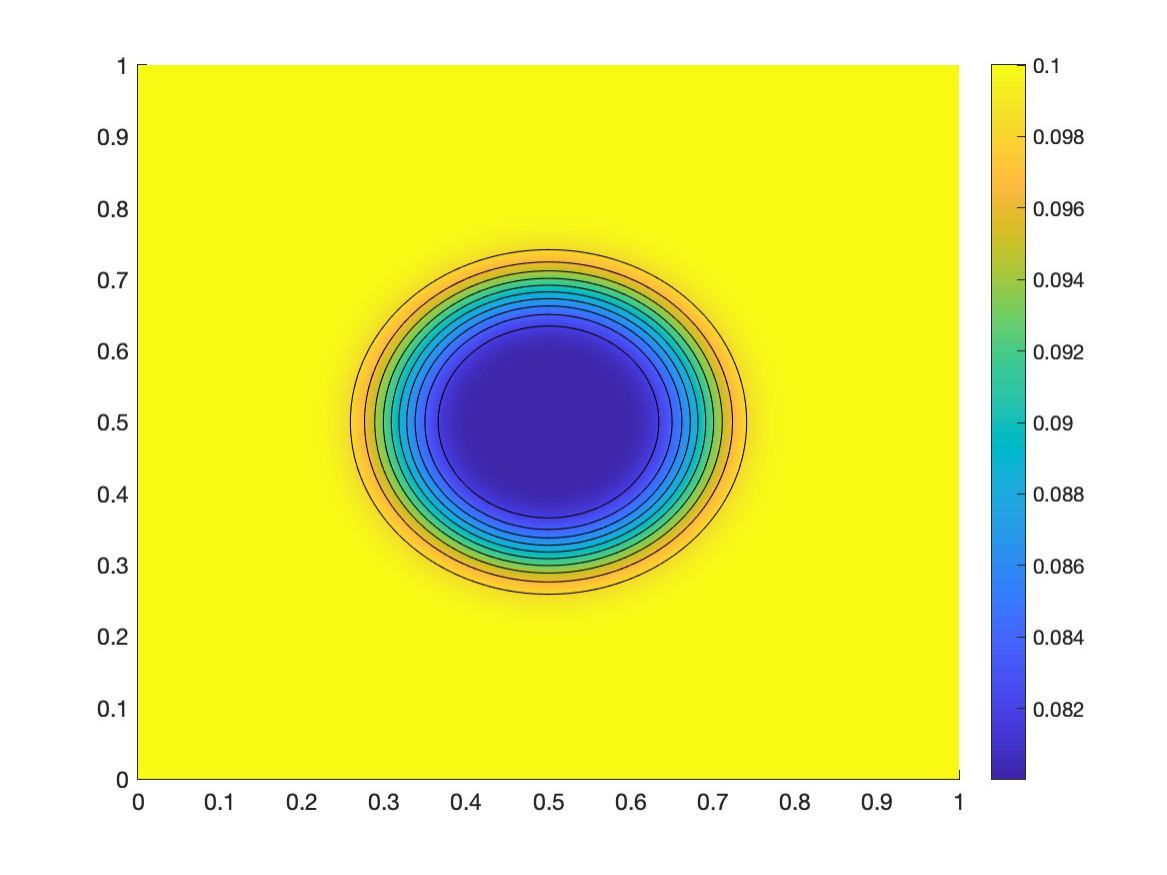} 
    \end{subfigure}
    \hfill
    \begin{subfigure}[b]{0.3\textwidth}
        \centering
        \includegraphics[width=2.5in]{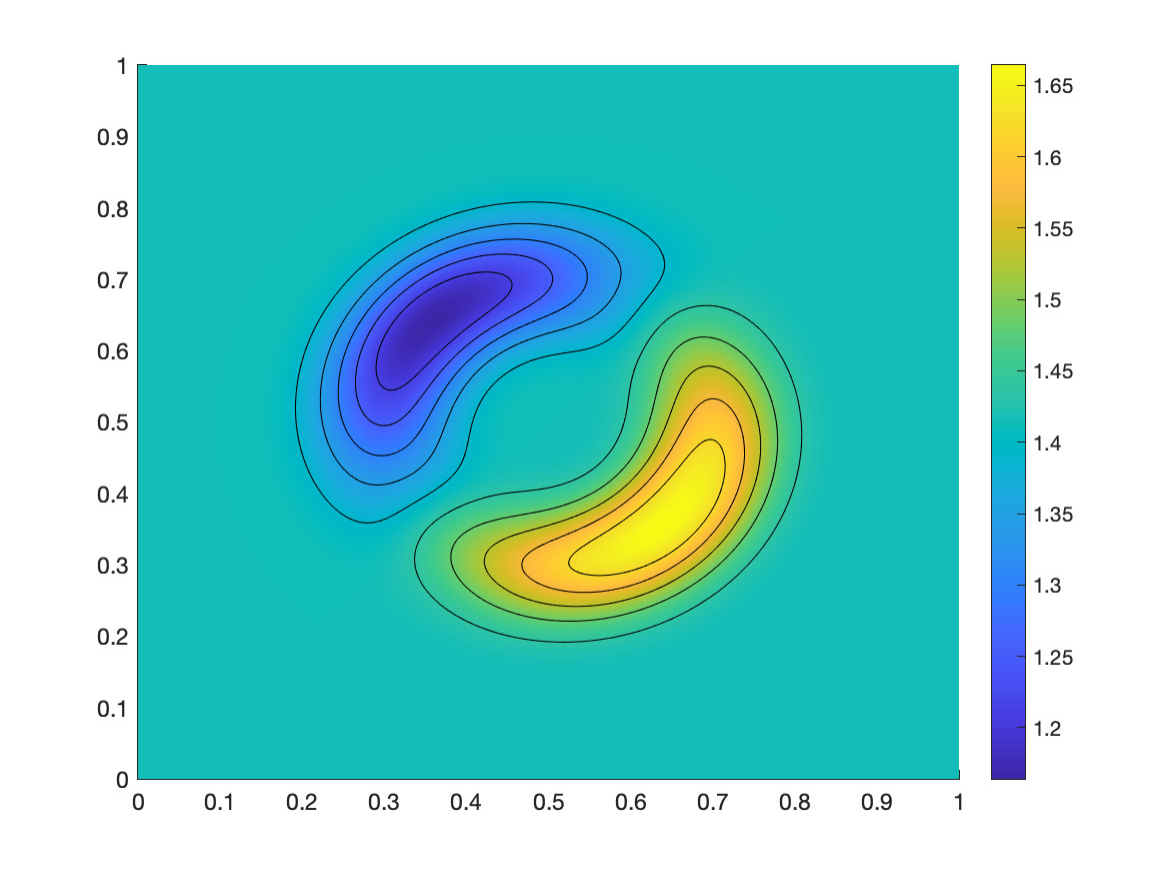} 
    \end{subfigure}
    
        \begin{subfigure}[b]{0.3\textwidth}
        \centering
        \includegraphics[width=2.5in]{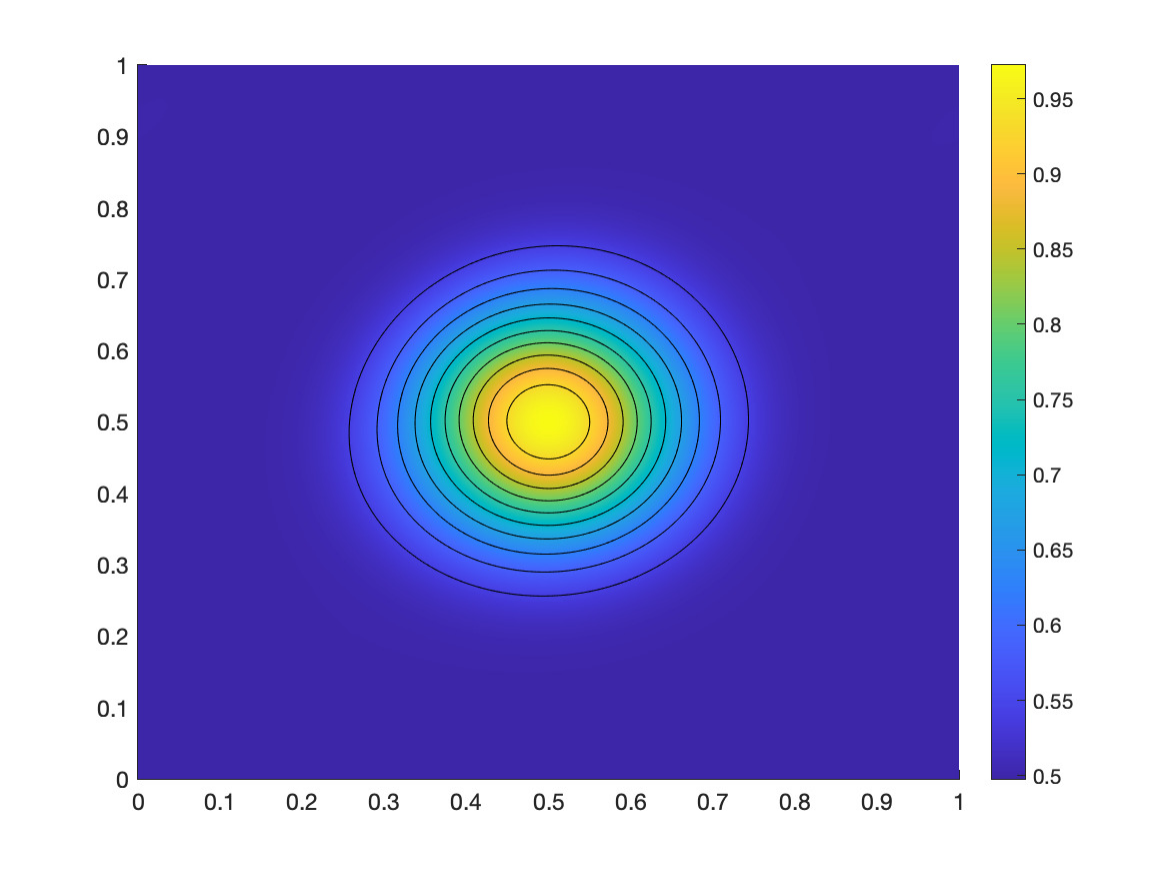} 
    \end{subfigure}
    \hfill
    \begin{subfigure}[b]{0.3\textwidth}
        \centering
        \includegraphics[width=2.5in]{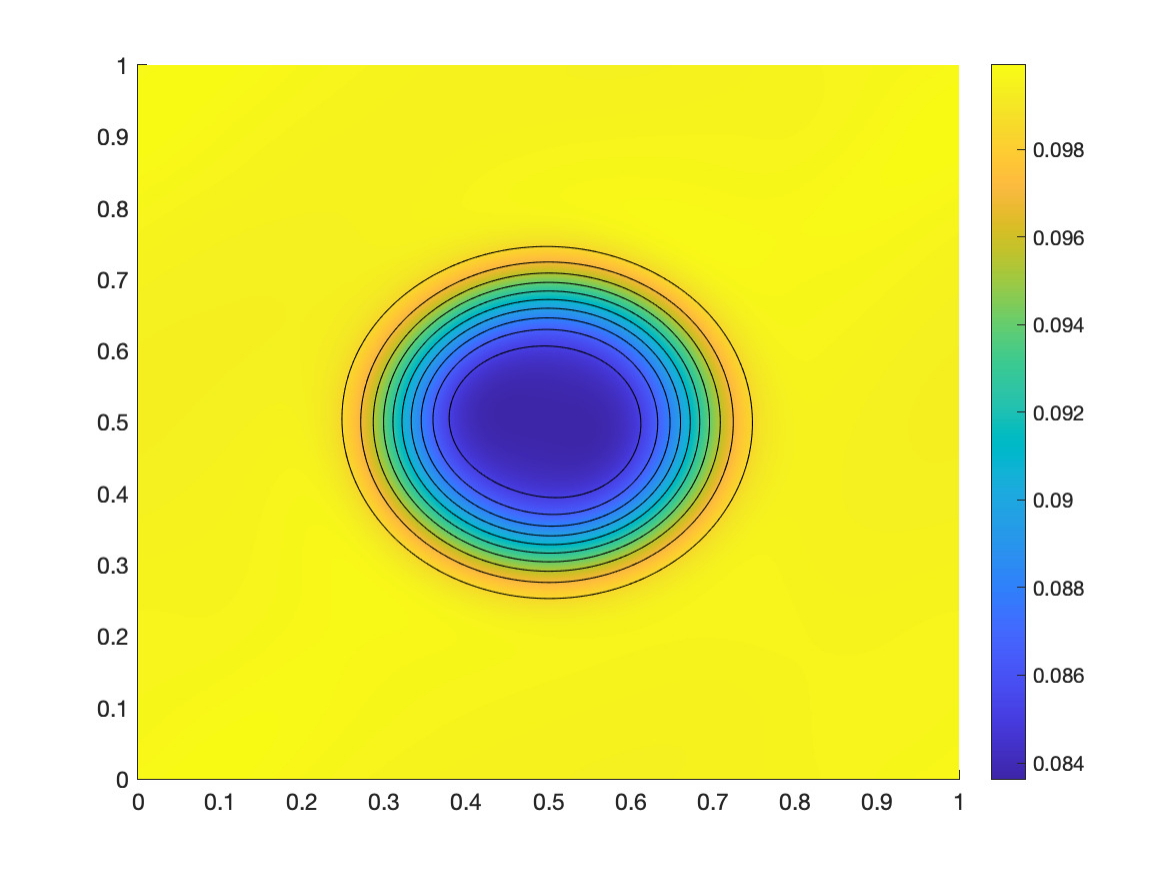} 
    \end{subfigure}
    \hfill
    \begin{subfigure}[b]{0.3\textwidth}
        \centering
        \includegraphics[width=2.5in]{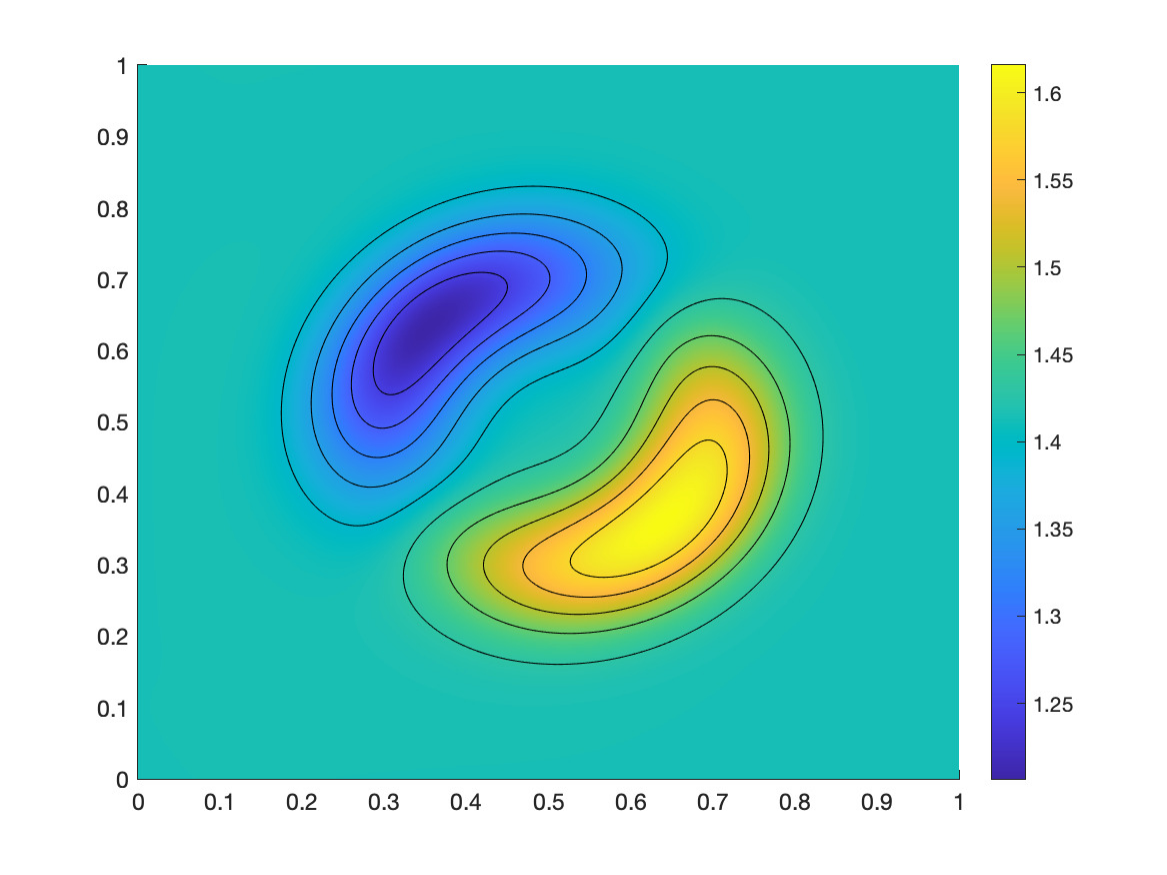} 
    \end{subfigure}
    \caption{Solutions structure of traveling vortex problem at $T=0$ (top), $T=1$ (bottom) on a grid with $512\times512$: $\rho$ (left), $p$ (middle), $|\vu|$ (right).}
    \label{Travel Vortex512}
\end{figure}

\begin{Example}\label{E1}
We consider a two-dimensional Riemann problem for the Euler equations,  consisting of two forward moving shocks and two stationary contact discontinuities. On $\Omega=[0,1]\times[0,1]$ the initial values have the form
\begin{equation}
(\rho,u,v,p)= 
\begin{cases}
(0.5313, 0, 0, 0.4),  & x > 0.5,  y>0.5, \\
(1, 0.7276, 0, 1), &   x < 0.5,  y>0.5,\\
(0.8, 0, 0, 1), & x < 0.5,  y<0.5,\\
(1, 0, 0.7276, 1), & x > 0.5,  y<0.5.
\end{cases}
\end{equation}
\end{Example}
The reference solution was computed on a mesh with $512\times512$ cells. Table~\ref{tab1} presents the error and shows the first-order convergence results for the density, momentum and energy. Figure~\ref{fig1} displays the solution structure of the density at different time with 512×512 grid cells. Solution contains two shocks and two curved contact waves.


\begin{table}[htbp]
   \centering
   \resizebox{0.75\textwidth}{!}{
   \begin{tabular}{c c c c c c c c c}  
         \hline
      $1/h$    & $\rho$ & EOC& $\rho u$ & EOC& $\rho v$ & EOC & $E$ & EOC\\
      \hline
      32     &   1.86e-02   & -&1.53e-02&- &1.53e-02 &- &6.11e-02  &-\\
      64      &  1.19e-02& 0.6429&1.19e-02  & 0.6358 &9.86e-03 &0.6359 & 4.00e-02& 0.6113\\
      128       & 6.40e-03  & 0.8999&5.25e-03 &0.9097 & 5.25e-03&0.9097 &2.13e-02&0.9124 \\
      256       & 2.49e-03  & 1.3638 &2.00e-03 & 1.3909 &  2.00e-03&1.3915 &7.93e-03&1.4221 \\
      \hline
   \end{tabular}}
   \caption{Errors in $L^1$-norm and EOC for density, momentum and energy at $t=0.15$ for Example~\ref{E1}.}
   \label{tab1}
\end{table}

\begin{figure}[h]
    \centering
    \begin{subfigure}[b]{0.3\textwidth}
        \centering
        \includegraphics[width=2.5in]{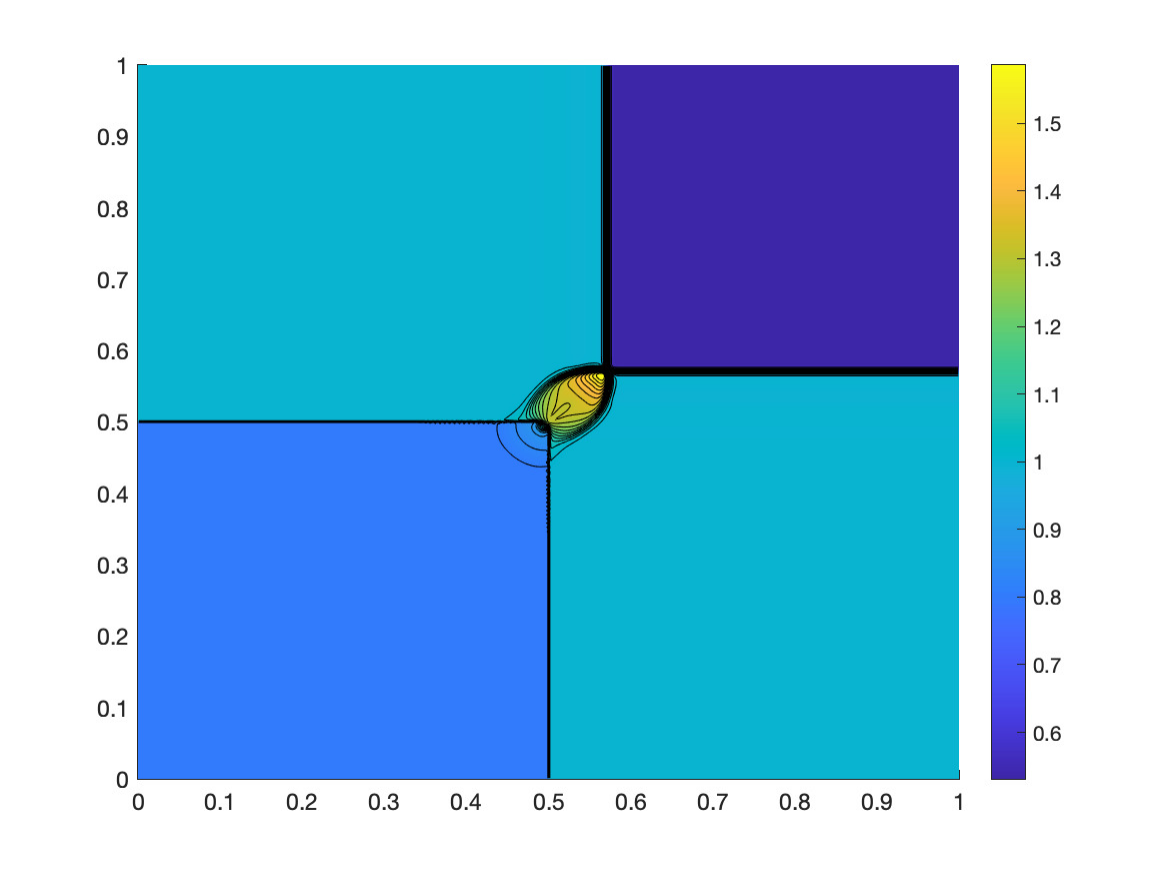} 
    \end{subfigure}
    \hfill
    \begin{subfigure}[b]{0.3\textwidth}
        \centering
        \includegraphics[width=2.5in]{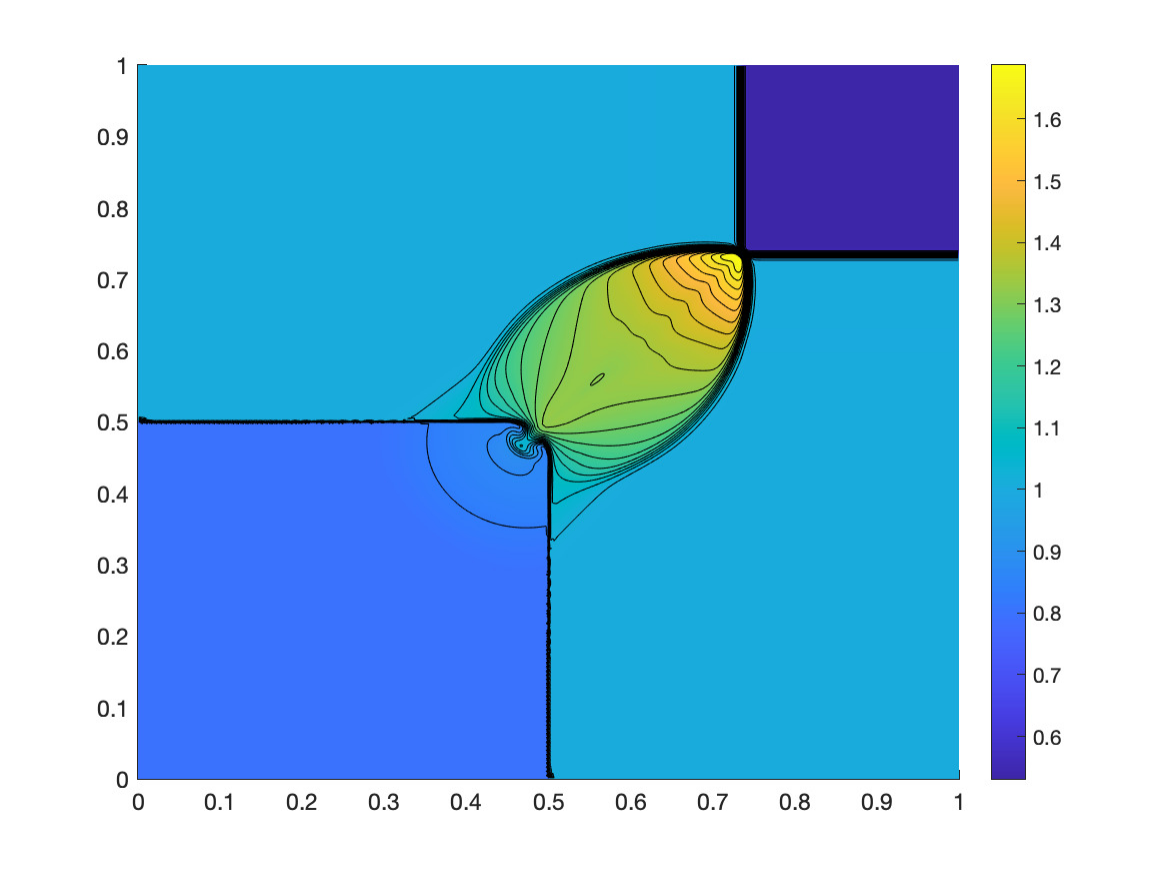} 
    \end{subfigure}
    \hfill
    \begin{subfigure}[b]{0.3\textwidth}
        \centering
        \includegraphics[width=2.5in]{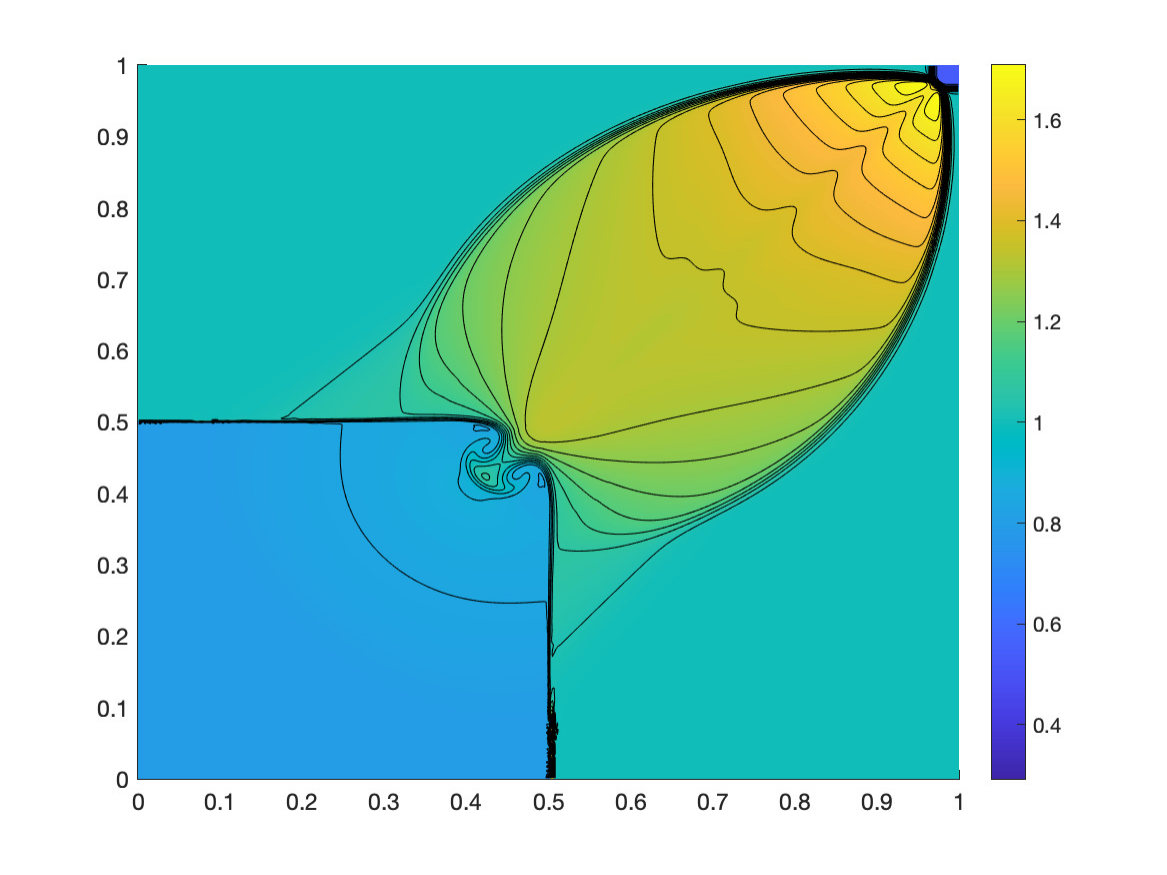} 
    \end{subfigure}
    \caption{Density at different time for the 2D Riemann problem described in Example \ref{E1} using 512$\times$512 grid cells, t=0.045 (left), t=0.15 (middle), t=0.3 (right).}
    \label{fig1}
\end{figure}

\begin{Example}\label{E2}
We consider a two-dimensional Riemann problem with the initial data
\begin{equation}
(\rho,u,v,p)= 
\begin{cases}
(0.5, 0.5, -0.5, 5),  & x > 0.5,  y>0.5, \\
(1, 0.5, 0.5, 5), &   x < 0.5,  y>0.5,\\
(2, -0.5, 0.5, 5), & x < 0.5,  y<0.5,\\
(1.5, -0.5, -0.5, 5), & x > 0.5,  y<0.5.
\end{cases}
\end{equation}
\end{Example}

These initial data yield the so-called spiral problem. We compute the error and convergence rate of density, momentum and energy by using the reference solution on a grid with 512×512 cells, the results are listed in Table~\ref{tab2}. Figure~\ref{fig2} presents the contour of the density obtained on different meshes. We can clearly recognize four contact waves forming a spiral singularity.

\begin{table}[htbp]
   \centering
   \resizebox{0.75\textwidth}{!}{
   \begin{tabular}{c c c c c c c c c} 
     \hline
      $1/h$    & $\rho$ & EOC& $\rho u$ & EOC& $\rho v$ & EOC & $E$ & EOC\\
      \hline
       32     &    4.76e-02  & -&4.14e-02&- &4.71e-02 &- &6.77e-02 &-\\
      64      &   3.09e-02 & 0.6253& 2.55e-02&  0.7002&2.85e-02 & 0.7240 &4.50e-02 & 0.5870\\
      128       & 1.79e-02  & 0.7855&1.46e-02 &0.8066 &1.60e-02 & 0.8318 & 2.62e-02 &0.7827\\
      256       & 7.71e-03 &  1.2150&6.31e-03 & 1.2064 &6.75e-03 & 1.2461&1.12e-02  & 1.2251 \\
     \hline
   \end{tabular}}
   \caption{Errors in $L^1$-norm and EOC for density, momentum and energy at $t=0.2$ for Example~\ref{E2}.}
   \label{tab2}
\end{table}

\begin{figure}[h]
    \centering
    \begin{subfigure}[b]{0.3\textwidth}
        \centering
        \includegraphics[width=2.5in]{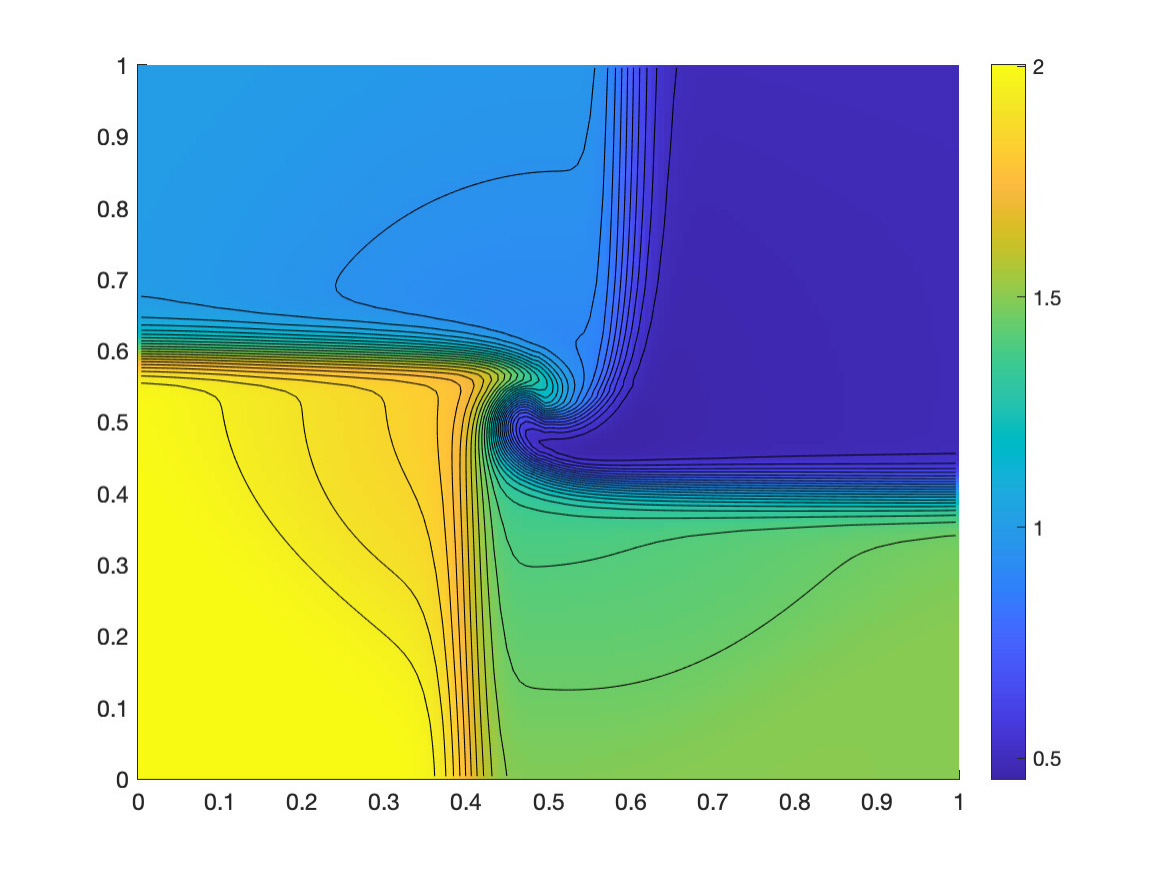} 
    \end{subfigure}
    \hfill 
    \begin{subfigure}[b]{0.3\textwidth}
        \centering
        \includegraphics[width=2.5in]{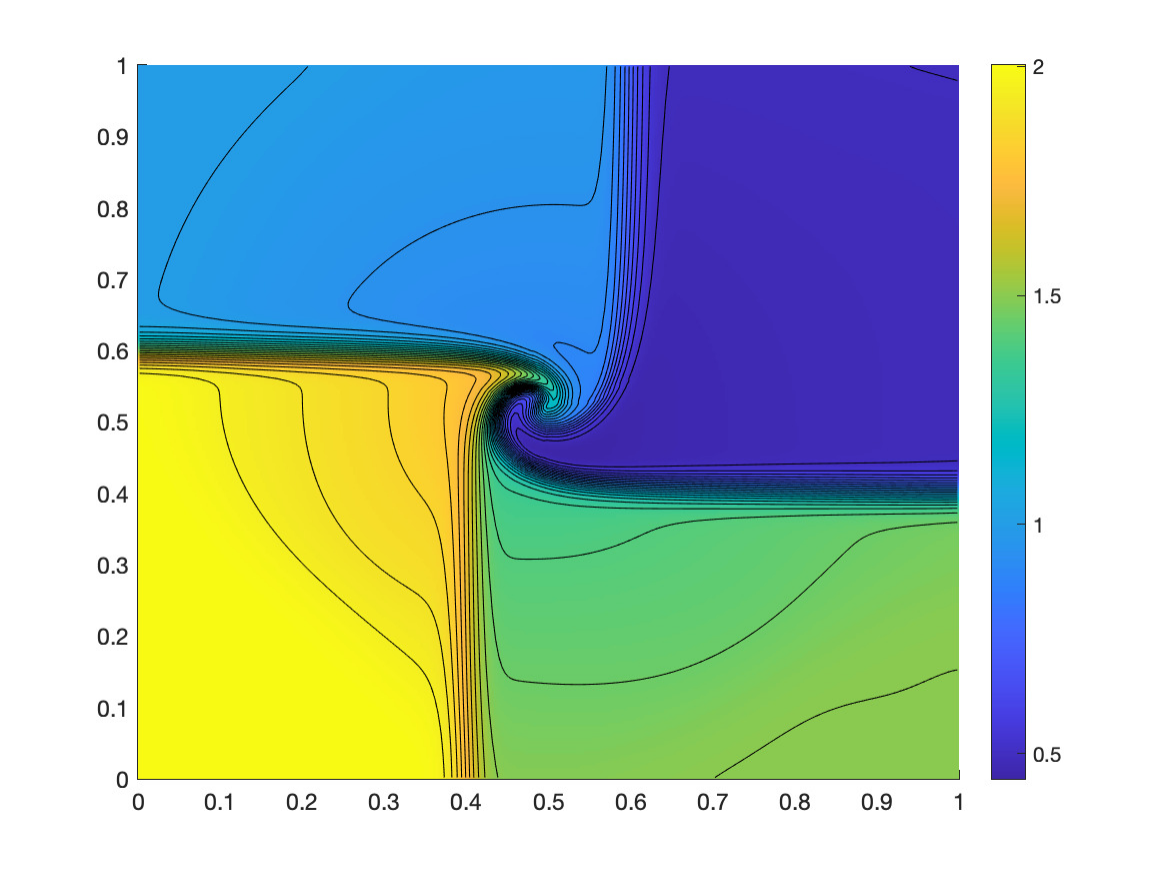} 
    \end{subfigure}
    \hfill
    \begin{subfigure}[b]{0.3\textwidth}
        \centering
        \includegraphics[width=2.5in]{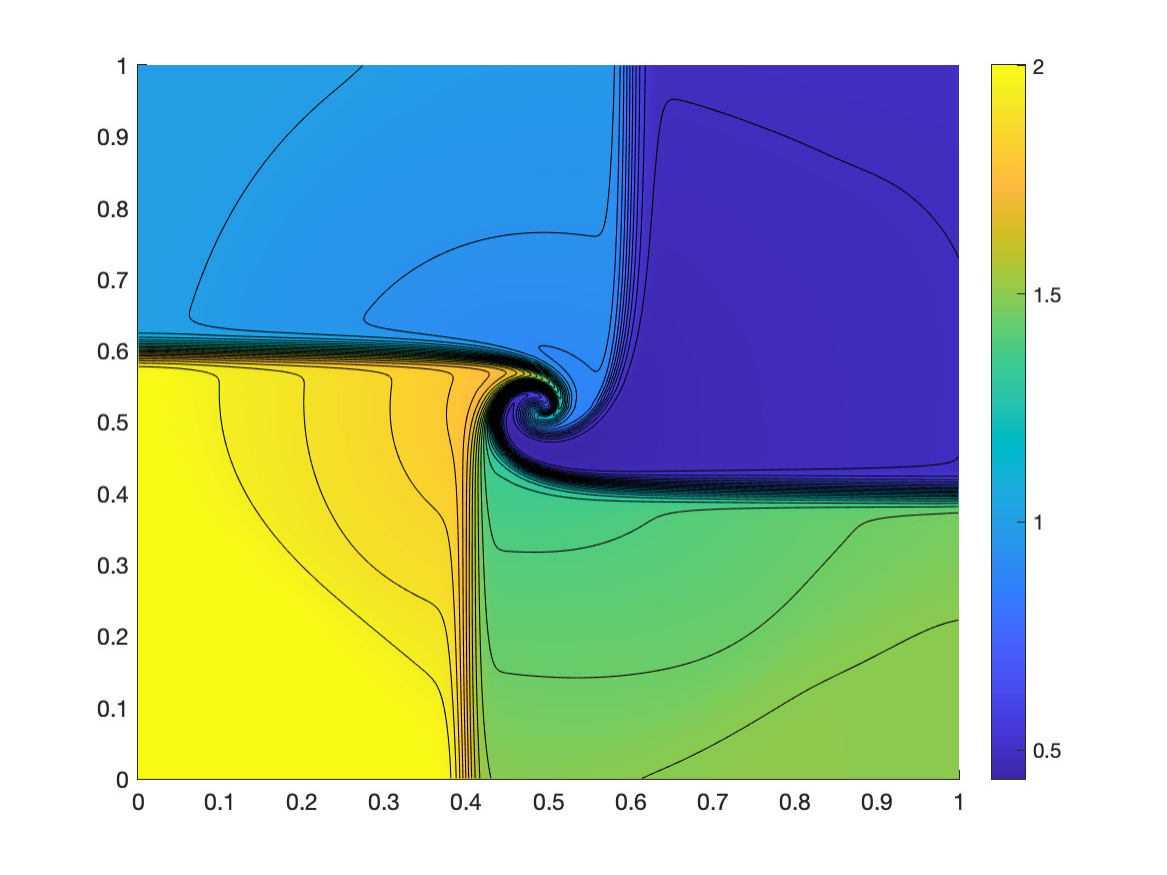} 
    \end{subfigure}
    \caption{Density on different meshes $128\times128$ (left), $256\times256$ (middle), $512\times512$ (right) for the spiral problem described in Example~\ref{E2}.}
    \label{fig2}
\end{figure}

\section*{Acknowledgements} The work of M.L. and Z.T. was supported by the Deutsche Forschungsgemeinschaft (DFG, German Research Foundation) - Project number DFG through the project 525800857 funded within the Focused Programme SPP 2410 ``Hypebrolic Balance Laws: Complexity, Scales and Randomness''. M.L. is grateful to the Gutenberg Research College for supporting her research. The work of  Y.Y. was supported by National Natural Science Foundation of China, No. 12401527. The authors sincerely thank Professor Christiane Helzel and Erick Chudzik for providing the code of the EG2 operator.

\appendix
\section{Exact evolution operator for the wave equation system}

Consider a cone with the apex $P=(x,y,t+\Delta t)$ and the base points $Q=Q(\theta)=(x+c\Delta t\cos\theta, y+c\Delta t\sin\theta, t)$ paramertrized by the angle $\theta\in[0,2\pi]$. Denote by $P'=(x,y,t)$ the center of the base of the cone. The lines from $Q(\theta)$ to $P$ generating the mantle of the so-called bicharacteristic cone are called bicharacteristics.

In \cite{LMW:2000} Luk{\'a}{\v c}ov{\'a}-Medvi{\softd}ov{\'a}, Morton and Warnecke have applied the method of bicharacteristics and derived the exact evolution operator for the wave equation system, which reads as:
\begin{eqnarray}\label{2.5}
\phi_P&=&\frac{1}{2\pi}\int^{2\pi}_0[\phi_Q-u_Q\cos\theta-v_Q\sin\theta]\dthe-\frac{1}{2\pi}\int^{t+\Delta t}_{t}\int^{2\pi}_0S({\tilde t},\theta)\dthe\dSt,
\end{eqnarray}

\begin{eqnarray}\label{2.6}
u_P&=&\frac{1}{2\pi}\int^{2\pi}_0[-\phi_Q\cos\theta+u_Q\cos^2\theta+v_Q\sin\theta\cos\theta]\dthe+\frac{1}{2}u_{P'}\\\nonumber
&~&+\frac{1}{2\pi}\int^{t+\Delta t}_{t}\int^{2\pi}_0S({\tilde t},\theta)\cos\theta\dthe\dSt-\frac{1}{2}c\int^{t+\Delta t}_{t}\phi_x(x,y,t)\dt,
\end{eqnarray}

\begin{eqnarray}\label{2.7}
v_P&=&\frac{1}{2\pi}\int^{2\pi}_0[-\phi_Q\sin\theta+u_Q\cos\theta\sin\theta+v_Q\sin^2\theta]\dthe+\frac{1}{2}v_{P^\prime}\\\nonumber
&~&+\frac{1}{2\pi}\int^{t+\Delta t}_{t}\int^{2\pi}_0S({\tilde t},\theta)\sin\theta\dthe\dSt-\frac{1}{2}c\int^{t+\Delta t}_{t}\phi_y(x,y,t)\dt.
\end{eqnarray}

The source term $S$ given as
\begin{equation}\label{A4}
S( \tilde t,\theta)=c[u_x(\tilde x,\tilde y, \tilde t)\sin^2\theta-(u_y(\tilde x,\tilde y,\tilde t)+v_x(\tilde x,\tilde y,\tilde t))\sin\theta\cos\theta+v_y(\tilde x,\tilde y,\tilde t)\cos^2\theta],
\end{equation}
where$(\tilde x,\tilde y)=(x+c(t+\Delta t-\tilde t)\cos\theta,y+c(t+\Delta t-\tilde t)\sin\theta)$ and $\tilde t\in[t,t+\Delta t]$.

 \begin{figure}[!h]
\centering 
\begin{tikzpicture}
\fill (2,0) circle (0.03);
\node at (2,0) [left] {$P^\prime$};
\node at (3.5,3) [above] {$P=(x,y,t+\Delta t)$};
\draw [-,very thick] (2,0) ellipse (2 and 0.5);
\draw [-,very thick] (2.5,0) arc ( 0:45: 1 and 0.25);
\draw [-,very thick] (2,0)--(2,3);
\draw [-,very thick] (0,0)--(2,3);
\draw [-,very thick] (2.5,0.455)--(2,3);
\draw [-,very thick] (2,0)--(4,0);
\draw [-,very thick] (2,0)--(2.5,0.455);
\node at (3.5,0.4) [above] {$Q(\theta)$};
\node at (2.5,0.2) [right] {$\theta$};
\node at (4.2,0)[right]{$P^\prime=(x,y,t)$};

\fill (9,0) circle (0.03);
\node at (9,0) [left] {$P^\prime$};
\node at (13.5,3) [above] {$P=(x,y,t+\Delta t)$};
\draw [-,very thick] (9,0) ellipse (2 and 0.3);
\draw [-,very thick] (9,0)--(12,3);
\draw [-,very thick] (7,0)--(12,3);
\draw [-,very thick] (11,0)--(12,3);
\draw [-,very thick] (9,0)--(11,0);
\draw [-,very thick] (9,0)--(10.5,0.2);
\node at (10.5,0.25) [above] {$Q(\theta)$};
\node at (11.2,0) [right] {$P^\prime=(x-u^\prime\Delta t,y-v^\prime\Delta ,t)$};
\end{tikzpicture}
\caption{Update of point value using evolution operator: wave equation(left), Euler equations(right).}
\end{figure}
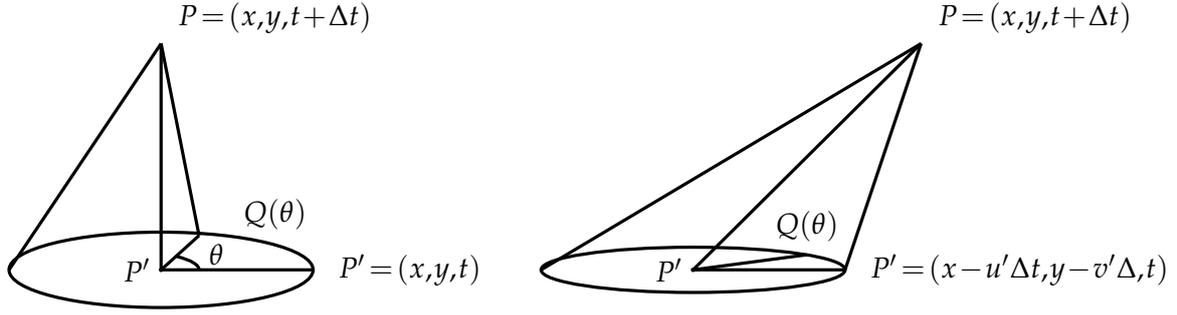

\section{Exact evolution operator for the Euler equations}\label{evo-Euler}
 
The linearized Euler equations with fixed constant states $\rho^{\prime},u^{\prime},v^{\prime},p^{\prime}$ arising in the Jacobian matrixes $A_1, A_2$ are given by 
\begin{equation}
\vU_t+\bm A_1\vU_x+\bm A_2\vU_y=0,
\end{equation}
where $\vU=(\rho,u,v,p)^T$ and  
\[\bm A_1=\begin{pmatrix} 
         u^{\prime} & \rho^{\prime} & 0 &0 \\
         0 & u^{\prime}  & 0 & \frac{1}{\rho^{\prime}}\\
         0 & 0 & u^{\prime} & 0 \\
         0 & \rho^{\prime}c^{{\prime}2} & 0 & u^{\prime}
   \end{pmatrix},
   \quad
   \bm A_2=\begin{pmatrix} 
         v^{\prime} & 0& \rho^{\prime}  &0 \\
         0 & v^{\prime}  & 0 &0 \\
         0 & 0 & v^{\prime} & \frac{1}{\rho^{\prime}}\\
         0 & 0& \rho^{\prime}c^{{\prime}2}& v^{\prime}
   \end{pmatrix}.\]
Consequently, the exact evolution operator for the Euler equations read as:
\begin{eqnarray}\label{}
\rho_P&=&\rho_{P^{\prime}}-\frac{p_{P^{\prime}}}{c^{{\prime}2}}+\frac{1}{2\pi}\int^{2\pi}_{0}\frac{p_Q}{c^{{\prime}2}}-\frac{\rho^{\prime}}{c^{\prime}}u_Q\cos\theta-\frac{\rho^{\prime}}{c^{\prime}}v_Q\sin\theta\dthe\\\nonumber
&~&-\frac{\rho^{\prime}}{c^{\prime}}\frac{1}{2\pi}\int^{t+\Delta t}_{t}\int^{2\pi}_0R(\bm x-({\tilde {\bm u}}-c\bm n(\theta))(t+\Delta t-\tilde t),{\tilde t},\theta)\dthe\dSt,
\end{eqnarray}
\begin{eqnarray}\label{}
u_P&=&\frac{1}{2\pi}\int^{2\pi}_{0}-\frac{p_Q}{\rho^{\prime}c^{\prime}}\cos\theta+u_Q(2\cos^2\theta-\frac{1}{2})+2v_Q\sin\theta\cos\theta\dthe\\\nonumber
&&+\frac{1}{2\pi}\int^{t+\Delta t}_{t}\int^{2\pi}_0R(\bm x-({\tilde {\bm u}}-c\bm n(\theta))(t+\Delta t-\tilde t),{\tilde t},\theta)\cos\theta \dthe\dSt\\\nonumber
&&-\frac{1}{2\tilde \rho}\int^{t+\Delta t}_{t}p_x(P^{\prime}(\tilde t))\dt+\frac{1}{2}u_{p^{\prime}},
\end{eqnarray}
\begin{eqnarray}\label{}
v_P&=&\frac{1}{2\pi}\int^{2\pi}_{0}-\frac{p_Q}{\rho^{\prime}c^{\prime}}\sin\theta+2u_Q\sin\theta\cos\theta+v_Q(2\sin^2\theta-\frac{1}{2}) \dthe\\\nonumber
&&+\frac{1}{2\pi}\int^{t+\Delta t}_{t}\int^{2\pi}_0R(\bm x-({\tilde {\bm u}}-c\bm n(\theta))(t+\Delta t-\tilde t),{\tilde t},\theta)\sin\theta \dthe\dSt\\\nonumber
&&-\frac{1}{2\tilde \rho}\int^{t+\Delta t}_{t}p_y(P^{\prime}(\tilde t))\dt+\frac{1}{2}v_{p^{\prime}},
\end{eqnarray}
\begin{eqnarray}\label{}
p_P&=&-p_{P^{\prime}}-\rho^{\prime}c^{\prime}u_Q\cos\theta-\rho^{\prime}c^{\prime}v_Q\sin\theta\dthe\\\nonumber
&&-\rho^{\prime}c^{\prime}\frac{1}{2\pi}\int^{t+\Delta t}_{t}\int^{2\pi}_0R(\bm x-({\tilde {\bm u}}-c\bm n(\theta))(t+\Delta t-\tilde t),{\tilde t},\theta)\dthe\dSt,
\end{eqnarray}
where
$$(\bm x-({\tilde {\bm u}}-c^\prime\bm n(\theta))(t+\Delta t-\tilde t)=(x-(u^\prime-c^\prime\cos\theta)(t+\Delta t-\tilde t),y-(v^\prime-c^\prime\sin\theta)(t+\Delta t-\tilde t)),$$
and 
$$R(\bm x,t,\theta):=c^\prime[u_x(\bm x,t,\theta)\sin^2\theta-(u_y(\bm x,t,\theta)+v_x(\bm x,t,\theta))\sin\theta\cos\theta+v_y(\bm x,t,\theta)\cos^2\theta].$$

\section{Proof of Lemma \ref{l31}}\label{app-proof}

\begin{proof}
By choosing the integral path in \eqref{qEC-1D} as $\vU^j_{\sigma}(\xi):=\frac{1}{2}(\vU^j_{\sigma}+\vU^{j+1}_{\sigma})+\xi\alpha^j_{\sigma}\vcr^j_{\sigma} = \frac{1}{2}(\vU^j_{\sigma}+\vU^{j+1}_{\sigma})+\xi  \jump {\vU_j} $, we shall rewrite $q^*_j$ as follows 
\begin{align*}
q^*_j&=\int^{\frac{1}{2}}_{-\frac{1}{2}}2\xi\braket{\bm{A} (\vU^j_{\sigma}(\xi))\vcr^j_{\sigma},\vcr^j_{\sigma}}\dxi\br
&=\int^{\frac{1}{2}}_{-\frac{1}{2}}2\xi\braket{\bm{A}(\vU^j_{\sigma}(\xi))\vcr^j(\vU^j_{\sigma}(\xi)),\vcr^j(\vU^j_{\sigma}(\xi))}\dxi\\\nonumber
&+\int^{\frac{1}{2}}_{-\frac{1}{2}}4\xi\braket{\bm{A}(\vU^j_{\sigma}(\xi))\vcr^j(\vU^j_{\sigma}(\xi)),\vcr^j_{\sigma}-\vcr^j(\vU^j_{\sigma}(\xi))}\dxi\\\nonumber
&+\int^{\frac{1}{2}}_{-\frac{1}{2}}2\xi\braket{\bm{A}(\vU^j_{\sigma}(\xi)) (\vcr^j_{\sigma}-\vcr^j(\vU^j_{\sigma}(\xi)),\vcr^j_{\sigma}-\vcr^j(\vU^j_{\sigma}(\xi))}\dxi := \sum_{i=1}^3 I_i.
\end{align*}
In what follows we estimate $q^*_j$, i.e.\ $I_i, i=1,2,3,$ term by term. 

Firstly, we shall control $I_1$ with
 \begin{align}
I_1&=\int^{\frac{1}{2}}_{-\frac{1}{2}}2\xi\braket{\bm{A}(\vU^j_{\sigma}(\xi))\vcr^j(\vU^j_{\sigma}(\xi)),\vcr^j(\vU^j_{\sigma}(\xi))}\dxi =\int^{\frac{1}{2}}_{-\frac{1}{2}}2\xi\lambda^j(\vU^j_{\sigma}(\xi))|\vcr^j(\vU^j_{\sigma}(\xi))|^2 \dxi  \br
&=\int^{\frac{1}{2}}_{-\frac{1}{2}}2\xi\lambda^j(\vU^j_{\sigma}(\xi)) \dxi 
 =\int^{\frac{1}{2}}_{-\frac{1}{2}}\left(\frac{1}{4}-\xi^2\right)\frac{\mathrm{d}}{\mathrm{d}\xi}\lambda^j(\vU^j_{\sigma}(\xi))
 \leqslant\frac{1}{4}[\lambda^j_{\sigma}]^+.
 \end{align}
 
Note that we have used the integration by part in the fourth equality, and obtained the last inequality by analyzing case by case (rarefaction wave and shock wave).
 
Secondly, applying Taylor's expansion:
$$
\lambda^j(\vU^j_{\sigma}(\xi))=\lambda^j_{\sigma}+\xi\frac{\mathrm{d}\lambda^j(\vU^j_{\sigma}(0))}{\mathrm{d}\xi}+\frac{\xi^2}{2}\frac{\mathrm{d}^2\lambda^j(\vU^j_{\sigma}(0))}{\mathrm{d}\xi^2}+(\bar{\lambda}^j_{\sigma}-\lambda^j_{\sigma}),$$
 $$
 \vcr^j(\vU^j_{\sigma}(\xi))=\vcr^j_{\sigma}+\xi\frac{\mathrm{d}\vcr^j(\vU^j_{\sigma}(0))}{\mathrm{d}\xi}+\frac{\xi^2}{2}\frac{\mathrm{d}^2\vcr^j(\vU^j_{\sigma}(0))}{\mathrm{d}\xi^2}+(\bar{\vcr}^j_{\sigma}-\vcr^j_{\sigma}),$$
noticing that 
\[ 
|\frac{\mathrm{d}^s\lambda^j(\vU^j_{\sigma}(\xi))}{\mathrm{d}\xi^s}|=O(|\jump {\vU_j}|)^s, \quad |\frac{\mathrm{d}^s \vcr^j(\vU^j_{\sigma}(\xi))}{\mathrm{d}\xi^s}|=O(|\jump {\vU_j}|)^s,
\]
we shall formulate $I_2$ as follows
 \begin{align}
I_2&=\int^{\frac{1}{2}}_{-\frac{1}{2}}4\xi\braket{\bm{A}(\vU^j_{\sigma}(\xi))\vcr^j(\vU^j_{\sigma}(\xi)),\vcr^j_{\sigma}-\vcr^j(\vU^j_{\sigma}(\xi))}\dxi =\int^{\frac{1}{2}}_{-\frac{1}{2}}4\xi\lambda^j(\vU^j_{\sigma}(\xi))\braket{\vcr^j(\vU^j_{\sigma}(\xi)),r^j_{\sigma}-\vcr^j(\vU^j_{\sigma}(\xi))}\dxi
\br
&=\int^{\frac{1}{2}}_{-\frac{1}{2}}-4\xi \left( \lambda^j_{\sigma}+\xi\frac{\mathrm{d}\lambda^j(\vU^j_{\sigma}(0))}{\mathrm{d}\xi}+\frac{\xi^2}{2}\frac{\mathrm{d}^2\lambda^j(\vU^j_{\sigma}(0))}{\mathrm{d}\xi^2}+(\bar{\lambda}^j_{\sigma}-\lambda^j_{\sigma}) \right)\cdot
\br
 &\quad\quad  \left( \xi\braket{\vcr^j_{\sigma},\frac{\mathrm{d}\vcr^j(\vU^j_{\sigma}(0))}{\mathrm{d}\xi}}+\frac{\xi^2}{2}\braket{\vcr^j_{\sigma},\frac{\mathrm{d}^2\vcr^j(\vU^j_{\sigma}(0))}{\mathrm{d}\xi^2}}+\braket{\vcr^j_{\sigma},(\bar{\vcr}^j_{\sigma}-\vcr^j_{\sigma})}+\xi^2\braket{\frac{\mathrm{d}\vcr^j(\vU^j_{\sigma}(0))}{\mathrm{d}\xi},\frac{\mathrm{d}\vcr^j(\vU^j_{\sigma}(0))}{\mathrm{d}\xi}} \right) \dxi
 \br
 &=\int^{\frac{1}{2}}_{-\frac{1}{2}} \Bigg( -4\xi^2 \lambda^j_{\sigma}\braket{\vcr^j_{\sigma},\frac{\mathrm{d}\vcr^j(\vU^j_{\sigma}(0))}{\mathrm{d}\xi}}-2\xi^3\lambda^j_{\sigma}\braket{\vcr^j_{\sigma},\frac{\mathrm{d}^2\vcr^j(\vU^j_{\sigma}(0))}{\mathrm{d}\xi^2}}-4\xi\lambda^j_{\sigma}\braket{\vcr^j_{\sigma},\bar{\vcr}^j_{\sigma}-\vcr^j_{\sigma})}
 \br
 &\hspace{1.5cm}  -4\xi^3\lambda^j_{\sigma}\braket{\frac{\mathrm{d}\vcr^j(\vU^j_{\sigma}(0))}{\mathrm{d}\xi},\frac{\mathrm{d}\vcr^j(\vU^j_{\sigma}(0))}{\mathrm{d}\xi}}-4\xi^3\frac{\mathrm{d}\vcr^j(\vU^j_{\sigma}(0))}{\mathrm{d}\xi}\braket{\vcr^j_{\sigma},\frac{\mathrm{d}\vcr^j(\vU^j_{\sigma}(0))}{\mathrm{d}\xi}} \Bigg) \dxi
 \br
 &=\int^{\frac{1}{2}}_{-\frac{1}{2}}\left( -4\xi^2 \lambda^j_{\sigma}\braket{\vcr^j_{\sigma},\frac{\mathrm{d}\vcr^j(\vU^j_{\sigma}(0))}{\mathrm{d}\xi}}\right)\dxi 
 = \frac13 \lambda^j_{\sigma} \braket{\vcr^j_{\sigma},\frac{\mathrm{d}\vcr^j(\vU^j_{\sigma}(0))}{\mathrm{d}\xi}} \br
 & \leqslant |\lambda^j_{\sigma}| + \frac1{36} \abs{\lambda^j_{\sigma}} \abs{\braket{\vcr^j_{\sigma},\frac{\mathrm{d}\vcr^j(\vU^j_{\sigma}(0))}{\mathrm{d}\xi}} }^2  
 = \abs{\lambda^j_{\sigma}} + \frac1{36} \abs{\lambda^j_{\sigma} } \abs{\braket{\vcr^j_{\sigma},\frac{\mathrm{d}\vcr^j(\vU^j_{\sigma}(0))}{\mathrm{d}\vU}\jump {\vU_j}  } }^2  
 \br
 &\leqslant |\lambda^j_{\sigma}|+\frac{(c_1{})^2c_2}{36}|\jump {\vU_j} |^2.
 \end{align}
We point out that we always omit the third order terms $O(|\jump {\vU_j}|^3)$ and higher in the above calculations. 
 
Finally, we control $I_3$ with
 \begin{align}
I_3&= \int^{\frac{1}{2}}_{-\frac{1}{2}}2\xi\braket{\bm{A}(\vU^j_{\sigma}(\xi))(\vcr^j_{\sigma}-\vcr^j(\vU^j_{\sigma}(\xi)),\vcr^j_{\sigma}-\vcr^j(\vU^j_{\sigma}(\xi))}\dxi
\leqslant\int^{\frac{1}{2}}_{-\frac{1}{2}}2\xi\rho(\bm{A}(\vU^j_{\sigma}(\xi)))|\vcr^j_{\sigma}-\vcr^j(\vU^j_{\sigma}(\xi))|^2 \dxi 
\br
&\leqslant c_4\int^{\frac{1}{2}}_{-\frac{1}{2}}2 \abs{\xi}|\vcr^j_{\sigma}-\vcr^j(\vU^j_{\sigma}(\xi))|^2\dxi \leqslant c_2c_1^2  |\jump {\vU_j} |^2 \int^{\frac{1}{2}}_{-\frac{1}{2}} 2 \abs{\xi}^3 \dxi = \frac1{16} c_2c_1^2  |\jump {\vU_j} |^2. 
 \end{align}
Combing above three estimates on $I_i,i=1,2,3$, we finish the proof. 
\end{proof}

\end{document}